%% file: ellipt_iso_revision.tex
\documentclass[a4paper,11pt]{article}

\usepackage{BeaCed_article_1}

\title{Elliptic dimers on minimal graphs and genus 1 Harnack curves}
\author{%
  C\'edric Boutillier%
  \thanks{%
    {\small Sorbonne Université, CNRS,
      Laboratoire de Probabilités Statistique et Modélisation, LPSM, UMR 8001,
      F-75005 Paris, France; Institut Universitaire de France.
      \texttt{cedric.boutillier@sorbonne-universite.fr}
}},
  David Cimasoni%
  \thanks{%
    {\small Université de Genève, Section de Mathématiques, 1211 Genève 4, Suisse.
     \texttt{david.cimasoni@unige.ch}
    }},
  B\'eatrice de Tili\`ere%
  \thanks{{\small %
PSL University-Dauphine, CNRS, UMR 7534, CEREMADE, 75016 Paris, France; Institut Universitaire de France.}{\small\texttt{ detiliere@ceremade.dauphine.fr}}
}
}

\begin{document}
\maketitle

\begin{abstract}
This paper provides a comprehensive study of the dimer model on infinite minimal graphs
with Fock's elliptic weights~\cite{Fock}. Specific instances of such models were
studied in~\cite{BdTR1,BdTR2,dT_Mass_Dirac}; we now handle the general genus 1
case, thus proving a non-trivial extension of the genus 0 results
of~\cite{Kenyon:crit,KO:Harnack} on isoradial critical models. We give an
explicit local expression for a two-parameter family of inverses of the
Kasteleyn operator with no periodicity assumption on the underlying graph.
  When the minimal graph satisfies a natural condition,
  we construct a family of dimer Gibbs measures from these inverses,
  and describe the phase diagram of the model by deriving asymptotics of 
  correlations in each phase.
In the $\ZZ^2$-periodic case,
this gives an alternative description of the full set of
  ergodic Gibbs measures constructed in~\cite{KOS}.
We also 
establish a correspondence between elliptic dimer models on
periodic minimal graphs and Harnack curves of genus 1.
Finally, we show that a bipartite dimer model is invariant
under the shrinking/expanding of~$2$-valent vertices and spider moves if and
only if the associated Kasteleyn coefficients are antisymmetric and satisfy
Fay's trisecant identity.
\end{abstract}

\section{Introduction}

This paper gives a full description of the bipartite dimer model on infinite, minimal graphs, with Fock's elliptic weights~\cite{Fock}.
In many instances, this finishes the study
initiated in~\cite{BdTR1} and~\cite{BdTR2} of models of statistical mechanics
related to dimers on infinite isoradial graphs
with local elliptic weights. Indeed, the massive Laplacian operator on a planar graph~$\Gs$
of~\cite{BdTR1} is related to the massive Dirac
operator~\cite{dT_Mass_Dirac}, which corresponds to an elliptic dimer model on
the bipartite double graph $\GD$, while the $Z$-invariant Ising model on~$\Gs$
of~\cite{BdTR2} is studied through an elliptic dimer model on the bipartite
graph~$\GQ$ (see Section~\ref{sec:previous_genus1}).
The papers~\cite{BdTR1,BdTR2} solve two specific instances of
elliptic bipartite dimer models; we now solve the general case. 

Let us be more precise.
Let $\Gs$ be an infinite \emph{minimal graph}~\cite{Thurston,GK},
meaning that it is planar, bipartite, and that its oriented train-tracks do not self intersect and do not form parallel bigons.
As proved in~\cite{BCdT:immersion}, a graph is minimal if and only if it
admits
a \emph{minimal immersion},
a concept generalising
that of isoradial embedding \cite{Kenyon:crit,KeSchlenk}.
Moreover, the space of such immersions can be described as an explicit subset of the space of half-angles maps associated to oriented train-tracks of $\Gs$,
see Section~\ref{sec:tt-def} below.
Minimal graphs with such half-angle maps
give the correct framework to study these models, see~\cite[Section~4.3]{BCdT:immersion}.
We consider \emph{Fock's elliptic Kasteleyn operator}~\cite{Fock}~$\Ks^{(t)}$ whose
non-zero coefficients correspond to edges of $\Gs$;
for an edge $\ws\bs$ of $\Gs$, the coefficient $\Ks^{(t)}_{\ws,\bs}$ is
explicitly given by
\begin{equation*}
\Ks^{(t)}_{\ws,\bs}=\frac{\theta(\beta-\alpha)}{\theta(t+\mapd(\bs)-\beta)\theta(t+\mapd(\bs)-\alpha)}, 
\end{equation*}
where $\theta(z)=\theta_1(z;q)=\theta_1(z|\tau)$ is Jacobi's first theta function,
$q=e^{i\pi\tau}$, $\tau$ {belongs to $i\RR_{>0}$} and $t\in\RR+\frac{\pi}{2}\tau$;
$\alpha,\beta$ are the half-angles assigned to the two train-tracks crossing the
edge $\ws\bs$, see Figure~\ref{fig:around_rhombus};
$\mapd$ is Fock's discrete Abel map, see Sections~\ref{sec:abel_map} and~\ref{sec:Kast_def}. 

Fock~\cite{Fock} actually introduces such an adjacency operator for all~$\ZZ^2$-periodic bipartite graphs,
for all parameters~$\tau$ and~$t$, and for theta functions of arbitrary genus;
in the present paper, we restrict ourselves to the genus 1 case, hence the name
\emph{Fock's elliptic operator},
and drop the periodicity assumption
(apart from the specifically dedicated Section~\ref{sec:periodic}).
Fock does not address the question of this operator being \emph{Kasteleyn}, \emph{i.e.},
corresponding to a dimer model with \emph{positive} edge weights.
Our first result, Proposition~\ref{prop:kastorient}, proves that this is indeed the
case when the graph is minimal, when the half-angles are chosen so as to define a
minimal immersion of~$\Gs$, and when the parameters $\tau,t$ are tuned as above.   

We now fix the parameter $t\in\RR+\frac{\pi}{2}\tau$ and omit it from the notation. 
One of our main results is an explicit \emph{local} expression for a two parameter family of
inverses~$(\A^{u_0})_{u_0\in D}$ of the elliptic Kasteleyn operator $\Ks$, where
$D:= (\RR/\pi\ZZ+[0,\frac{\pi}{2}\tau])\setminus \{\alpha_T\ ;\ T\in\T\}$ is
pictured in Figure~\ref{fig:domaineD_intro}, see also Figure~\ref{fig:domaineD},
and $(\alpha_T)_{T\in\T}$ are the half-angles assigned to {the elements of $\T$, consisting of the oriented} train-tracks of $\Gs$, {see Section~\ref{sec:tt-def}}.  
\begin{figure}[ht]
\centering
\def\svgwidth{5cm}
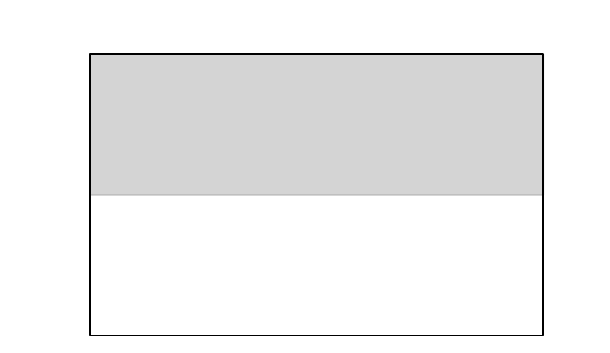
\caption{The domain $D$ as a shaded area of the torus $\TT(q):=\CC/\Lambda$,
  where $\Lambda$ is the lattice generated by $\pi$ and $\pi\tau$. The different
  cases corresponding to the possible locations of the parameter $u_0$. 
}
\label{fig:domaineD_intro}
\end{figure}
This result, which has remarkable probabilistic consequences, 
can be stated as follows, see also Definition~\ref{def:operator_A},
Lemma~\ref{lem:Kinv_alternative} and Theorem~\ref{thm:K_inverse_family}. 
 
\begin{thm}\label{thm:intro_main_1}
For every $u_0\in D$, the operator $\A^{u_0}$ whose coefficients are given,
for every black vertex $\bs$ and every white vertex $\ws$,
by
\[
\A^{u_0}_{\bs,\ws}=\frac{i\theta'(0)}{2\pi}\int_{\Cs_{\bs,\ws}^{u_0}}g_{\bs,\ws}(u)du,
\]
is an inverse of the elliptic Kasteleyn operator $\Ks$. The function $g$ is
defined in Section~\ref{sec:null_functions},
and the contour of integration~$\Cs_{\bs,\ws}^{u_0}$
in Section~\ref{sec:preliminaries_contours}, see also Figure~\ref{fig:domains_of_integration}. 
\end{thm}

\begin{rem}\label{rem:thm_intro_inverse}\leavevmode
\begin{enumerate}
 \item Locality of $\A^{u_0}_{\bs,\ws}$ stems from that of the function $g$
   which is defined as a product of terms associated to edges of a path from
   $\bs$ to $\ws$ in the
   associated quad-graph~$\GR$. 
   A key fact used in proving this theorem is that the
   functions $g_{\bs,\, \cdot \,}(u)$ and $g_{\, \cdot \,,\ws}(u)$ are in the
   kernel of the operator $\Ks$~\cite{Fock}, see also Proposition~\ref{prop:ker};
   this results from \emph{Fay's trisecant identity}, known as \emph{Weierstrass
   identity} in the genus 1 case, see Corollaries~\ref{cor:Faystrisecant}
   and~\ref{cor:FayFock}. 
 \item This theorem has the same flavour as the results
   of~\cite{Kenyon:crit,BeaCed:isogen} in the genus 0 case, and as those
   of~\cite{BdTR1,BdTR2} in the genus 1 case. It can be understood as the genus
   1 pendent (with an additional remarkable feature specified in the next point
   of this remark) of the dimer results of~\cite{Kenyon:crit},
   while~\cite{BdTR1} is the genus~1 version of the Laplacian results
   of~\cite{Kenyon:crit}, and~\cite{BdTR2} handles a specific elliptic dimer
   model arising from the Ising model. Going from genus~0 to genus~1 is a highly
   non-trivial step; indeed there is no straightforward way to identify the
   dimer weights and the function $g$ in the kernel of the Kasteleyn operator.
   It is much easier to recover the genus~0 results from the genus~1 ones by
   taking an appropriate limit for the elliptic parameter $\tau$; this is the
   subject of Section~\ref{sec:rational}. It is also not immediate how to
   recover the specific elliptic dimer models of~\cite{BdTR2,dT_Mass_Dirac} from
   this more general elliptic dimer model; this is explained in
   Section~\ref{sec:previous_genus1}.
 \item Another remarkable feature of Theorem~\ref{thm:intro_main_1} is that it
   provides a local expression for a \emph{two-parameter} family of inverses
   while in the references~\cite{Kenyon:crit,BeaCed:isogen,BdTR1,BdTR2}, 
   a \emph{single} inverse was considered.
   This allows us to prove a local formula for a
   two-parameter family of Gibbs measures, and not only for the maximal entropy
   Gibbs measure as was the case in the other references, see
   Theorem~\ref{thm:intro_main2} below.
 \item The explicit expression of Theorem~\ref{thm:intro_main_1} is very useful
   to perform asymptotic expansions
   of~$\A^{u_0}_{\bs,\ws}$ when the graph distance between~$\bs$ and~$\ws$ gets large;
   this is the subject of Propositions~\ref{prop:asympt_gas} and~\ref{prop:asymp_liq1} of Section~\ref{sec:asymptotics}.
   There are three different regimes depending on the position of the parameter~$u_0$ in $D$,
   pictured as Cases~$1, 2, 3$ in Figure~\ref{fig:domaineD_intro}. In Case 1
   (resp.\@ Case 2), one has exponential decay
   (resp.\@ linear decay up to gauge transformation) of~$\A^{u_0}_{\bs,\ws}$,
   while in Case~3, the state of edges is
   deterministic.
   These results allow us to derive the phase diagram of the corresponding dimer model,
   see Theorem~\ref{thm:intro_main2} below.
 \item
The local function~$g$
of Theorem~\ref{thm:intro_main_1}, in the kernel of the Kasteleyn operator,
gives an explicit {formula for}
of a {realization} of the dual graph $\Gs^*$~\cite{KLRR,CLR},
as described in Section~\ref{sec:circlepatterns}.
\end{enumerate}
\end{rem}

We now assume that the minimal graph $\Gs$ is $\ZZ^2$-periodic. A notable fact
is that periodicity of the graph and of half-angles associated to train-tracks
are not enough to ensure periodicity of the elliptic Kasteleyn operator $\Ks$.
In Proposition~\ref{prop:angles_perio}, we prove a necessary and sufficient
condition for that to be the case; intuitively it amounts to picking one of the
integer points of the geometric Newton polygon, see~\cite{GK} and
Section~\ref{sec:tt-per} for definitions.

To a $\ZZ^2$-periodic bipartite dimer model is naturally associated
a \emph{spectral curve} $\C$, and its \emph{amoeba} $\mathscr{A}$,
see Section~\ref{sec:charact_polynom} for definitions.
Our first result on this subject is Proposition~\ref{prop:param_curve}
proving an explicit birational parameterization of the spectral curve $\C$
by the torus $\TT(q)$
using the function $g$ of Theorem~\ref{thm:intro_main_1}:
we describe how the domain of definition of the function $g$,
\emph{i.e.}, the torus $\TT(q)$, is mapped to the spectral curve $\C$, thus establishing
that it is a Harnack curve of geometric genus~1.
As a byproduct we know how the domain $D$ of Figure~\ref{fig:domaineD_intro} is mapped
to the amoeba $\mathscr{A}$; this plays an important role in understanding the phase diagram
of the dimer model, also in the non-periodic case, see Point 3 of
Remark~\ref{rem:intro_Gibbs} below.
Our main result on this topic is that the converse also holds, 
see Theorem~\ref{thm:Harnack} for a precise statement. 

\begin{thm}\label{thm:intro_main3}
Every genus 1 Harnack curve with a marked point on the oval is the spectral
curve of an explicit dimer model on a minimal graph $\Gs$ with Fock's elliptic Kasteleyn operator,
for a unique parameter $t\in\RR/\pi\ZZ+\frac{\pi}{2}\tau$, and a half-angle map defining a minimal immersion of~$\Gs$. 
\end{thm}

Let us describe the context of this theorem.
By~\cite{KOS,KO:Harnack,GK} we know that bipartite dimer models are in correspondence with Harnack curves.
This correspondence is made explicit in~\cite{KO:Harnack} in the case of  generic genus~0 Harnack curves and dimer models on isoradial graphs with Kenyon's critical weights~\cite{Kenyon:crit}.
A result of the same flavor is obtained in~\cite{BdTR1} where an explicit correspondence is established between genus~1 Harnack curves with central symmetry and rooted spanning forests with well chosen elliptic weights.
Also on this topic, Fock~\cite{Fock} assigns an explicit ``dimer model'' to every algebraic curve;
his construction is very general but does not focus on curves being Harnack and ``dimer models''
having \emph{positive} weights.
Theorem~\ref{thm:intro_main3} is thus the pendent of~\cite{KO:Harnack}
in the genus 1 case
with general (possibly non triangular) Newton polygons;
it extends the result of~\cite{BdTR1} by removing the symmetry assumption on the curve.
Our proof uses~\cite{GK}, see also~\cite{Gulotta}, for reconstructing a minimal
graph from
the Newton polygon of
the spectral curve.

In~\cite{Sheffield}, the author proves that $\ZZ^2$-periodic bipartite dimer models have a two-parameter family of ergodic Gibbs measures, then~\cite{KOS} provide an explicit expression for these measures using Fourier transforms and magnetic field coordinates.
They also identify the phase diagram as the amoeba $\mathscr{A}$ of the spectral curve $\C$.
In this article, we reverse this point of view by considering \emph{a priori}
  a compact Riemann surface (the torus $\TT(q)$) which, together with appropriate 
  half-angle maps, induce dimer models on minimally immersed graphs.
For any such dimer model, we then
  construct a two-parameter family of Gibbs measures $(\PP^{u_0})_{u_0\in D}$
  from the inverses $(\As^{u_0})_{u_0\in D}$, indexed by the
  domain~$D$, which plays the role of the phase diagram. What is noteworthy
  is that, assuming Condition $(*)$ below, this also holds for \emph{non-periodic} graphs
  even though the spectral curve and the amoeba do not exist.
 Here is our main statement, which is a combination of
  Theorem~\ref{prop:Gibbs_measures}, Corollary~\ref{thm:Gibbs_measures} and
  Theorem~\ref{thm:gibbs_non_perio}. It holds for any minimal graph~$\Gs$ satisfying
  the following assumption, which is trivially true for $\ZZ^2$-periodic graphs and is
  believed to hold for all minimal graphs:

$(*)$ Every finite, simply connected subgraph~$\Gs_0$ of
the minimal graph $\Gs$ can be embedded in a \emph{periodic} minimal
graph~$\Gs'$ so that parallel train-tracks in~$\Gs_0$ remain parallel in~$\Gs'$.

\begin{thm}\label{thm:intro_main2}
  Consider the dimer model with Fock's elliptic weights on an infinite, minimal graph~$\Gs$ satisfying Condition~$(*)$.
  Then, for every $u_0\in D$, the operator~$\A^{u_0}$ of Theorem~\ref{thm:intro_main_1}
  defines a Gibbs measure $\PP^{u_0}$
  on dimer configurations of $\Gs$, whose expression on cylinder sets is explicitly given by,
for every subset of
distinct edges $\{\es_1=\ws_1 \bs_1,\ldots,\es_k=\ws_k \bs_k\}$ of $\Gs$,
\begin{align}
\PP^{u_0}(\es_1,\ldots,\es_k)&=
\Bigl(\prod_{j=1}^k
\Ks_{\ws_j,\bs_j}\Bigr)\times 
\det_{1\leq i,j\leq k} \Bigl(\A^{u_0}_{\bs_i,\ws_j}\Bigr)\,.
\end{align}

The set $D$ gives the phase diagram of the model: when $u_0$ is on
the top boundary of $D$, the dimer model
is gaseous; when $u_0$ is in the interior of the set $D$, the model is liquid;
when $u_0$ is a
point corresponding to one of the connected components of the lower boundary of
$D$, the model is solid.

When $\Gs$ is $\ZZ^2$-periodic,
this gives an alternative description of the full set of
ergodic Gibbs measures~\cite{KOS}.
\end{thm}

\begin{rem}\label{rem:intro_Gibbs}
\leavevmode 
\begin{enumerate}
  \item
      One of the main features of these Gibbs measures $\PP^{u_0}$ is their
      \emph{locality property}, inherited from $\A^{u_0}$: the correlations
      between edges $\es_1,\ldots,\es_k$ only depend on the geometry of the
      graph in a ball containing those edges. This locality is a key ingredient
      to extend the proof of Theorem~\ref{thm:intro_main2} from periodic to
      general graphs satisfying Condition~$(*)$, by a now standard
      argument~\cite{BeaQuad,BeaCed:isogen,BdTR1}.
 \item
     When $\Gs$ is $\ZZ^2$-periodic, the correspondence between the set of
     ergodic Gibbs measures from~\cite{KOS} and the family $(\PP^{u_0})_{u_0\in D}$
     is proved by showing that the Fourier
     transform expressions of the inverses~\cite{CKP,KOS}
     as double integrals
     coincide with the
     inverses of Theorem~\ref{thm:intro_main_1}. The fundamental step is
     the explicit evaluation of one of the integrals via
     residues, and
     a change
     of variable in the remaining integral
     which uses the explicit parameterization of the spectral curve
     $\C$ together with the key Lemma~\ref{lem:forme_holom} establishing that the
     combination of the denominator in the integrand and of the
     Jacobian is in fact trivial.
     Note that the locality property of these ergodic Gibbs measures in the
     periodic case was not known before,
     except when the spectral
     curve had genus~0~\cite{Kenyon:crit,KO:Harnack}.
 \item As shown in~\cite{KOS}, ergodic Gibbs measures
   for $\ZZ^2$-periodic graphs
   can alternatively be parameterized
   by their \emph{slope}, \emph{i.e.}, by their expected horizontal and vertical
   height change. In Theorem~\ref{thm:slope}, we prove an explicit expression
   for the slope of the Gibbs measure $\PP^{u_0}$ involving the explicit
   parameterization of the spectral curve $\C$ and appropriate contours of
   integration. This is a refined version of Theorem 5.6. of~\cite{KOS}, where
   the slope was only identified up to a sign and modulo~$\pi$.
 \item Note also that such an explicit formula
   lends itself well to explicit computations using the residue theorem.
   As an example, single-edge probabilities are computed in the three different
   phases in Proposition~\ref{lem:single_edges}.
 \end{enumerate}
\end{rem}

From the point of view of statistical mechanics, local formulas for
probabilities are expected to exist for models that are invariant under
elementary transformations of the graph~$\Gs$. For example, in the case of the
Ising model, the latter are the well known \emph{star-triangle} transformations;
in the case of the dimer model, they are the \emph{spider move} and the
\emph{shrinking/expanding of a 2-valent vertex}~\cite{Kuperberg,Thurston,Postnikov,GK}.
In Proposition~\ref{prop:2-valent} and Theorem~\ref{thm:spider} of
Section~\ref{sec:spider}, we prove that invariance 
under these two moves is \emph{equivalent} to the Kasteleyn coefficients being
antisymmetric (as functions of the train-track half-angles) and satisfying Fay's identity in the form of
Corollary~\ref{cor:FayFock}.
As a consequence, we recover that this holds for the elliptic dimer model
considered in this paper, a fact already known to Fock~\cite{Fock}.

As a final remark to this introduction, let us mention our forthcoming
paper~\cite{BCdT:genusg}, where we handle Fock's adjacency operator and its
consequences for the dimer model in the case of arbitrary
positive genus.
It will in particular
take care of the additional difficulties related to more involved algebraic and complex geometry.

\subsection*{Outline of the paper}
\begin{itemize}
 \item[$\bullet$] In Section~\ref{sec:general}, we recall concepts and results needed for our work:
train-tracks, half-angle maps, minimal graphs, minimal immersions~\cite{BCdT:immersion}, basics on the dimer model, 
Fock's definition of the \emph{discrete Abel map}~\cite{Fock} and Jacobi theta functions.
\item[$\bullet$] In Section~\ref{sec:Kasteleyn}, we define Fock's elliptic adjacency 
operators~$\K[t]$~\cite{Fock} and determine under which conditions it is Kasteleyn. We introduce a family of functions in the kernel of~$\K[t]$ and study the relative positions of their poles and zeros.
Then, building on ideas of~\cite{KLRR} we show that these functions define explicit immersions of the dual graph. 
 \item[$\bullet$] In Section~\ref{sec:inv}, we fix $t\in\RR+\frac{\pi\tau}{2}$ and
introduce a family of \emph{local}
operators~$(\A^{(t),u_0})_{u_0\in D}$ parameterized by a subset~$D$ of the cylinder~$\mathbb{R}/\pi\ZZ + [0,\frac{\pi}{2}\tau]$. We then state and prove
Theorem~\ref{thm:intro_main_1}.

\item[$\bullet$] Section~\ref{sec:periodic} deals with the case of~$\ZZ^2$-periodic 
minimal graphs. We determine for which half-angle maps
the corresponding elliptic Kasteleyn operator itself is~$\ZZ^2$-periodic.
We use the functions of Section~\ref{sec:null_functions} to give an explicit parameterization of the spectral curve of the model. We then state and prove Theorem~\ref{thm:intro_main3}
and the periodic version of Theorem~\ref{thm:intro_main2}. Finally, we derive an
explicit expression for slopes of the Gibbs measures.
 \item[$\bullet$] In Section~\ref{sec:gibbs_non_perio}, we drop the periodicity assumption on the minimal graph. We then prove Theorem~\ref{thm:intro_main2} defining a two parameter family of
Gibbs measures with three phases as in the periodic case. We compute single-edge probabilities and asymptotics of the inverse
operators~$(\A^{(t),u_0})_{u_0\in D}$ in these three phases.
 \item[$\bullet$] In Section~\ref{sec:spider} we show that 
invariance of the dimer model under some natural
elementary transformations on bipartite graphs, is equivalent to the corresponding Kasteleyn coefficients being antisymmetric
and satisfying Fay's identity; in particular this holds for 
the dimer model with Fock's elliptic weights.
 \item[$\bullet$] Finally, in Section~\ref{sec:connection}, we present relations
 between the present work and previously studied models.
We first show how Kenyon's critical dimer
models~\cite{Kenyon:crit} can be obtained as rational limits of our elliptic models. Then, 
we explain how the models
of~\cite{BdTR2} and of~\cite{dT_Mass_Dirac} are special cases
of the constructions of this paper.
 \end{itemize}

\subsection*{Acknowledgments}
This project was started when the second-named author was visiting the first and third-named authors at the LPSM, Sorbonne Universit\'e, whose hospitality is thankfully acknowledged.  The first- and third-named authors are partially supported by the \emph{DIMERS} project~ANR-18-CE40-0033 funded by the French National Research Agency.
The second-named author is partially supported by the Swiss NSF  grant 200020-200400.
We would like to thank Vladimir Fock for helpful discussions
and inspiration, {and the anonymous referee for valuable comments.}

\subsection*{Declarations}
\begin{description}
\item[Data Availability.] Not applicable to this article since no data sets were generated or analysed
during the current study.
\item[Conflict of interest.] The authors have no relevant financial or non-financial interests to
disclose.
\end{description}

\section{Generalities}
\label{sec:general}

The aim of this first section is to introduce well-known concepts and results needed for the rest of the paper.
Section~\ref{sec:tt-def} deals with train-tracks associated to planar graphs,
minimal graphs, and a special class of half-angle maps associated with train-tracks of minimal graphs.
In Section~\ref{sec:dimer_model} , we briefly explain the basics of the dimer model on bipartite planar graphs.
Finally, in Section~\ref{sec:abel_map}, we recall the definition of Fock's discrete Abel map, and 
the definition and main features of Jacobi theta functions.

\subsection{Train-tracks, minimal graphs and monotone angle maps}
\label{sec:tt-def}

Consider a locally finite graph~$\Gs=(\Vs,\Es)$ embedded in the plane so that
its faces are bounded topological discs;
{in particular, the graph~$\Gs$ is infinite.}
Denote by~$\Gs^*=(\Vs^*,\Es^*)$ the dual embedded graph.
The associated \emph{quad-graph}~$\GR$ is obtained from the vertex set~$\Vs\sqcup\Vs^*$ by joining a primal vertex~$\vs\in\Vs$ and a dual vertex~$\fs\in\Vs^*$ each time~$\vs$ lies on the boundary of the face corresponding to~$\fs$. This quad-graph embeds in the plane with faces consisting of (possibly degenerate) quadrilaterals, whose
diagonals are pairs of dual edges of~$\Gs$ and~$\Gs^*$ (see Figure~\ref{fig:graph}).

A \emph{train-track} of~$\Gs$~\cite{Kenyon:crit,KeSchlenk} is a
maximal
chain of adjacent quadrilaterals of~$\GR$ such that when it enters a quadrilateral, it exits through the opposite edge. A train-track can also be thought of as a path in~$(\GR)^*$ crossing opposite edges of the quadrilaterals,
and we often make this slight abuse of terminology. Note that by construction, the graphs~$\Gs$ and~$\Gs^*$ have the same set of train-tracks.

\begin{figure}[tb]
    \centering
    \def\svgwidth{8cm}
   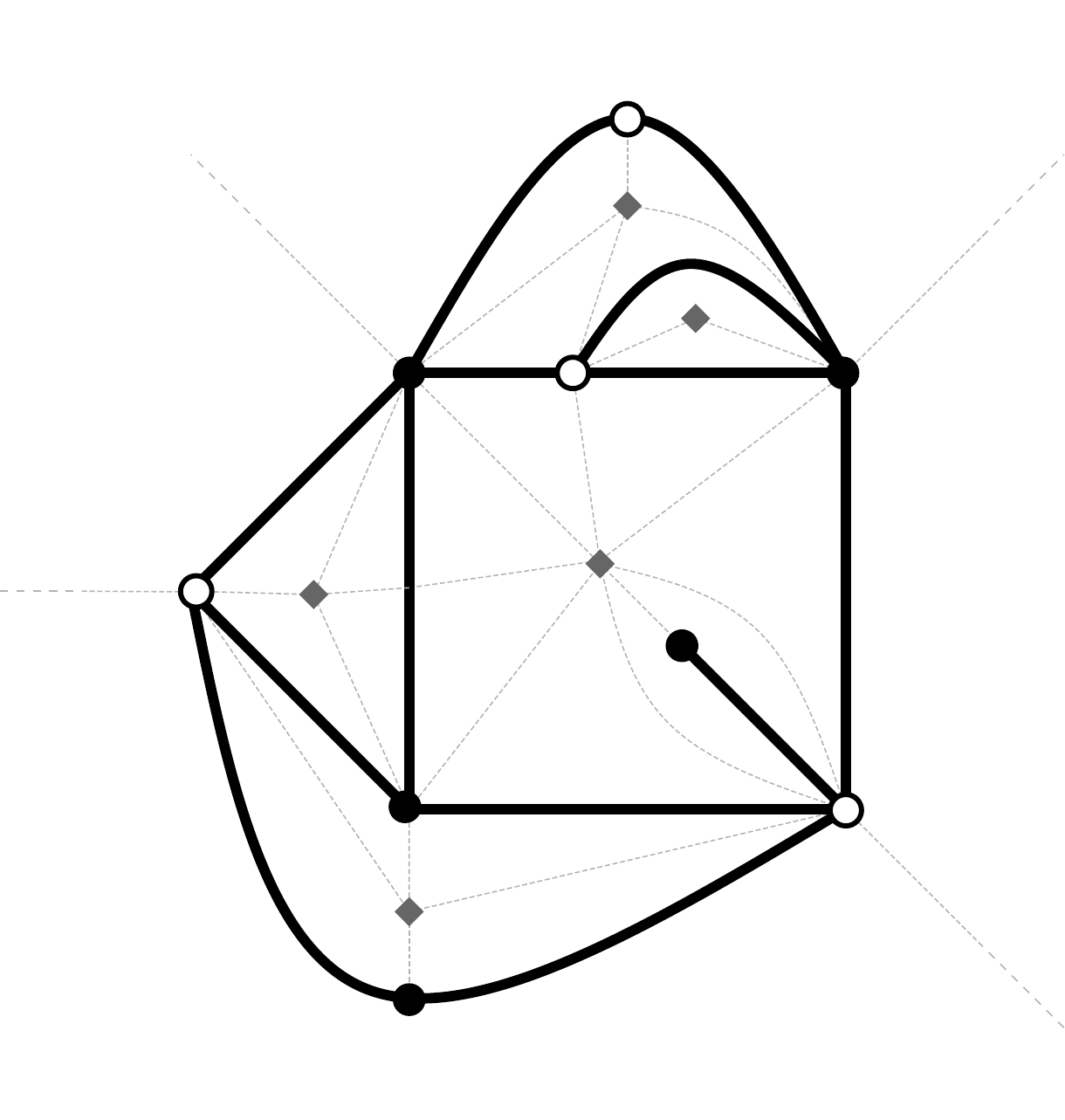
    \caption{A finite portion of a bipartite graph~$\Gs$, with black edges and black/white vertices.
    The associated dual graph~$\Gs^*$ has vertices marked with grey diamonds, while
    the associated quad-graph~$\GR$ is represented with {dashed} grey edges.
    The corresponding train-tracks form two self-intersections ($T_2$ and~$T_6$)
    and two parallel bigons (one between~$T_1$ and~$T_6$, the other between~$T_2$
    and~$T_5$). In particular,~$\Gs$ is not a minimal graph.}
    \label{fig:graph}
  \end{figure}

We now assume that~$\Gs$ is \emph{bipartite}, \emph{i.e.}, that
its vertex set can be partitioned into two sets~$\Vs=\Bs\sqcup\Ws$ of black and white vertices such that no edge of~$\Es$ connects two vertices of the same color.
In such a case, paths corresponding to train-tracks can be consistently oriented with, say, black vertices on the right and white vertices on the left of the path, as
illustrated in Figure~\ref{fig:graph}.
We let~$\T$ denote the set of consistently oriented train-tracks of the bipartite graph~$\Gs$.

For the definition of our model (see Section~\ref{sec:Kast_def} below), we need to assign a half-angle~$\alpha_T\in\RR/\pi\ZZ$ to each oriented train-track~$T$ of~$\T$.
Many of our results hold for arbitrary
half-angle maps~$\mapalpha\colon\T\to\RR/\pi\ZZ$ defined on arbitrary bipartite planar graphs.
However, several results only hold for a specific
class of such graphs, and for a restricted space of angle maps. We now define these classes
of graphs and maps.

Following~\cite{Thurston,GK}, we say that
a bipartite, planar graph~$\Gs$ is \emph{minimal}
if oriented train-tracks of~$\T$ do not self-intersect and if
there are no \emph{parallel bigons}, \emph{i.e.}, no pairs of paths intersecting twice and joining these two intersection points in the same direction;
we refer to Figure~\ref{fig:graph} for an example with such forbidden train-track
configurations. {Note that this implies that train-tracks cannot form loops. Indeed, if this were the case then, since faces of $\Gs$ are bounded topological disks, either the train-track would self intersect, which is forbidden, or it would cross another train-track at least twice, and thus form a parallel bigon, which is also forbidden.} {The minimality condition also }
implies that~$\Gs$ has neither multiple edges,
nor degree 1 vertices. In particular, a minimal graph
is a simple graph. 

More details on the next part of this section can be found in the paper~\cite{BCdT:immersion}.
The restriction on the half-angle maps can be motivated geometrically as follows.
Given a bipartite, planar graph~$\Gs$, a map~$\mapalpha\colon\T\to\RR/\pi\ZZ$
defines a {realization} of~$\GR$
in~$\RR^2$ by {drawing} every directed edge of~$\GR$ crossed by an oriented train-track~$T$ from left to right as the unit vector~$e^{2i\alpha_T}$, {see \cite[Section 3.1]{BCdT:immersion}} for details.
In this way, each face of~$\GR$ is mapped to a rhombus of unit edge length, with a rhombus angle in~$[0,2\pi)$ naturally defined from the value of~$\mapalpha$ on the two train-tracks crossing this face.
Adding up the corresponding rhombus angles {at the vertices of~$\GR$} we obtain angle {sums} that are,
in general, arbitrary {positive} integer multiples of~$2\pi$.
Following~\cite{BCdT:immersion},
we say that~$\mapalpha$ defines a \emph{minimal immersion}
of~$\Gs$ if the rhombus angles never vanish and add up to~$2\pi$ around each vertex of~$\GR$. {Intuitively, in a minimal immersion, around each vertex the rhombi do one turn, while in an arbitrary realization, they are allowed to do more than one turn.}
This notion is a natural generalization of the isoradial embeddings of Kenyon and Schlenker~\cite{KeSchlenk},
where rhombi with rhombus angle in~$(\pi,2\pi)$ are folded along
their dual edge {as illustrated in Figure~\ref{fig:folded_rhombus}.}

A map~$\mapalpha\colon\T\to\RR/\pi\ZZ$ defines a minimal immersion if it respects some natural cyclic order on~$\T$, whose definition we now recall, see also~\cite[Section 2.3]{BCdT:immersion}.
Let us assume that~$\Gs$ is a minimal graph. We say that two oriented train-tracks are \emph{parallel} (resp. \emph{anti-parallel}) if they are disjoint
and there exists a topological disc
that they cross in the same direction (resp.\@ in opposite directions).
Consider a triple of oriented train-tracks~$(T_1,T_2,T_3)$ of~$\Gs$, pairwise non-parallel. If two of these train-tracks intersect infinitely often, then they do so in opposite directions: replace this pair of train-tracks by anti-parallel disjoint oriented curves. We are now left with three
bi-infinite
oriented planar curves that intersect a finite number of times. Consider a compact disk~$B$ outside of which they do not meet, and order~$(T_1,T_2,T_3)$ cyclically according to the outgoing points of the corresponding oriented curves in the circle~$\partial B$. A choice was made when replacing anti-parallel train-tracks by disjoint curves, but the resulting (partial) cyclic order on~$\T$
is easily seen not  to depend on this choice. Note that when~$\Gs$ is~$\mathbb{Z}^2$-periodic, this cyclic
order is the same as the natural cyclic order defined via the homology classes of the projections of the train-tracks onto~$\Gs/\mathbb{Z}^2$, see Section~\ref{sec:periodic}.

Following~\cite{BCdT:immersion}, we denote by~$X_\Gs$ the set of half-angle maps~$\mapalpha\colon\T\to\RR/\pi\ZZ$ that are monotone with respect to the cyclic orders on~$\T$ and~$\RR/\pi\ZZ$, and that map pairs of intersecting or anti-parallel train-tracks to distinct angles. One of the main results of~\cite{BCdT:immersion} can now be stated as follows: a planar bipartite graph~$\Gs$ admits a minimal immersion if and only if~$\Gs$ is minimal; in such a case, the space of minimal immersions contains~$X_\Gs$, and coincides with it in the~$\mathbb{Z}^2$-periodic case.
{A piece of a minimal immersion is shown in Figure~\ref{fig:folded_rhombus}.}

\begin{figure}
  \centering
  \includegraphics[width=3cm]{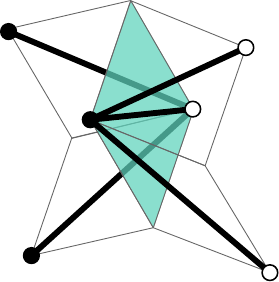}
  \caption{{A small piece of a minimal immersion of $\Gs$ where a rhombus
  (in light blue) has an angle in $(\pi,2\pi)$.}\label{fig:folded_rhombus}}
\end{figure}

\subsection{The dimer model}
\label{sec:dimer_model}

We here recall basic facts and definitions on the dimer model. More details can
be found for example in~\cite{KenyonIntro} and references therein.

In the whole of this section,~$\Gs$ is a planar, bipartite graph, finite or infinite.
A \emph{dimer configuration} of $\Gs$, also known as a \emph{perfect matching}, is
a collection~$\Ms$ of edges of $\Gs$ such that every vertex is incident to exactly
one edge of~$\Ms$; we denote by $\M(\Gs)$ the set of dimer configurations of the graph $\Gs$ and assume that this set is non-empty. 

Suppose that edges are assigned a positive weight function $\nu=(\nu_\es)_{\es\in\Es}$. When the graph is finite, the \emph{dimer Boltzmann measure} $\PP$ on $\M(\Gs)$ is defined by:
\[
\PP(\Ms)=\frac{\nu(\Ms)}{Z(\Gs,\nu)},
\]
where $\nu(\Ms)=\prod_{\es\in\Ms}\nu_\es$ is the weight of the dimer configuration $\Ms$, and $Z(\Gs,\nu)=\sum_{\Ms\in\M(\Gs)}\nu(\Ms)$ is the \emph{dimer partition function}.

When the graph $\Gs$ is infinite, the notion of Boltzmann measure is replaced by that of \emph{Gibbs measure}. By definition, a Gibbs measure $\PP$ needs to satisfy the \emph{DLR conditions}~\cite{Dobrushin,LanfordRuelle}: if one fixes a dimer configuration in an annular region of the graph~$\Gs$, dimer configurations inside and outside of the annulus are independent; moreover the
probability of any dimer configuration in the connected region inside the
annulus is proportional to the product of its edge weights.

Following~\cite{KOS}, two dimer models
given by two positive weight functions~$\nu$ and~$\nu'$ on~$\Gs$
are said to be \emph{gauge equivalent} if there exists a positive function~$\sigma$ on~$\Vs$ such that,
for each edge~$\ws\bs\in\Es$, we have~$\nu'_{\ws\bs}=\sigma_\ws\,\nu_{\ws\bs}\,\sigma_\bs$.
Suppose now that the graph~$\Gs$ is finite, then two gauge equivalent dimer models are easily seen
to yield the same Boltzmann measure.
Therefore, many of the edge weight parameters are non-essential as far as the associated
Boltzmann measure is concerned.
For a bipartite, planar, weighted graph~$(\Gs,\nu)$, a family of associated
essential parameters is given as follows.
The~\emph{face weight}~${\Wscr}_\fs$ of a degree~$2n$ face~$\fs$ is defined to be 
\begin{equation}
  {\Wscr}_\fs:=\prod_{j=1}^n
  \frac{\nu_{\ws_j\bs_j}}{\nu_{\ws_j\bs_{j-1}}},
  \label{eq:altproduct}
\end{equation}
where~$\ws_1,\bs_1,\ws_2,\ldots, \ws_n, \bs_n$ are the vertices on the boundary
of~$\fs$ oriented counterclockwise with cyclic notation for indices, see Figure~\ref{fig:faceweight}.
When $\Gs$ is planar, which is assumed to be the case here,
two dimer models on~$\Gs$ are gauge equivalent if and only if
the corresponding edge weights define equal face weights for all bounded faces.
Moreover, the associated Boltzmann measure can be described using these face weights.
This requires the concept of \emph{height function}, that we now recall.

Let us fix a reference dimer configuration~$\Ms_1$,
and take an arbitrary~$\Ms\in\M(\Gs)$.
Considering~$\Ms$ and~$\Ms_1$ as consistently oriented from white to black vertices,
their difference~$\Ms-\Ms_1$ is a {union of disjoint} oriented cycles in~$\Gs$. Since this graph is embedded in the plane,
{each} oriented cycle bounds a collection of faces. In other words,
we have~$\Ms-\Ms_1=\partial(\sum_{\fs\in\Fs} h_\Ms(\fs)\fs)$ for
some function~$h_\Ms\colon\Fs\to\ZZ$, uniquely defined up to a global additive constant.
This is called the \emph{height function} of~$\Ms$ (with respect to~$\Ms_1$).
As one easily checks, we then have
\[
\PP(\Ms)=\frac{{\Wscr}(h_\Ms)}{Z(\Gs,{\Wscr})}\,,
\]
where~${\Wscr}(h_\Ms)=\prod_{\fs\in\Fs} {\Wscr}_\fs^{h_\Ms(\fs)}$ and
\begin{equation}
\label{eq:Z-face}
Z(\Gs,{\Wscr})=\sum_{\Ms\in\M(\Gs)}{\Wscr}(h_\Ms)\,.
\end{equation}
In a nutshell, fixing a reference dimer configuration allows to reformulate the Boltzmann measure
on~$\M(\Gs)$ with (many non-essential) parameters~$(\nu_\es)_{\es\in\Es}$ as a measure
on the associated height functions with (only essential) parameters~$({\Wscr}_\fs)_{\fs\in\Fs}$.

One of the key tools for studying the dimer model is the \emph{Kasteleyn
matrix}~\cite{Kasteleyn1,TF,Percus}. Suppose that edges are oriented so that
around every bounded face of the graph $\Gs$,
there are an odd number of edges oriented clockwise. Define $\Ks$ to be the
corresponding oriented, weighted, adjacency matrix: rows of
$\Ks$ are indexed by white vertices, columns by black ones, non-zero
coefficients correspond to edges of $\Gs$, and when $\ws\sim\bs$,
$\Ks_{\ws,\bs}=\pm\nu_{\ws\bs}$, where the sign is~$+$, resp.~$-$, if the edge is
oriented from $\ws$ to $\bs$, resp.\@ from $\bs$ to $\ws$. When the graph $\Gs$ is
finite, the partition function of the dimer model is equal to $|\det
\Ks|$~\cite{Kasteleyn2,TF}. Kenyon~\cite{Kenyon1} derives an explicit expression
for the dimer Boltzmann measure $\PP$ in terms of~$\Ks$,
establishing that the dimer model is a determinantal process.

This was extended by Kuperberg~\cite{Kuperberg} as follows.
Consider a weighted adjacency matrix of $\Gs$ with possibly complex coefficients,
\emph{i.e.}, a matrix~$\Ks$ as above with~$\Ks_{\ws,\bs}=\omega_{\ws\bs}\nu_{\ws\bs}$,
this time allowing for~$\omega_{\ws\bs}$ to be any modulus~$1$ complex number
(as opposed to only~$\pm 1$ above).
Let us assume that for any bounded face~$\fs$ of~$\Gs$, the phase~$\omega$ satisfies the
following \emph{Kasteleyn condition}:
\[
\prod_{j=1}^n \frac{\omega_{\ws_j\bs_j}}{\omega_{\ws_j\bs_{j-1}}}=(-1)^{n+1}\,,
\]
assuming the notation of Figure~\ref{fig:faceweight}.
Then, the dimer partition function and Boltzmann measure can be computed
from~$\Ks$ and its inverse. When this is the case, $\Ks$ is said to be \emph{Kasteleyn};
we also refer to $\Ks$ as a \emph{Kasteleyn matrix}
for the dimer model on~$(\Gs,\nu)$.

The situation in the case of finite graphs embedded in the torus is different; the key facts are recalled when needed, that is at the beginning of Section~\ref{sec:Gibbs_measures}.

A Kasteleyn matrix~$\Ks$ can be seen as a linear operator from the complex valued functions on black vertices to those on white vertices of~$\Gs$, via the equality~$(\Ks f)_\ws=\sum_{\bs} \Ks_{\ws,\bs}f_b$ for~$f\in \CC^\Bs$. Therefore, we will refer to $\Ks$ as a Kasteleyn matrix or \emph{operator}, and similarly for weighted adjacency matrices/operators.

\begin{figure}[h]
    \centering
    \def\svgwidth{8cm}
    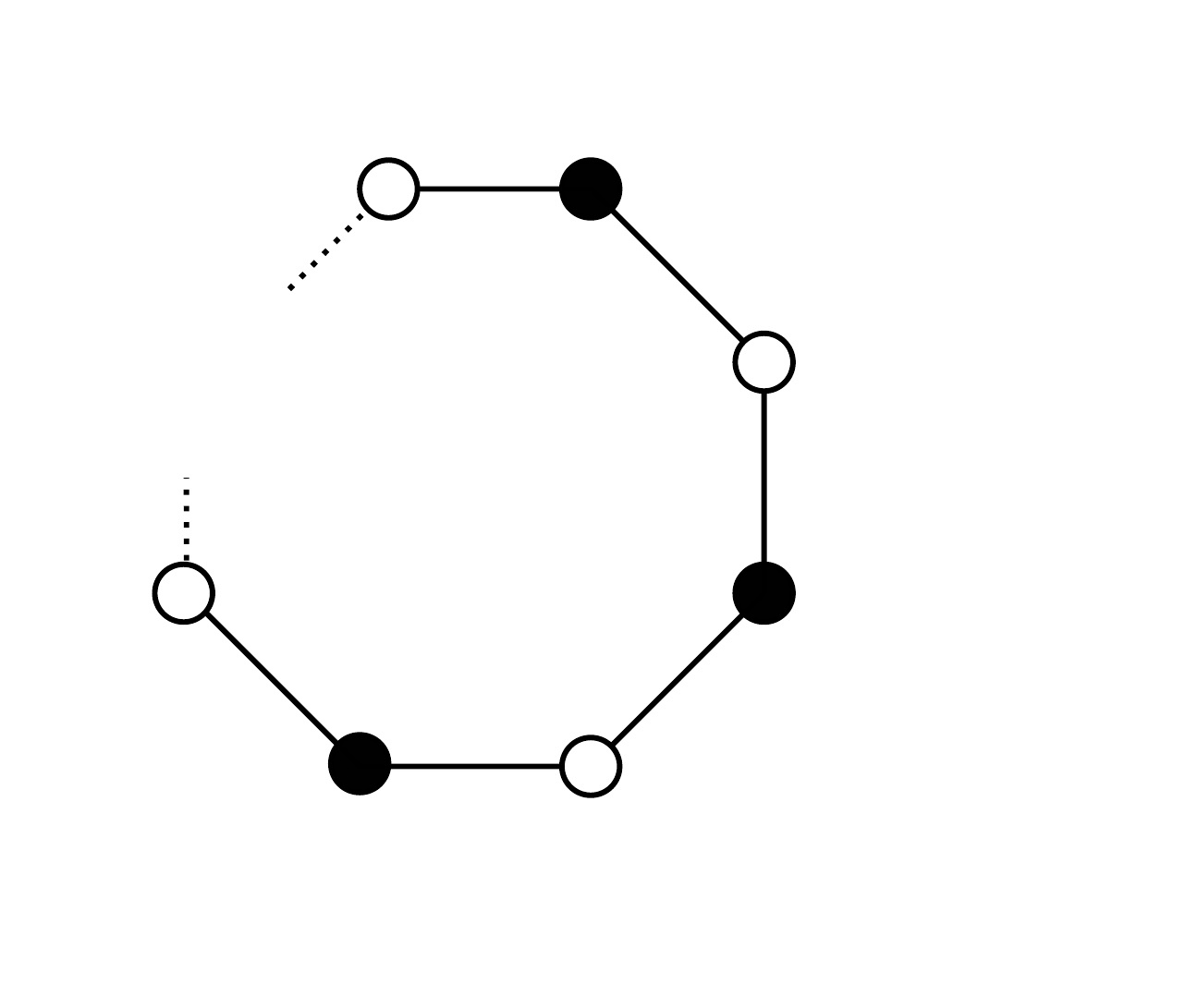
    \caption{Vertices and train-track half-angles around an arbitrary face.}
    \label{fig:faceweight}
  \end{figure}

\subsection{Discrete Abel map and Jacobi theta functions}\label{sec:abel_map}

In order to define Fock's elliptic adjacency operator in the next section, we need two preliminary definitions: the discrete Abel map, and elliptic theta functions.

\textbf{Discrete Abel map}. Following Fock~\cite{Fock}, we iteratively construct a function $\mapd$, denoted by \textbf{\textit{d}} in~\cite{Fock}, which assigns to every vertex of the quad-graph $\GR$ a linear combination of train-track half-angles $(\alpha_T)_{T\in\T}$ with integer coefficients.

The map $\mapd$ is constructed as follows.
Choose a vertex~$\vs_0$ of~$\GR$ and
set (arbitrarily) $\mapd(\vs_0)=0$.
Then along an edge of the
quad-graph $\GR$
crossed by a train-track $T$, the value of $\mapd$
increases, resp.\@ decreased by $\alpha_T$, if $T$ goes from right
to left, resp.\@ from left to right, when traversing the edge.
One easily checks that this gives a well defined map on
the vertices of~$\GR$.
This formal~$\ZZ$-linear combination of half-angles can be understood as an element of $\RR/\pi\ZZ$
by evaluating the combination modulo $\pi$.
An example of computation around a
face of $\GR$ is given in Figure~\ref{fig:around_rhombus} below.

\begin{figure}[ht]
  \centering
  \definecolor{electricblue}{RGB}{0,117,255}
  \begin{tikzpicture}[auto]
    \node  (w) at (-1.5,0) [shape=circle,draw,fill=white,label=left:$\ws$] {};
    \node (b) at (1.5,0) [shape=circle,draw,fill=black,label=right:$\bs$] {};
    \node (o) at (0,0) {};
    \coordinate [label=above:$\fs$](f)  at (0.3,1.5)  ;
    \coordinate [label=below:$\fs'$](g)  at (-0.2,-1)  ;
    \draw  (b) -- (f) -- (w) -- (g) --  (b);
    \draw[very thick] (w) -- (b);
    \coordinate [label=left:$\beta$] (d1) at (-1,1);
    \coordinate [label=right:$\alpha$] (d2) at (1,1);
    \path[<-<,draw=red,ultra thick]  (-1,1) to [bend left]   (o)  to [bend right] (1,-1) ;
    \path[<-<,draw=electricblue,ultra thick]   (1,1) to [bend right]   (o)  to [bend left] (-1,-1) ;
  \end{tikzpicture}
  \caption{%
    A face of $\GR$ has four vertices corresponding to a white vertex $\ws$, a black
    vertex $\bs$ and two faces $\fs$ and $\fs'$ of the minimal graph, and is at the intersection of two
    train-tracks. The train-tracks are oriented in such a way that they turn
    counterclockwise around $\ws$ and clockwise around $\bs$.
    We have
    $\mapd(\bs)=\mapd(\fs)+\alpha=\mapd(\fs')+\beta=\mapd(\ws)+\alpha+\beta$.
  }
  \label{fig:around_rhombus}
\end{figure}

\textbf{Jacobi  theta functions and Weierstrass/Fay's identity}.
Classically, there are four theta functions,
denoted either by $\theta_{i,j}$ with $i,j=0,1$, or by $\theta_\ell$ with $\ell\in\{1,\dots,4\}$, whose definition may
slightly vary depending on the sources. Among these four, only one is an odd
function. This function, $\theta_{1}$ in the second notation, is the
function we mainly use here, and we simply denote it by $\theta$ when
there is no ambiguity.
Let us recall its definition.
Let $\tau$ be a complex number with positive imaginary part and let $q=e^{i\pi\tau}$.
The \emph{(first) Jacobi theta function} $\theta$ is the entire holomorphic function
defined by the following series:
  \begin{equation}
    \theta(z)=\theta_1(z;q)=\theta_1(z|\tau)=2 \sum_{n=0}^{\infty}(-1)^n
    q^{(n+\frac{1}{2})^2} \sin[(2n+1)z].
    \label{eq:def_theta}
  \end{equation}
The function $\theta$ is antisymmetric and $2\pi$-periodic:
  \begin{equation*}
    \forall\ z\in\mathbb{C},\quad \theta(z+\pi)=\theta(-z)=-\theta(z),
  \end{equation*}
  and also satisfies the quasiperiodic relation
  \begin{equation}\label{equ:quasi_period_theta}
    \forall\ z\in\mathbb{C},\quad \theta(z+\pi\tau) = (-qe^{2iz})^{-1}
    \theta(z).
  \end{equation}

\begin{rem}
  \label{rem:lambda-ellipt}
  The zeros of $\theta$ in $\mathbb{C}$ form a two-dimensional lattice $\Lambda$,
  generated by $\pi$ and $\pi\tau$, and we let $\TT(q):=\CC/\Lambda$ denote the torus
  obtained from $\Lambda$.
  The function $\theta$ is an elementary brick
  to build $\Lambda$-periodic, meromorphic functions, \emph{i.e.}, \@ \emph{$\Lambda$-elliptic} functions.
  For
  example, the ratio
  $\frac{\theta(z-a_1)\theta(z-a_2)}{\theta(z-b_1)\theta(z-b_2)}$ is an elliptic
  function with two simple zeros
  (at~$a_1$ and~$a_2$) and two simple poles (at~$b_1$ and $b_2$) on~$\TT(q)$
  as soon as
  $a_1$, $a_2$, $b_1$ and $b_2$ are distinct, and satisfy
  $a_1+a_2\equiv b_1+b_2 \mod {\Lambda}$;
  and every $\Lambda$-elliptic function with two zeros and two poles is of that
  form (see \emph{e.g.}~\cite[Theorem~15(c)]{Baxter}).
  \end{rem}

A crucial role is played by the following functional identity satisfied by the
theta function. 

\begin{prop}[Fay's trisecant identity/Weierstrass identity]
  For all $(s,u)\in \CC^2$, for all $(a,b,c)\in\CC^3$,
  \begin{multline}
    \frac{\theta(b-a)}{\theta(s-a)\theta(s-b)}
    \frac{\theta(u+s-a-b)}{\theta(u-a)\theta(u-b)}
    +
    \frac{\theta(c-b)}{\theta(s-b)\theta(s-c)}
    \frac{\theta(u+s-b-c)}{\theta(u-b)\theta(u-c)}
    \\
    +
    \frac{\theta(a-c)}{\theta(s-c)\theta(s-a)}
    \frac{\theta(u+s-c-a)}{\theta(u-c)\theta(u-a)}
    =0.
    \label{eq:fay}
  \end{multline}
\end{prop}

This identity, which can be derived from the Weierstrass
identity~\cite{Weierstrass}, see
e.g.\@~\cite[Ch.~1, ex.~4]{Lawden},
can be seen as the genus~1 case
of the more general Fay identity~\cite{Fay} satisfied by the Riemann theta
functions and the associated prime forms on Riemann surfaces. Fay's identity is
a cornerstone of the work of Fock~\cite{Fock} on the inverse spectral problem for Goncharov-Kenyon integrable systems.
We refer the reader to~\cite{george2019spectra}
for an analogous of Fock's results for spectral
curves of Laplacians on minimal periodic planar graphs in connection with Fay's
\emph{quadrisecant identity}.

Translating $a,b,c$ by elements of $\Lambda$ leaves Equation~\eqref{eq:fay}
invariant. Translating $s$ or $u$ gives a global multiplicative factor which
does not change the fact that the sum is zero. Therefore, the parameters in this
identity can really been interpreted as elements of the torus $\TT(q)$.

 By letting $c$ tend to $u$ in Fay's trisecant identity, we immediately obtain the following
 telescopic identity which only depends on $a,b,u$ and $s$, giving
 the version we mostly use:

\begin{cor}\label{cor:Faystrisecant}
 For all $(s,u)\in \CC^2$, for all $(a,b)\in\CC^2$,
  \begin{equation}\label{eq:telefay}
    \frac{\theta(u-s)\theta(b-a)}{\theta(s-a)\theta(s-b)}
    \cdot
    \frac{\theta(u+s-a-b)}{\theta(u-a)\theta(u-b)}
    =
    F^{(s)}(u;b)-F^{(s)}(u;a)
  \end{equation}
  where
  \begin{equation*}
    F^{(s)}(u;a)= \frac{1}{\theta'(0)}
    \left(%
    \frac{\theta'}{\theta}(s-a)
    -
    \frac{\theta'}{\theta}(u-a)
    \right).
  \end{equation*}
\end{cor}

\begin{rem}\label{rem:Fay}
{Corollary~\ref{cor:Faystrisecant} is central to this paper. It is at the heart of the construction of functions in the kernel of Fock's elliptic Kasteleyn operator, see the forthcoming~Proposition~\ref{prop:ker}. The latter is then one of the cornerstones of the proof of Theorem~\ref{thm:intro_main_1}.
This identity plays the role of
(and in fact becomes in a certain limit)
the identity $\frac{a-b}{(z-a)(z-b)}=\frac{1}{z-a}-\frac{1}{z-b}$ used in the genus~0 case by Kenyon~\cite[p. 420]{Kenyon:crit}. More details on how to recover the genus~0 case from the genus~1 case is given in Section~\ref{sec:rational}.}
\end{rem}

Note also that multiplying Equation~\eqref{eq:fay}
by~$\theta(t-a)\theta(t-b)\theta(t-c)\theta(u-a)\theta(u-b)\theta(u-c)$ and
writing~$t:=u+s$ and $d:=s$, we immediately obtain the following elegant version of Fay's
identity~\cite{Fock}:

\begin{cor}\label{cor:FayFock}
For all~$(a,b,c,d)\in\CC^4$, {$t\in\CC$},
\begin{equation}\label{eq:FayFock}
F_t(a,b)F_t(c,d)+F_t(a,c)F_t(d,b)+F_t(a,d)F_t(b,c)=0\,,
\end{equation}
where~$F_t(a,b):=\theta(a-b)\theta(a+b-t)$.
\end{cor}

\section{Family of elliptic Kasteleyn operators}
\label{sec:Kasteleyn}

Let~$\Gs$ be an infinite planar, bipartite graph.
In Section~\ref{sec:Kast_def}, we introduce Fock's one-parameter family of adjacency operators~\cite{Fock} in the genus 1 case, denoted by~$(\K[t])_{t\in\CC}$,
which depend on a half-angle map~$\mapalpha\colon\T\to\RR/\pi\ZZ$ and on a modular parameter~$\tau$.
In Proposition~\ref{prop:kastorient}, we use the results of~\cite{BCdT:immersion} to prove 
that if~$\Gs$ is minimal, if~$\mapalpha$ belongs to~$X_\Gs$ (recall Section~\ref{sec:tt-def}), if the parameter~$\tau$ {belongs to $i\RR_{>0}$} and $t$ lies in~$\RR+\frac{\pi}{2}\tau$, then the operator~$\K[t]$ is actually a Kasteleyn operator (recall Section~\ref{sec:dimer_model}).
In Section~\ref{sec:null_functions}, coming back to the general setting of an arbitrary
graph~$\Gs$, half-angle map~$\mapalpha$ and complex parameter~$t$,
we introduce a
family of functions in the kernel of $\K[t]$.
In Section~\ref{sec:circlepatterns}, we show how these functions define explicit immersions of the dual graph~$\Gs^*$,
in the spirit of the recent paper~\cite{KLRR}.
Finally, in Section~\ref{sec:poles_domainD}, we assume once again the hypotheses of Proposition~\ref{prop:kastorient}
and study 
the relative positions of the poles and the zeros of these functions, a fact used in Section~\ref{sec:preliminaries_contours}.

\subsection{Kasteleyn elliptic operators}\label{sec:Kast_def}

Let~$\Gs$ be an infinite, planar, bipartite graph, and let us fix a half-angle map~$\mapalpha\colon\T\to\RR/\pi\ZZ$ and a modular parameter~$\tau$. Recall that, by Section~\ref{sec:abel_map}, this allows to define the discrete Abel map~$\mapd$ and the Jacobi theta function~$\theta$.

\begin{defi}
\emph{Fock's elliptic adjacency operator}~$\K[t]$ is the complex weighted, adjacency operator of the graph~$\Gs$, whose non-zero coefficients are given as follows:
for every edge~$\ws\bs$ of~$\Gs$ crossed by train-tracks with half-angles~$\alpha$ and~$\beta$ as in Figure~\ref{fig:around_rhombus}, we have
  \begin{equation}\label{def:Kast_elliptic}
    \K[t]_{\ws,\bs}=\frac{\theta(\beta-\alpha)}{\theta(t+\mapd(\bs)-\beta)\theta(t+\mapd(\bs)-\alpha)}\,.
  \end{equation}
\end{defi}

Several remarks are in order.

\begin{rem}\label{rem:half-angle}\leavevmode
\begin{enumerate}
  \item This operator  is the genus~1 case
  of a more general operator
introduced in~\cite{Fock} by Fock on periodic minimal graphs involving Riemann theta
functions of positive genus and their associated prime forms.
  \item By the equality~$\theta(z+\pi)=-\theta(z)$, the coefficient~$\K[t]_{\ws,\bs}$ is unchanged
when
adding a multiple of $\pi$ to $\alpha$, $\beta$ or
$\mapd(\bs)$.
Hence, the operator~$\K[t]$ only depends on the half-angle map~$\T\to\RR/\pi\ZZ$, as it should.
\item By definition of $\mapd$, we have $\mapd(\bs)=\mapd(\ws)+\alpha+\beta$, and
  the denominator can be rewritten differently depending on
  whether we wish to focus on the black vertex~$\bs$, the white vertex~$\ws$ or the
  neighboring faces~$\fs$ or $\fs'$:
  \begin{equation*}
    \theta(t+\mapd(\bs)-\beta)\theta(t+\mapd(\bs)-\alpha) =
    \theta(t+\mapd(\ws)+\alpha)\theta(t+\mapd(\ws)+\beta) =
    \theta(t+\mapd(\fs))\theta(t+\mapd(\fs')).
  \end{equation*}
\end{enumerate}
\end{rem}

We now show that, under some hypotheses on these parameters,
the operator~$\K[t]$ is Kasteleyn. As a consequence, the Boltzmann measure on dimer configurations of a finite connected subgraph of~$\Gs$ can be constructed as a determinantal processes via~$\K[t]$, as stated in Section~\ref{sec:dimer_model}.

Recall that~$X_\Gs$ denotes the space of half-angle maps~$\mapalpha\colon\T\to\RR/\pi\ZZ$ that are monotone with respect to the cyclic orders on~$\T$ and~$\RR/\pi\ZZ$, and that map pairs of intersecting or anti-parallel train-tracks to distinct half-angles.

\begin{prop}
  \label{prop:kastorient}
 Let~$\Gs$ be a minimal graph,~$\mapalpha$ belong to~$X_\Gs$,~$\tau$ {belong to $i\RR_{>0}$}, and~$t$ lie in~$\RR+\frac{\pi}{2}\tau$.
 Then, Fock's elliptic adjacency operator~$\K[t]$ is Kasteleyn. 
\end{prop}

\begin{proof}
Let us compute the argument of the complex number~$\K[t]_{\ws,\bs}$ up to gauge
equivalence (recall Section~\ref{sec:dimer_model}). To do so, first observe that
the theta functions~$\theta=\theta_1$ and~$\theta_4$ are related by
\begin{equation}
\theta_1(u+\frac{\pi\tau}{2})=iq^{-1/4}e^{-iu}\theta_4(u)
\label{eq:theta14}
\end{equation}
for all~$u\in\CC$ (see e.g.~\cite[(1.3.6)]{Lawden}), with~$\theta_4(u)=\theta_4(u|\tau)$ strictly positive for~$u$ real
and~$\tau$ {in $i\RR_{>0}$}.
Since~$t=t'+\frac{\pi\tau}{2}$ with~$t'\in\RR$, we obtain
\[
\K[t]_{\ws,\bs}=\frac{-q^{1/2}\,e^{2i(t'+\mapd(b))}}{\theta_4(t'+\mapd(b)-\beta)\theta_4(t'+\mapd(b)-\alpha)}\theta(\beta-\alpha)e^{-i(\alpha+\beta)}\,.
\]
Note that in the fraction above, the numerator
can be discarded up to gauge equivalence, \emph{i.e.},\@ cancels out when computing the face weight,
while the denominator is strictly positive.
Therefore, up to gauge equivalence, the argument of~$\K[t]_{\ws,\bs}$ is simply
given by the argument of~$\theta(\beta-\alpha)e^{-i(\alpha+\beta)}$. Since~$\theta(u)$ is positive for~$u\in(0,\pi)$ and negative for~$u\in(\pi,2\pi)$, one easily checks that this
argument is equal to~${\phi_\mapalpha:=}\frac{\pi}{2}+\arg(e^{2i\beta}-e^{2i\alpha})$.

{The proposition is now a consequence of the main results of~\cite{BCdT:immersion}, as follows.}
Since~$\Gs$ is minimal and~$\mapalpha$ belongs to~$X_\Gs$,
Theorem~23 of~\cite{BCdT:immersion} can be applied, and the angle map~$\mapalpha$ defines a
minimal immersion of~$\Gs$ {as defined in} Section~\ref{sec:tt-def}.
{To be more precise, this theorem gives a full description of the space~$Y_\Gs$ of minimal immersions of~$\Gs$,
a space which in known to contain~$X_\Gs$.}
Now, by~\cite[Theorem~31]{BCdT:immersion}, {the spaces~$X_\Gs\subset Y_\Gs$ are included in the space~$K_\Gs$ of maps~$\mapalpha$ such that the phase~$\phi_\mapalpha$ satisfies Kasteleyn's condition,
concluding the proof.}
\end{proof}

\begin{rem}\label{rem:conj}\leavevmode
Let us briefly discuss the hypotheses of this proposition.
\begin{enumerate}
  \item
    As explained in detail in
    Section~\ref{sec:periodic}, when~$\Gs$ is~$\ZZ^2$-periodic and~$\mapalpha$ is
    chosen such that~$\K[t]$ is
    a periodic Kasteleyn operator, the associated spectral curve is a Harnack
    curve~\cite{KO:Harnack,KOS}, of geometric genus 1, parameterized by the torus~$\TT(q)$.
    By maximality of Harnack curves, the real locus of this spectral curve has
    two connected components, and hence, so should the real locus of~$\TT(q)$.
    This happens if and only if~$\TT(q)$ is a rectangular torus, \emph{i.e.},
    iff~$\tau$ {belongs to $i\RR_{>0}$}.
    Therefore, at least in the~$\ZZ^2$-periodic case,
    the proposition above does not hold unless~{$\tau\in i\RR_{>0}$}.
\item By the proof above, if~$\tau$ {belongs to $i\RR_{>0}$} and~$t$ lies in~$\RR+\frac{\pi}{2}\tau$, then the argument of~$\K[t]_{\ws,\bs}$ is given by~$\arg(e^{2i\beta}-e^{2i\alpha})$ up to gauge equivalence. Furthermore, if~$\Gs$ is minimal and~$\mapalpha$ belongs to~$X_\Gs$,
{then~$\mapalpha$ belongs to the space~$K_\Gs$ of maps such that}
this argument satisfies Kasteleyn's condition.
Actually~\cite[Theorem~31]{BCdT:immersion} proves that if~$\Gs$ is non-minimal, then~{$K_\Gs$ is empty.
In other words,} there is no half-angle map such that~$\arg(e^{2i\beta}-e^{2i\alpha})$ satisfies Kasteleyn's condition.
Therefore, minimal graphs form the largest class of bipartite planar graphs where
the above argument can be applied.
\end{enumerate}
\end{rem}

\subsection{Functions in the kernel of the elliptic Kasteleyn operator}\label{sec:null_functions}

Inspired by~\cite{Fock}, we introduce a complex valued function
$g^{(t)}$ defined on pairs of vertices of the quad-graph $\GR$ and depending on
a complex parameter $u$, which is in the kernel of the operator $\K[t]$.
This definition extends to the elliptic case that of the function~$f$
of~\cite{Kenyon:crit}. Note that in the
critical case of~\cite{Kenyon:crit}, there is no extra parameter $t$.

When both vertices are equal to a vertex $\xs$ of $\GR$, set $g^{(t)}_{\xs,\xs}(u)\equiv 1$. Next, let
us define $g^{(t)}$ for pairs of adjacent vertices $\vs,\fs$ of $\GR$, where
$\vs$ (resp. $\fs$) is a vertex of $\Gs$ (resp. $\Gs^*$);
let~$\alpha$ be the half-angle of the train-track crossing the edge~$\vs\fs$. Then,
depending on whether~$\vs$ is a white vertex~$\ws$ or a black vertex~$\bs$ of~$\Gs$, we set:
\begin{align*}
g^{(t)}_{\fs,\ws}(u)&=(g^{(t)}_{\ws,\fs}(u))^{-1}=
\frac{\theta(u+t+\mapd(\ws))}{\theta(u-\alpha)},\\
g^{(t)}_{\bs,\fs}(u)&=(g^{(t)}_{\fs,\bs}(u))^{-1}=
\frac{\theta(u-t-\mapd(\bs))}{\theta(u-\alpha)}.
\end{align*}
{These two functions are the extension to the genus 1 case of the functions defined in~\cite[Equations (4) and (5)]{Kenyon:crit}, see also Equation~\eqref{eq:g_crit} in Section~\ref{sec:rational} for more details on the connection.}

Now let $\xs,\ys$ be any two vertices of the quad-graph $\GR$ and consider a path $\xs=\xs_1,\ldots,\xs_n=\ys$
of $\GR$ from $\xs$ to $\ys$. Then, as in the critical case of~\cite{Kenyon:crit},~$\g[t]$ is taken to be the product of the contributions along edges of the path:
\[
g^{(t)}_{\xs,\ys}(u)=\prod_{i=1}^{n-1} g^{(t)}_{\xs_i,\xs_{i+1}}(u).
\]

\begin{lem} For every pair of vertices $\xs,\ys$ of $\Gs$,
the function $g^{(t)}_{\xs,\ys}$ is well-defined, \emph{i.e.}, independent of the choice of path in $\GR$ joining $\xs$ and $\ys$.
\end{lem}
\begin{proof}
It suffices to check that $g^{(t)}$ is well defined around a rhombus $\ws,\fs',\bs,\fs$ of the quad-graph;
let $\alpha$, $\beta$ be the half-angles of the train-tracks defining the rhombus,
see Figure~\ref{fig:around_rhombus}. Then, by definition, the product
$g^{(t)}_{\fs',\ws}\ g^{(t)}_{\ws,\fs}\ g^{(t)}_{\fs,\bs}\ g^{(t)}_{\bs,\fs'}$ is equal to:
\begin{equation*}
\frac{\theta(u+t+\mapd(\ws))}{\theta(u-\alpha)}
\frac{\theta(u-\beta)}{\theta(u+t+\mapd(\ws))}
\frac{\theta(u-\alpha)}{\theta(u-t-\mapd(\bs))}
\frac{\theta(u-t-\mapd(\bs))}{\theta(u-\beta)}
=1.\qedhere
\end{equation*}
\end{proof}

\begin{rem}\label{rem:g_wb}
  In the particular case of a black vertex $\bs$ and a white vertex $\ws$
  along an edge of the graph $\Gs$, using the notation of
Figure~\ref{fig:around_rhombus} and the fact that $\mapd(\ws)=\mapd(\bs)-\alpha-\beta$,
we have
\begin{align*}
  \g[t]_{\bs,\ws}(u)&=\frac{%
    \theta(u+t+\mapd(\ws))\theta(u-t-\mapd(\bs))
  }{%
    \theta(u-\alpha)\theta(u-\beta)
  } \\
                &=\frac{%
    \theta(u+t+\mapd(\bs) -\alpha-\beta)\theta(u-t-\mapd(\bs))
  }{%
    \theta(u-\alpha)\theta(u-\beta)
  }\\
                &=\frac{%
    \theta(u+t+\mapd(\ws))\theta(u-t-\mapd(\ws)-\alpha-\beta)
  }{%
    \theta(u-\alpha)\theta(u-\beta)
  },
\end{align*}
which is
a $\Lambda$-elliptic function,
by Remark~\ref{rem:lambda-ellipt}.
Being a product of $\Lambda$-elliptic functions, 
$\g[t]_{\xs,\ys}(u)$ is itself $\Lambda$-elliptic whenever $\xs$
and $\ys$ are both vertices of $\Gs$. In this case, we consider the parameter
$u$ as living on the torus $\TT(q):=\CC/\Lambda$.
However, this property is not true in general when
$\xs$ or $\ys$ is a dual vertex of $\Gs^*$. Note that $\g[t]_{\bs,\ws}(u)$ is also well defined when the half-angles $\alpha,\beta$ of train-tracks separating $\bs$ and $\ws$
are considered in $\RR/\pi\ZZ$, and that the same holds for $\g[t]_{\xs,\ys}(u)$ when both vertices $\xs,\ys$ belong to $\Gs$.
\end{rem}

The next proposition states that for any
given $u$, the rows and columns of the matrix~$\g[t](u)$,
restricted to white and black vertices respectively, are in the
kernel of $\K[t]$. Although with a different vocabulary, this result is actually contained in Theorem 1 of~\cite{Fock}, hence the attribution. We provide a proof since it is not immediate how to translate Fock's algebraic geometry point of view into ours.

\begin{prop}[\cite{Fock}] \label{prop:ker}
Let $u\in\CC$, and let $\xs$ be a vertex of the quad-graph $\GR$, then:
\begin{enumerate}
  \item
$\g[t]_{\xs,\,\cdot\,}(u)$, seen as a row vector indexed
by white vertices of $\Gs$, is in the left kernel of~$\K[t]$;
equivalently, for every black vertex $\bs$ of $\Gs$, we have
$
\sum_{\ws} \g[t]_{\xs,\ws}(u)\ \K[t]_{\ws,\bs}  =0.
$
\item
$\g[t]_{\,\cdot\,,\xs}(u)$, seen as a column vector indexed
by black vertices of $\Gs$, is in the right kernel of $\K[t]$;
equivalently, for every white vertex $\ws$ of $\Gs$, we have
$
  \sum_{\bs} \K[t]_{\ws,\bs}\ \g[t]_{\bs,\xs}(u)=0.
$
\end{enumerate}
\end{prop}

\begin{proof}
Let us prove the first identity. Using the product form of $\g[t]$, we write~$\g[t]_{\xs,\ws}(u)=\g[t]_{\xs,\bs}(u)\g[t]_{\bs,\ws}(u)$ and factor out
$\g[t]_{\xs,\bs}(u)$, so that we can assume without loss of generality that $\xs=\bs$.

Let $\alpha,\beta$ be the parameters of the train-tracks crossing the edge $\ws\bs$, see Figure~\ref{fig:around_rhombus}.
Then using the definition of the elliptic Kasteleyn operator~\eqref{def:Kast_elliptic} and Remark~\ref{rem:g_wb}, we have
\begin{equation*}
\g[t]_{\bs,\ws}(u)\ \K[t]_{\ws,\bs}  =
\theta(u-t-\mapd(\bs))
\frac{\theta(u+t+\mapd(\bs) -\alpha-\beta)}{\theta(u-\alpha)\theta(u-\beta)}
\frac{\theta(\beta-\alpha)}{\theta(t+\mapd(\bs)-\alpha)\theta(t+\mapd(\bs)-\beta)}.
\end{equation*}
Now using Corollary~\ref{cor:Faystrisecant}, with $s=t+\mapd(\bs),a=\alpha,b=\beta$,
we obtain
\begin{equation*}
\g[t]_{\bs,\ws}(u) \K[t]_{\ws,\bs}= F^{(t+\mapd(\bs))}(u;\beta) - F^{(t+\mapd(\bs))}(u;\alpha).
\end{equation*}

As a consequence, for $\bs$ fixed,
the right-hand side is the generic term of a telescopic sum, which gives zero when summing over the white neighbors of $\bs$.

The proof of the second identity follows the same lines, and it is enough to
check the case where $\xs=\ws$. With the same notation as above,
  rewriting the expression of $\K[t]_{\ws,\bs}$ using
  $\mapd(\ws)=\mapd(\bs)-\alpha-\beta$, we obtain:
  \begin{equation*}
    \K[t]_{\ws,\bs}\ \g[t]_{\bs,\ws}(u)  =
    \frac{\theta(\beta-\alpha)}{\theta(t+\mapd(\ws)+\alpha)\theta(t+\mapd(\ws)+\beta)}
    \frac{\theta(u-t-\mapd(\ws) -\alpha-\beta)}{\theta(u-\alpha)\theta(u-\beta)}
    \theta(u+t+\mapd(\ws)).
  \end{equation*}
  Applying Corollary~\ref{cor:Faystrisecant} again with $s=-t-\mapd(\ws)$
  implies that, for $\ws$ fixed, this is the generic term of a telescopic sum which gives
  zero when summing over the black neighbors of~$\ws$.
\end{proof}

\begin{rem}\label{rem:kernel_gen}
From the function $\g[t]$, it is possible to
  construct more functions in the kernel of $\K[t]$. For example, fix a black
  vertex $\bs$ and let $\Phi$ be
  a generalized function (\emph{e.g.}\@ a measure, or a linear combination of
  evaluations of derivatives) on $\mathbb{C}$ with, for
  definiteness, compact support
  avoiding poles of $g^{(t)}_{\bs,\ws}(u)$, for any white vertex $\ws$. Then the
  action $\langle \Phi, g^{(t)}_{\bs,\ws}\rangle$ of $\Phi$ on each of the entries
  of the vector $\bigl(g^{(t)}_{\bs,\ws}(u)\bigr)_{\ws\in \Ws}$ is a row vector in the
  left kernel of $\K[t]$ by linearity:
  \begin{equation*}
    \langle \Phi, g^{(t)}\rangle \cdot \K[t] = \langle \Phi, g^{(t)}\cdot
    \K[t]\rangle = \langle \Phi, 0\rangle = 0.
  \end{equation*}
  We wonder if all functions in the kernel of $\K[t]$ are of this form.
\end{rem}

\subsection{Graph {realizations} and circle patterns}
\label{sec:circlepatterns}

{In the recent paper~\cite{KLRR}, the authors establish a correspondence between the dimer model on a bipartite graph $\Gs$ and circle patterns with the combinatorics of that graph. More precisely, when the graph is finite and the outer face has a restriction on its degree, or when it is infinite, $\ZZ^2$-periodic and the dimer model is in the liquid phase, the authors assign a circle pattern to the graph $\Gs$ and a convex embedding to the dual graph $\Gs^*$,
see~\cite[Theorem~2 and Theorem 10]{KLRR}. The convex embedding of the dual is referred to as a t-embedding in~\cite{CLR}, see also~\cite[Section~4]{CLR21} for further developments and for the notion of perfect t-embedding.}
Note that the ``t'' in the name t-embedding has nothing
to do with our parameter $t$.

{Prior to tackling the question of the geometric properties of the dual graph $\Gs^*$, \emph{i.e.}, checking that edges are non-intersecting and that faces are convex, the authors define a realization of the dual graph $\Gs^*$}
using
functions in the kernel of the corresponding Kasteleyn operator $\Ks$ when they
exist. {In accordance with the literature, we refer to the latter as a \emph{t-realization}\footnote{{Note that in the terminology of this paper, the most natural term would be \emph{t-immersion}, but we chose the terminology suited to the papers cited in this section.}} of the dual graph.}
More precisely, if $F$, resp.\@ $G$ is in the right, resp.\@ left, kernel of
$\Ks$, then ~$\omega_{\ws\bs}:=G_\ws \Ks_{\ws,\bs}F_\bs$ defines a
divergence free flow $\omega$, so that it can be written as
an increment
\[
\Psi(\fs)-\Psi(\fs')=\omega_{\ws,\bs}
\]
where $\fs\fs'$ is the dual edge of $\ws\bs$, see Figure~\ref{fig:around_rhombus}. {The maps $F$ and $G$ are said to give a \emph{Coulomb gauge} for $\Gs$~\cite[Section 3.3]{KLRR}}.
The {\emph{t-realization}} of $\Gs^*$ is the mapping $\Psi$, defined up to an additive constant by the relation above.

{In our setting of Fock's elliptic adjacency operator, when the graph is infinite, we have explicit, local expressions for a family of Coulomb gauges for $\Gs$. Note that we do not need Fock's operator to be Kasteleyn for this construction to work, but of course, in general we have no control on geometric properties of the t-realizations of $\Gs^*$. } Proposition~\ref{prop:ker} gives explicit functions in the kernel of Fock's
elliptic adjacency operator~$\K[t]$, and thus
  by taking $F^{(t)}=\g^{(t)}_{\cdot,\xs}(u)$ and $G^{(t)}=\g^{(t)}_{\xs,\cdot}(u)$ for some fixed
  vertex~$\xs$ of $\GR$, 
one defines a family {of Coulomb gauges $F^{(t)},G^{(t)}$, and a family}
$(\Psi^{(t)}_{u})_{t\in\CC,u\in \TT(q)}$ of {t-realizations} of the graph $\Gs^*$ indexed by~$t\in\CC$ and~$u\in\TT(q)$. {The Coulomb gauges are \emph{local} in the sense that they are defined as the product of increments along edges; this property is inherited from the incremental definition of the function $g$.}

{As said, for arbitrary values of~$t\in\CC$ and~$u\in \TT(q)$, we have no control on geometric properties of the t-realizations $(\Psi^{(t)}_{u})_{t\in\CC,u\in \TT(q)}$. }
 However, when {the conditions of Proposition~\ref{prop:kastorient} are
 satisfied, then} $\Ks$ is Kasteleyn {and} it is known that $\Psi^{(t)}_{u}$
  defines a local embedding; if furthermore~$\Gs$ and~$\Ks$ are periodic,
  {then $F^{(t)}$ and $G^{(t)}$ are quasiperiodic and} $\Psi^{(t)}$ is
  a \emph{global} {periodic} convex embedding~\cite[{Remark 8} and Theorem 10]{KLRR}.

Using Corollary~\ref{cor:Faystrisecant} as in the proof of
Proposition~\ref{prop:ker} gives an explicit expression for
the increments of the map
$\Psi^{(t)}_{u}$:
\begin{equation}
\begin{split}
\Psi^{(t)}_{u}(\fs)-\Psi^{(t)}_{u}(\fs')= 
\hcancel{\theta'(0)}\g^{(t)}_{\bs,\ws}(u)\Ks^{(t)}_{\ws,\bs}&=
\hcancel{\theta'(0)}
  F^{(t+\mapd(\bs))}(u;\beta) - F^{(t+\mapd(\bs))}(u;\alpha)
\\
&=
\hcancel{\theta'(0)}
F^{(-t-\mapd(\ws))}(u;\beta) - F^{(-t-\mapd(\ws))}(u;\alpha).
\end{split}
\label{eq:gradpsiGdual}
\end{equation}

We refer to Figure~\ref{fig:t-embedding} (left) for an example of such a {t-realization} which
is actually a local embedding of~$\Gs^*$.

The {t-realization}~$\Psi^{(t)}_{u}$ of $\Gs^*$
can be extended into {a realization} of $\GR$ as
follows.
Fix an arbitrary function $\Xi\colon \Vs\rightarrow \RR^2$.
Let $\vs$ and $\fs$ be neighboring vertices in $\GR$ corresponding
respectively to a vertex and a face of $\Gs$, and separated by a train-track
with half-angle $\alpha$. Depending on whether $\vs$ is  a
black vertex $\bs$ or a white vertex $\ws$, the increment of~$\Psi^{(t)}_{u}$
between $\vs$ and $\fs$ is given by the following formulas:
\begin{equation}
  \begin{split}
  \Psi^{(t)}_{u}(\bs)-\Psi^{(t)}_{u}(\fs) &=
  \Xi(\bs)
  +{\frac{1}{\theta'(0)}}\Bigl(\frac{\theta'}{\theta}(t+\mapd(\fs))
  -\frac{\theta'}{\theta}(u-\alpha)\Bigr), \\
  \Psi^{(t)}_{u}(\ws)-\Psi^{(t)}_{u}(\fs) &=
  \Xi(\ws)
  +{\frac{1}{\theta'(0)}}\Bigl(\frac{\theta'}{\theta}(t+\mapd(\fs))
  +\frac{\theta'}{\theta}(u-\alpha)\Bigr).
  \end{split}
  \label{eq:incPsiGdiam}
\end{equation}
{%
Note that this realization of $\GR$ is not related in a simple way to its corresponding
minimal immersion. In particular, the quadrilaterals obtained as the image of
the boundary of a face of $\GR$ are generically not rhombi.
}

\begin{figure}
  \centering
  \begin{minipage}[c]{7cm}
  \includegraphics[width=7cm]{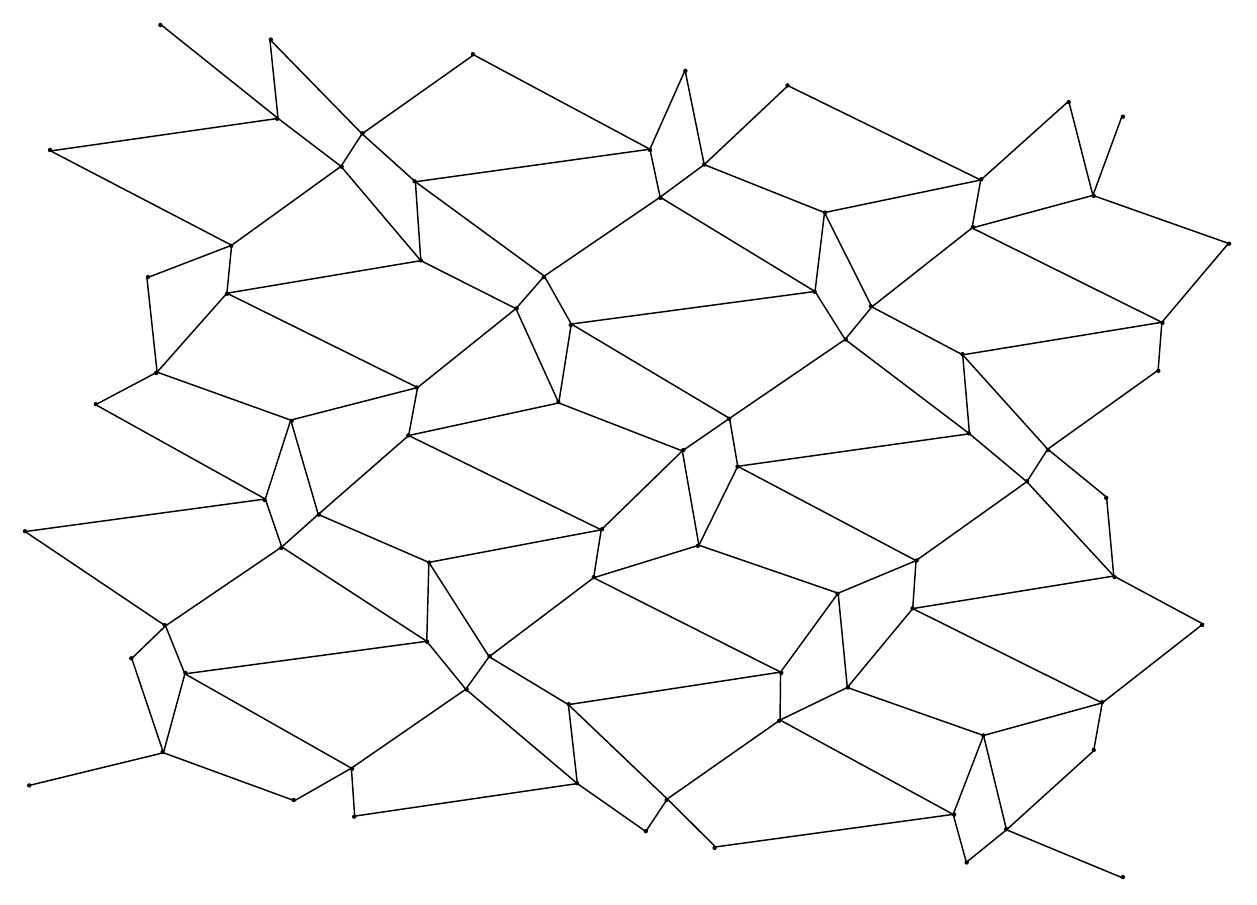}
  \end{minipage}
  \begin{minipage}[c]{7cm}
  \includegraphics[width=7cm]{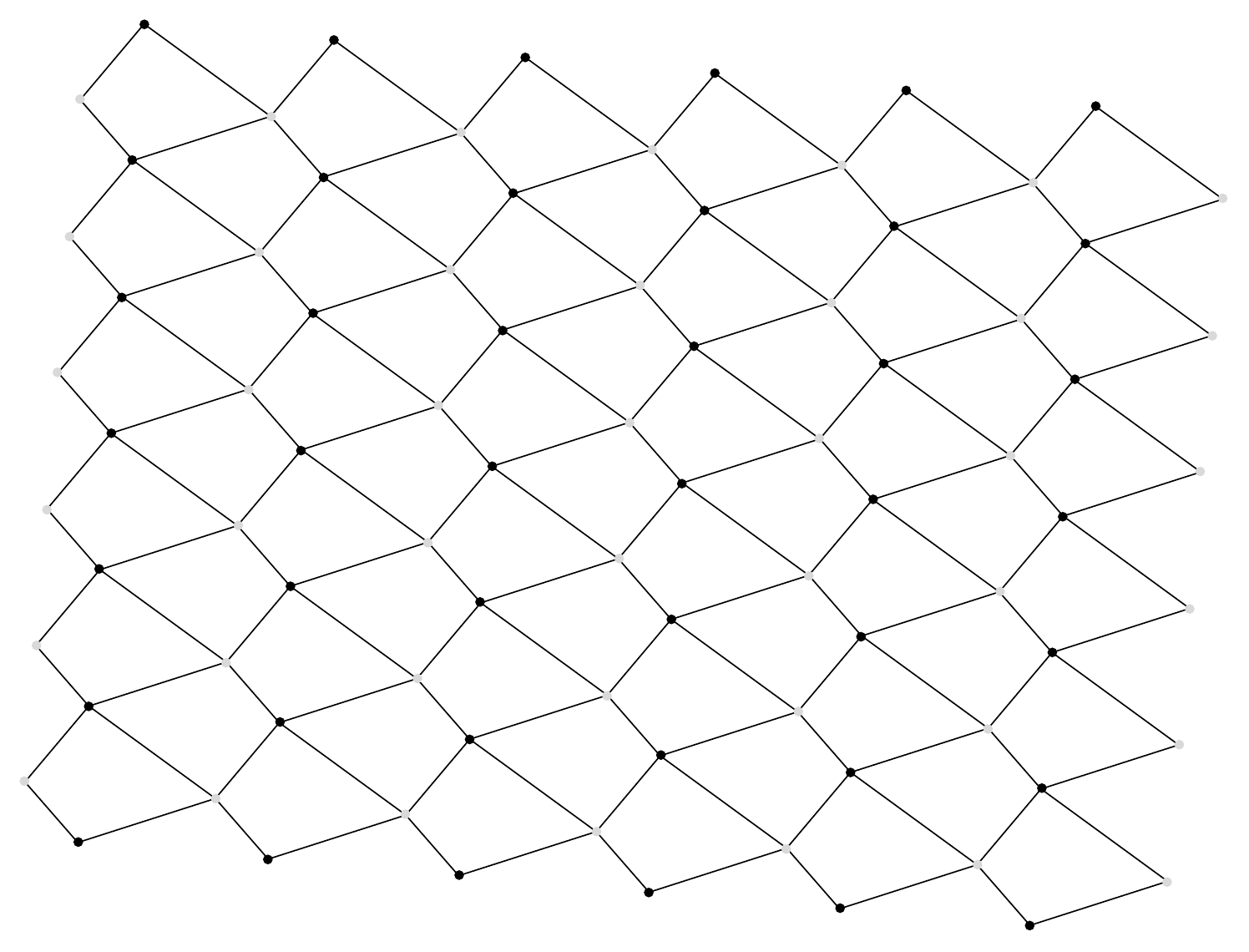}
  \end{minipage}
  \caption{Left: piece of the quasiperiodic 
    local t-embedding of the dual of~$\ZZ^2$ defined
  by~$\Psi_u^{(\pi\tau/2)}$.
    The half-angles assigned to the four train-tracks around every white vertex
    are~$\alpha=0$,~$\beta=\frac{\pi}{6}\simeq 0.52$,~$\gamma=\frac{e}{2}\simeq 1.36$
    and~$\delta=2.5$, while we have~$q=e^{i\pi\tau}=\frac{1}{10}$ and
    $u=0.62+0.70i$. Although~$\ZZ^2$ is periodic as a graph, the fact that the
    half-angles are pairwise incommensurable implies that~$\mapd$ is not
    periodic, but only quasiperiodic. Right: corresponding {realization} of~$\ZZ^2$
    given by the extension of
  $\Psi^{(\pi\tau/2)}_u$ obtained by choosing~$\Xi(\bs)=0$ for all black
  vertices~$\bs$ and $\Xi(\ws)=2i$ for all white vertices~$\ws$.
  {Note that this one is periodic, the increments of $\Psi^{(\pi\tau/2)}$
  between neighboring black and white vertices computed from
  Equation~\ref{eq:incPsiGdiam} do not depend on $\mapd$, which was the only
  remaining source for non-periodicity.}
  This particular immersion is an embedding, a fact which does not hold in general.
 }
\label{fig:t-embedding}
\end{figure}

\begin{lem}
Up to an arbitrary additive constant, the mapping~$\Psi^{(t)}_{u}$ is well-defined
on the vertices of~$\GR$, and
extends the definition of~$\Psi^{(t)}_{u}$ on~$\Gs^*$ given by
Equation~\eqref{eq:gradpsiGdual}.
\end{lem}

\begin{proof}
  The fact that~$\Psi^{(t)}_u$ is well defined on~$\GR$ is an immediate consequence of the fact
  that the four increments 
  around a quadrangular face of
  $\GR$ sum to~$0$.
  Now, let $\fs$ and $\fs'$ be neighbors in $\Gs^*$ and $\bs$ be the black vertex on
  the face of $\GR$ shared by $\fs$ and $\fs'$. Using the notation of
  Figure~\ref{fig:around_rhombus} and
  the equalities $\mapd(\bs)=\mapd(\fs')+\beta=\mapd(\fs)+\alpha$, we obtain:
    \begin{multline*}
    \Psi^{(t)}_u(\fs)-\Psi^{(t)}_u(\fs')=
    \left(\Psi^{(t)}_u(\bs)-\Psi^{(t)}_u(\fs')\right)-
    \left(\Psi^{(t)}_u(\bs)-\Psi^{(t)}_u(\fs)\right)\\
    =
    \Bigl(
      \Xi(\bs)
      +{\frac{1}{\theta'(0)}}\Bigl[\frac{\theta'}{\theta}(t+\mapd(\fs'))
      -\frac{\theta'}{\theta}(u-\beta)\Bigr]
    \Bigr) -
    \Bigl(
      \Xi(\bs)
      +{\frac{1}{\theta'(0)}}\Bigl[\frac{\theta'}{\theta}(t+\mapd(\fs))
      -\frac{\theta'}{\theta}(u-\alpha)\Bigr]
    \Bigr)  \\
    =\hcancel{\theta'(0)}
    F^{(t+\mapd(\bs))}(u;\beta)-
    F^{(t+\mapd(\bs))}(u;\alpha),
  \end{multline*}
  which indeed coincides with~\eqref{eq:gradpsiGdual}.
\end{proof}

We refer to Figure~\ref{fig:t-embedding} (right) for an example of such a {t-realization} of~$\Gs$ (actually an embedding).

Whereas changing $\Xi$ does not have an influence on the {realization} of $\Gs^*$, it obviously has consequences on the {realization} of $\Gs$. For example, adding a constant to
$\Xi(\vs)$ translates the image of $\vs$ without moving the rest.

Once the image by $\Psi^{(t)}_u$ (or $\Xi$) of a single vertex of $\Gs$ is
fixed, there is a unique way to extend $\Psi^{(t)}_u$ to a \emph{circle pattern},
where white and black vertices
around a face $\fs$ are sent to points on a circle centered at $\Psi_u^{(t)}(\fs)$,
as can be seen from~\cite{CLR} and the so-called \emph{origami map}.
However, we do not require this property here.

If the difference between $\Xi_1$ and $\Xi_2$ is bounded, then the two induced
{realization} of $\Gs$ are quasi-isometric.
A trivial choice for $\Xi$ is the constant 0. Another bounded interesting choice is
\begin{equation}
  \Xi(\bs)={\frac{1}{\theta'(0)}}\frac{\theta'}{\theta}(u-t-\mapd(\bs)),
  \quad
  \Xi(\ws)={\frac{1}{\theta'(0)}}\frac{\theta'}{\theta}(u+t+\mapd(\ws)),
  \label{eq:Xilogg}
\end{equation}
which satisfies
\begin{equation*}
  \Psi^{(t)}_u(\bs)-\Psi^{(t)}_u(\ws)={\frac{1}{\theta'(0)}}\frac{d}{du}\log g^{(t)}_{\bs,\ws}(u),
\end{equation*}
for any pair $(\bs,\ws)$ of black and white vertices. {Indeed, by Remark~\ref{rem:g_wb} this is true for adjacent black and white vertices $(\bs,\ws)$, and by the multiplicative nature of $g_{\bs,\ws}$, this extends to any pair $(\bs,\ws)$.}
It follows in particular that for any bounded choice of $\Xi$, the following
estimate is true as
soon as the graph distance between $\bs$ and $\ws$ is large:
\begin{equation}
  \Psi^{(t)}_{u}(\bs)-\Psi^{(t)}_{u}(\ws) = {\frac{1}{\theta'(0)}} \frac{d}{du}\log g^{(t)}_{\bs,\ws}(u)+O(1).
  \label{eq:psi_quasi_isom}
\end{equation}

In Section~\ref{sec:rational}, we give another choice of bounded $\Xi$, well
suited for the connection to the isoradial case~\cite{Kenyon:crit}.

\subsection{Poles and zeros of \texorpdfstring{$\g[t]_{\xs,\ys}$}{g}}\label{sec:poles_domainD}

In this section we consider two vertices $\xs,\ys$ of $\GR$ and study the poles and zeros of 
$u\mapsto \g[t]_{\xs,\ys}(u)$ on the real circle $C_0:=\mathbb{R}/\pi\mathbb{Z}$. More specifically, in Lemma~\ref{lem:sep_zeros_poles} we prove 
that they are well separated. This property is used to define \emph{angular sectors} for the purpose of 
Section~\ref{sec:preliminaries_contours}.

Consider {an oriented simple path} $\Pi$ from $\xs$ to $\ys$ in
the quad-graph $\GR$,
{%
The \emph{intersection number} of $T$ with $\Pi$, denoted $\Pi\wedge T$, is the number times it crosses
$\Pi$ from right to left, minus the number of times it crosses $\Pi$ from left
to right. Because a train-track $T$ cannot cross itself, this intersection number
takes only values in $\{-1,0,1\}$. It does not depend on $\Pi$, only on its
endpoints.
If it is not zero, we say that
$T$ \emph{separates $\xs$ from $\ys$}.
}

Zeros and poles of $g_{\xs,\ys}$ on $C_0$ arise from terms of the 
form $\theta(u-\alpha)^{\pm 1}$ in the product definition of $g_{\xs,\ys}$. More precisely, zeros (resp. poles) $(\alpha_T)$ 
are half-angles of train-tracks $(T)$
{with an intersection number of $+1$ (resp.\ $-1$) with~$\Pi$, \emph{i.e.}, intersecting $\Pi$
  from right to left (resp.\ from left to right).
}
This property implies 
the following result.  

\begin{lem}
  \label{lem:sep_zeros_poles}
 Suppose that the graph $\Gs$ is minimal and that the half-angle map~$\mapalpha$ belongs to~$X_\Gs$. Then,
  there exists a partition of $C_0$ into two intervals, such that one
  contains no poles of $\g[t]_{\xs,\ys}$,
  and the other no zeros.
\end{lem}

\begin{proof}
  {Consider a large ball $B$ of the graph~$\Gs$ containing $\Pi$ and on which one
    can read
    the cyclic order of all the train-tracks separating $\xs$ from $\ys$. 
    On the boundary of $B$,
  every such train-track has an \emph{entry point} where it enters the interior of the ball, and an \emph{exit point}.

  For a train-track $T$ crossing $\Pi$ (possibly several times), we call the \emph{tail} of $T$
  the part from its entry point to its first intersection with $\Pi$. We call its
  \emph{head} the part from the last intersection with $\Pi$ to its exit point.
  The rest of $T$ is called its \emph{body}.

  We assume that $\g[t]_{\xs,\ys}$ has at least two poles and two zeros in $C_0$,
  otherwise, the statement is trivial.
  It is sufficient to show that if $S_0$ and $S_1$ (resp. $T_0$ and $T_1$) are
  distinct train-tracks crossing~$\Pi$ and contributing to poles (resp.\@
  zeros) of $\g[t]_{\xs,\ys}$, then we cannot have the cyclic order
  $\alpha_{T_0} < \alpha_{S_0} < \alpha_{T_1} < \alpha_{S_1}$ on $C_0$.

  Let us fix $S_0$ and $S_1$.
  Since the statement only depends on~$\xs$ and~$\ys$ but not on the path
  between them, we can
  deform the path $\Pi$ so that the
  heads of $S_0$ and $S_1$ no longer intersect.
  Concatenating the heads of $S_0$ and $S_1$, the segment of $\Pi$ between their
  attachment points and one of the two arcs of $\partial B$ between the two exit
  points of~$B$, one obtains the boundary of a topological rectangle $R$ inside
  $B$. For definiteness, we
  suppose that the positively oriented arc of $\partial B$ contained in $R$ starts from
  the exit point of $S_0$ and ends at the exit point of $S_1$, see
  Figure~\ref{fig:sep_zeros_poles}.

  \begin{figure}[ht!]
  \centering
  \def\svgwidth{10cm}
  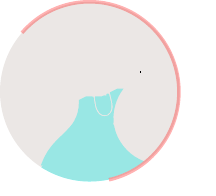
  \caption{Representation of the relative position of the entry point, exit
    point of $T$, and the attachment point $T^+$ of its head with respect to the
    train-track $S_1$ in the proof of Lemma~\ref{lem:sep_zeros_poles}. The
    train-track $S_0$ is slightly faded, as the argument focuses on $S_1$.
  }
  \label{fig:sep_zeros_poles}
\end{figure}

  Consider a train-track $T$ corresponding to a zero of $\g[t]_{\xs,\ys}$, \emph{i.e.},
  with an intersection number with $\Pi$ equal to $+1$. Let us
  show that its exit point never lies on the positively oriented arc of $\partial B$ from the exit point of $S_0$ to that of $S_1$
  the particular arc of $\partial B$
  delimited by the exit points of $S_1$ and $S_0$ (which is inside $R$).
  Indeed, assume the opposite, and consider the last entrance of $T$ into the
  region $R$ before exiting
  through that arc. It is not a point of $\Pi$, as the head of
  $T$ should leave $\Pi$ from its left, and $R$ is attached on its right.
  Assume that it is an intersection point~$\es$ with
  $S_1$
  (the argument with $S_0$ is the same).
  The exit point of $T$ from~$B$ is thus on the right
  of $S_1$, as in Figure~\ref{fig:sep_zeros_poles}. Since the part of $T$ before
  $\es$ cannot intersect the part of $S_1$ before~$\es$ (otherwise, this would create a parallel bigon),
  the attachment point $T^+$ of the head of $T$
  should be on the left of $S_1$.
  For similar reasons, the entry point of $T$ in~$B$ should be on the positively oriented arc of $\partial B$
  starting from the exit point of $T$ and ending at the entry point of $S_1$ (represented
  in light red on Figure~\ref{fig:sep_zeros_poles}).
  But now, the continuous path made of
  the tail of $S_1$ attached to $\Pi$ on its left at point $S_1^-$, the
  segment of $\Pi$ from $S_1^-$ to $T^+$, and the head of
  $T$, is blocking the tail of $T$ from connecting to the right side of $\Pi$.
  This is in contradiction with the fact
  that the intersection number of~$T$ with~$\Pi$ is~$+1$.
  Therefore, the end point of $T$ has to lie on the
  positively oriented arc of $\partial B$ delimited by the exit points of $S_1$ and $S_0$ which is disjoint from $R$.

  Since~$\mapalpha$ belongs to~$X_\Gs$, the cyclic order of the angles is the
  same as that of the train-tracks, implying that one cannot have zeros simultaneously inside both connected components of $C_0\setminus\{\alpha_{S_0},\alpha_{S_1}\}$.
}  
\end{proof}

Let us now restrict to the case where $\xs$ is a black vertex $\bs$ of $\Gs$ and $\ys$ is a white 
one $\ws$. When computing the product for  
$\g[t]_{\bs,\ws}$, all the terms of the form $\theta(u+t+\mapd(\ws'))$ and 
$\theta(u-t-\mapd(\bs'))$
cancel out except the
two terms $\theta(u-t-\mapd(\bs))\theta(u+t+\mapd(\ws))$ in the numerator. 
As a consequence, all the poles of 
$\g[t]_{\bs,\ws}$ are on $C_0$ and, from the above, correspond to half-angles of 
train-tracks
{separating $\bs$ from $\ws$ and leaving $\bs$ on their right.}

The following definition is used in Section~\ref{sec:preliminaries_contours} for defining the contours of integration of our explicit local expressions for inverse Kasteleyn operators.

\begin{defi}
  \label{def:sectors}
  If $\g[t]_{\bs,\ws}$ has at least one zero and one pole on $C_0$, we define the
  \emph{angular sector} (or simply \emph{sector}) associated to $\g[t]_{\bs,\ws}$, denoted by $s_{\bs,\ws}$, to be the part of the
  partition of~$C_0$ containing the poles. If $\g[t]_{\bs,\ws}$ has no zeros on $C_0$ (which
  happens when $\bs$ and $\ws$ are neighbors), then the sector $s_{\bs,\ws}$ is defined
  to be the geometric arc from $\alpha$ to $\beta$ in the positive direction,
  with the convention of Figure~\ref{fig:around_rhombus}.
\end{defi}

\begin{rem}
  In previous works~\cite{BeaCed:isogen, BdTR1}, we had a similar result for
 isoradially embedded graphs,
  where all the rhombus angles are in
  $(0,\frac{\pi}{2})$, using a convexification algorithm~\cite{BeaQuad}.
  Equivalently, this is described in~\cite[Lemma~3.5]{KeSchlenk}.
  The
  resulting sectors were shorter than half of $C_0$. Here, because of the
  possible presence of folded rhombi with angles greater than $\frac{\pi}{2}$,
  the length of this sector may be larger than
  half of~$C_0$.
\end{rem}

{%
  The geometric property described in Lemma~\ref{lem:sep_zeros_poles} has the
following consequence.
\begin{lem}
  \label{lem:inters_sectors}
    Suppose that the graph $\Gs$ is minimal and that the half-angle map $\mapalpha$ belongs to $X_{\Gs}$.
    Let $\bs$ and $\bs'$ be two black vertices of $\Gs$. If $\bs$ and $\bs'$
    are distinct, then the union of sectors~$s_{\bs',\ws}$ with~$\ws$
    adjacent to~$\bs'$
    is strictly smaller than $C_0$. In other words, there is at least a point of
    $C_0$ which belongs to none of the sectors $s_{\bs',\ws}$.
\end{lem}
\begin{proof}
  Let $d$ be the degree of $\bs$. Let $T_1,\ldots, T_d$ be the train-tracks
  crossing the edges of $\GR$ attached to~$\bs$, labeled
  counterclockwise according to their tips, and let
  $\alpha_1,\ldots,\alpha_d$ be their respective half-angles; using cyclic notation whenever appropriate.
  By \cite[Lemma~8]{BCdT:immersion} and the definition of
  the space $X_\Gs$, these parameters satisfy the cyclic order~$\alpha_1<\cdots<\alpha_d$ in~$C_0$.

  Since $\bs$ and $\bs'$ are distinct and~$\Gs$ is minimal, one of the train-tracks $T_i$
  separates $\bs$ and $\bs'$. Let us choose such a train-track, denote it
  by $T$ and its half-angle by $\alpha$.
  Also, let us denote by~$\fs$ the other endpoint of the edge of~$\GR$ attached to~$\bs$

  Since $\bs$ and $\bs'$ are distinct and~$\Gs$ is minimal, at least one of the train-tracks $T_i$
  separates $\bs$ and $\bs'$. For simplicity, let us denote by $T$ such a train-track
  and by $\alpha$ its half-angle (we keep in mind that label $i$ is attached to
  $T$, so that $T=T_i$ and $\alpha=\alpha_i$).
  Once $T$ is fixed, let us denote by~$\fs$ the other endpoint of the edge of~$\GR$ attached to~$\bs$
  and crossed by~$T$,
  see Figure~\ref{fig:intersec_sect}.

\begin{figure}[ht]
\centering
\begin{overpic}[width=10cm]{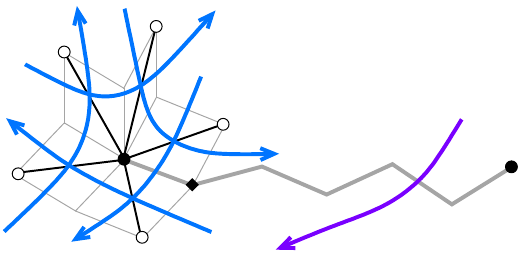}
   \put(27,0){\scriptsize $\ws_0$}
   \put(45,25){\scriptsize $\ws_1$}
   \put(9,40){\scriptsize $\ws$}
   \put(37,10){\scriptsize $\fs$}
   \put(22,15){\scriptsize $\bs$}
   \put(100.5,17){\scriptsize $\bs'$}
   \put(92,10){\scriptsize $\Pi$}
   \put(42.5,46){\scriptsize $T_1=T_j$}
   \put(14,49.5){\scriptsize $T_2=T_{j+1}$}
   \put(-4,29){\scriptsize $T_3=T_{i-1}$}
   \put(8,-2){\scriptsize $T_4=T$}
   \put(55,19){\scriptsize $T_5=T_{i+1}$}
   \put(50,0){\scriptsize $S$}
\end{overpic}
\caption{Notation for the proof of Lemma~\ref{lem:inters_sectors}.}
\label{fig:intersec_sect}
\end{figure}

  Take a simple path $\Pi$ in $\GR$ from $\bs'$ to
  $\bs$, such that all the steps except the last one are on the left of $T$,
  and
  that the last one is the edge between $\fs$ and $\bs$ crossed by $T$ as in Figure~\ref{fig:intersec_sect}.
  We have $\Pi\wedge T=1$, and $\alpha$ is a zero of $\g[t]_{\bs',\bs}$.

  We first deal with the case when $\bs$ has degree 2. In this situation, the
  two white neighbors $\ws_0$ and $\ws_1$ of $\bs$ can be reached from $\bs'$ by
  crossing the same train-tracks in the same direction: first follow $\Pi$ until
  its penultimate step, then bifurcate from $\fs$ to reach $\ws_0$ or $\ws_1$
  by crossing the other train-track around $\bs$. Therefore, the two sectors
  $s_{\bs',\ws_0}$ and $s_{\bs',\ws_1}$ are equal, and the statement follows trivially.

  Suppose now that $\bs$ has degree at least 3.
  We claim that for any white vertex~$\ws$ adjacent to~$\bs$, the
  sector $s_{\bs',\ws}$ does not contain $\alpha$,
  thus implying the statement of the lemma.
  The argument depends on the relative position of $\ws$ with respect to $T$.

  Let us first assume that~$\ws$ is on the right of $T$, as $\bs$ (which is the case for all white
  neighbors of $\bs$ except two vertices, that we call $\ws_0$ and $\ws_1$).
  Then a simple path $\Pi'$ from $\bs'$ to $\ws$ can be obtained by adding two
  steps to $\Pi$, crossing train-tracks $T_j$ and $T_{j+1}$ around $\bs$ that are different from
  $T$ (the fact that the train-tracks $T,T_j,T_{j+1}$ are distinct is a consequence of minimality, see~\cite[Lemma 8]{BCdT:immersion}). These steps correspond to extra factors which can create additional
  poles or possibly remove zeros in $\g[t]_{\bs',\ws}$, when compared to
  $\g[t]_{\bs',\bs}$, at the angle parameters~$\alpha_j,\alpha_{j+1}$ corresponding to these train-tracks.
  Since the graph
  $\Gs$ is minimal and $\mapalpha$ belongs to~$X_\Gs$, all these train-tracks have
  parameters distinct from $\alpha$. As a consequence, $\alpha$ remains a zero of
  $\g[t]_{\bs',\bs}$ and is thus in the complement of
  $s_{\bs',\ws}$ in $C_0$.

  Let us now assume that~$\ws$ is either $\ws_0$ or $\ws_1$. Then, when compared to
  $\g[t]_{\bs',\bs}$, the set of zeros of
  $\g[t]_{\bs',\ws}$ is obtained by possibly removing the zero at $\alpha$ while the
  set of poles is changed by adding a pole at $\alpha_{i-1}$ or at $\alpha_{i+1}$,
  respectively. If one of these three events does not occur, then the complement of both sectors
  $s_{\bs',\ws_0}$ and $s_{\bs',\ws_1}$ contains a small neighborhood around
  $\alpha$ and the statement holds.
  If these three events occur, we claim that it is also the case. Indeed, let us
  look in
  detail at $\ws_0$. We have two cases depending on how close $\bs'$ and
  $\ws_0$ are. The first situation is when
  $\g[t]_{\bs',\ws_0}$ has no zero on $C_0$, which means that $\bs'$ and $\ws_0$
  are neighbors. Then they are separated by the train-tracks $T_{i-1}$ (with
  parameter $\alpha_{i-1}$) and $T'$ (with parameter $\alpha'$).
  By \cite[Lemma~8]{BCdT:immersion}, the corresponding parameters satisfy the cyclic order
  $\alpha'<\alpha_{i-1}<\alpha$ around $C_0$. With our
  convention to define the sector for this particular situation, the complement
  of the sector $s_{\bs',\ws_0}$ contains $\alpha$.

  The second situation occurs when $\g[t]_{\bs',\ws_0}$ has at least a zero $\beta\neq\alpha$. This zero had
  to be present in $\g[t]_{\bs',\bs}$ and comes from a train-track $S$
  crossing $\Pi$, from right to left, see Figure~\ref{fig:intersec_sect}. The same kind of planarity arguments used
  in the proof of the previous lemma show that since the cyclic order~$\alpha_{i-1}<\alpha<\alpha_{i+1}$ holds in~$C_0$,
  then so should the cyclic order~$\alpha_{i-1}<\beta<\alpha_{i+1}$ (without knowing the relative
  position of $\alpha$ and $\beta$ on the oriented arc of $C_0$ from
  $\alpha_{i-1}$ to
  $\alpha_{i+1}$). This is enough to conclude that both complements of the sectors
  $s_{\bs',\ws_0}$ and
  $s_{\bs',\ws_1}$ contain at least the intersection of the component of $C_0$
  containing the zeros of $\g[t]_{\bs',\bs}$ and the interior of the positive
  arc from~$\alpha_{i-1}$ to $\alpha_{i+1}$, and this intersection contains~$\alpha$.
\end{proof}
}

\section{Inverses of the Kasteleyn operator}
\label{sec:inv}

We place ourselves in the context where Fock's elliptic adjacency operator is
Kasteleyn, \emph{i.e.}, we suppose that {$\tau\in i\RR_{>0}$}, that the fixed
parameter $t$ belongs to $\RR+\frac{\pi}{2}\tau$ and that the graph $\Gs$ is
minimal with half-angle map~$\mapalpha\in X_\Gs$. 

In this section, we introduce a family of operators $(\A^{(t),u_0})_{u_0\in D}$ acting as inverses of the Kasteleyn
operator $\K[t]$, parameterized by a subset $D$ of the cylinder $\mathbb{R}/\pi\ZZ + [0,\frac{\pi}{2}\tau]$.
This is one of the main results of this paper. These inverses have the remarkable property of
being \emph{local}, meaning that the coefficient $\A^{(t),u_0}_{\bs,\ws}$ is computed using the information of a path in
 the quad-graph $\GR$ from $\bs$ to $\ws$.

The general idea of the argument to define a local formula for an inverse follows~\cite{Kenyon:crit}:
find functions in the kernel of $\K[t]$ depending on a complex
parameter, \emph{i.e.}, the functions $\g[t]$ introduced in Section~\ref{sec:null_functions} in the elliptic setting of this paper; then
define coefficients of the inverse as contour integrals of these functions, with
appropriately defined paths of integration.
On top of handling the elliptic setting, the novelty of this paper is to introduce an additional parameter $u_0\in D$, leading to
three different asymptotic behaviors for the inverses, morally corresponding
to the three phases of the dimer model: liquid, gaseous, and solid. The three cases are determined by the position of $u_0$ in $D$.

In Section~\ref{sec:preliminaries_contours}, we define the domain $D$ for the parameter $u_0$,
and the paths of integration.
Relying on this, in Section~\ref{sec:inverses}
we introduce the family of inverses $(\A^{(t),u_0})_{u_0\in D}$.
Finally, in Section~\ref{sec:def_Hu0}, we give the explicit form of the function~$H^{u_0}$ involved in
an alternative expression of~$\A^{(t),u_0}$. 

From now on, we omit the superscript $(t)$ in the notation of $\K[t],\A[t]$ and $\g[t]$.

\subsection{Domain \texorpdfstring{$D$}{D} and paths of integration}\label{sec:preliminaries_contours}

Let $\bs,\ws$ be a black and a white vertex of $\Gs$ respectively. Recall that the function $g_{\bs,\ws}$ of Section~\ref{sec:null_functions} is defined on the torus
$\TT(q)=\CC/\Lambda$, {where $q=e^{i\pi\tau}$}, and also recall the angular sector $s_{\bs,\ws}$ of Definition~\ref{def:sectors}.
Since the parameter $\tau$ {belongs to $i\RR_{>0}$},
the real locus of the torus $\TT(q)$ has two connected components,
$C_0=\mathbb{R}/\pi\ZZ$ and $C_1=(\mathbb{R}+\frac{\pi}{2}\tau)/\pi\ZZ$. 

We define the domain $D$ of the parameter $u_0$ indexing the family of inverses
$(\A^{u_0})_{u_0\in D}$ as follows.
Consider the set of angles $\{\alpha_T\ ;\ T\in\T\}$ assigned to the train-tracks
of $\GR$, then the domain $D$ is, see also Figure~\ref{fig:domaineD},
\begin{equation*}
  D=\left(\mathbb{R}/\pi\mathbb{Z} + \Bigl[0,\frac{\pi}{2}\tau\Bigr]\right) \setminus\{\alpha_T\ ;\ T\in\T\}.
\end{equation*}

\begin{figure}[ht]
\centering
\def\svgwidth{7cm}
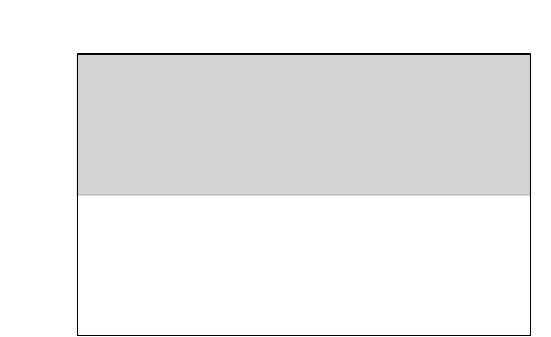
\caption{The domain $D$ as a shaded area of the torus $\TT(q)$ and the different cases corresponding to the possible locations of the parameter $u_0$.
The horizontal contours $C_0$ and $C_1$ winding around the torus are the two
connected components of the real locus of~$\TT(q)$, {and the crosses on $C_0$ represent the set of angles $\{\alpha_T\ ;\ T\in\T\}$}.}
\label{fig:domaineD}
\end{figure}

We now introduce paths/contours of integration for $\A_{\bs,\ws}^{u_0}$, denoted by~$\Cs_{\bs,\ws}^{u_0}$.
We distinguish three cases depending on the position of $u_0$ in $D$. Note that in order to keep notation as light as possible,
we do not add indices specifying the cases, hoping that this creates no confusion.

\underline{\textbf{Case 1}}: $u_0$ is
on the top boundary $C_1=\RR/\pi\ZZ + \frac{\pi}{2}\tau$
of the domain $D$. Then, see also Figure~\ref{fig:domains_of_integration} (left),
$\Cs_{\bs,\ws}^{u_0}$ is a simple contour in $\TT(q)$ winding around the torus once from bottom to top,
 such that its intersection with $C_0$ avoids the angular sector $s_{\bs,\ws}$.

\underline{\textbf{Case 2}}: $u_0$ belongs to the interior of $D$. Then, see also Figure~\ref{fig:domains_of_integration} (center),
$\Cs_{\bs,\ws}^{u_0}$ is a simple path in $\TT(q)$ connecting $\bar{u}_0$ to $u_0$, crossing $C_0$ once but not $C_1$, and
  avoiding the sector $s_{\bs,\ws}$. 

\underline{\textbf{Case 3}}: $u_0$ belongs to the lower boundary of $D$,
\emph{i.e.}, it is a point corresponding
to one of the connected components of $C_0 \setminus \{\alpha_T\ ;\ T\in \T\}$.
Then, see also Figure~\ref{fig:domains_of_integration} (right),
$\Cs_{\bs,\ws}^{u_0}$ is a simple, homologically trivial contour in $\TT(q)$, oriented counterclockwise, crossing
$C_0$ twice: once in the complement of the angular sector $s_{\bs,\ws}$, 
from bottom to
top, and once in the open interval
of $C_0\setminus\{\alpha_T\ ;\ T\in \T\}$
containing the point $u_0$, 
from top to bottom.
Note that this contour may well
not contain all poles of the integrand $g_{\bs,\ws}$.

In each of the three cases, we consider a meromorphic function $H^{u_0}$ on 
$\TT(q)\setminus\Cs_{\bs,\ws}^{u_0}$ with a discontinuity jump of $+1$ when
crossing $\Cs_{\bs,\ws}^{u_0}$ from right to left,
and a collection of
homologically trivial contours 
$\gamma_{\bs,\ws}^{u_0}$ surrounding all the poles of $g_{\bs,\ws}$ and of
$H^{u_0}$ counterclockwise. In Cases 1 and 2,
the collection $\gamma_{\bs,\ws}^{u_0}$ consists of a single contour, while in
Case 3, it consists of two contours; see Figure~\ref{fig:domains_of_integration}.
We refer to Section~\ref{sec:def_Hu0} for explicit candidates for
$H^{u_0}$.

\begin{figure}[ht]
\centering
\begin{overpic}[width=\linewidth]{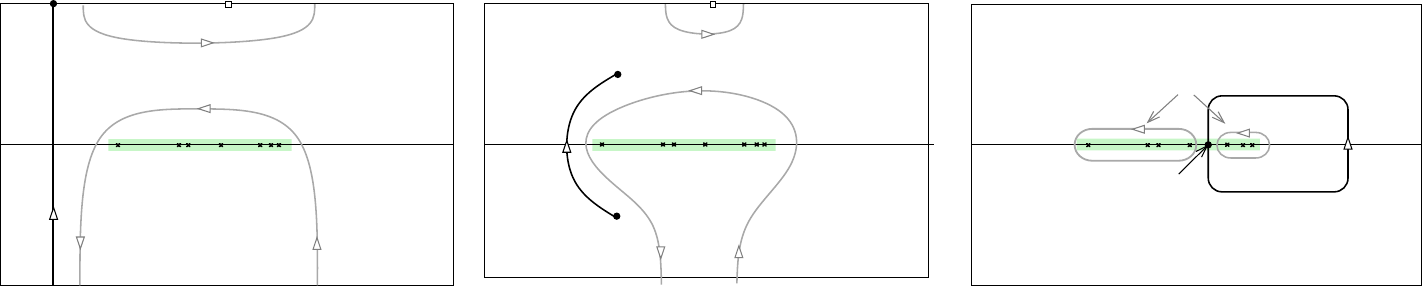}
   \put(4,13){\scriptsize $\Cs_{\bs,\ws}^{u_0}$}
   \put(22,8){\scriptsize $\gamma_{\bs,\ws}^{u_0}$}
   \put(2.8,21){\scriptsize $u_0$}
   \put(36,13){\scriptsize $\Cs_{\bs,\ws}^{u_0}$}
   \put(56,8){\scriptsize $\gamma_{b,w}^{u_0}$}
   \put(43,16){\scriptsize $u_0$}
   \put(43,3){\scriptsize $\bar{u}_0$}
 \put(95,13){\scriptsize $\Cs_{\bs,\ws}^{u_0}$}
 \put(82,15){\scriptsize$\gamma_{\bs,\ws}^{u_0}$}
   \put(82,6.8){\scriptsize $u_0$}
\end{overpic}
\caption{The contours/paths $\Cs^{u_0}_{\bs,\ws}$ and $\gamma^{u_0}_{\bs,\ws}$
  in: Case 1 (left), Case 2 (center), Case 3 (right). The angular sector $s_{\bs,\ws}\subset C_0$
containing the poles of $g_{\bs,\ws}$ is represented in light green. 
{In Cases 1 and 2, the white square represents the pole $\frac{\pi}{2}\tau$ of the choice of function $H^{u_0}$ made in Section~\ref{sec:def_Hu0}; in Case 3, the function $H^{u_0}$ has no pole.}
}
\label{fig:domains_of_integration}
\end{figure}

\subsection{Family of inverses}\label{sec:inverses}

We now define the family of operators $(\A^{u_0})_{u_0\in D}$ in two equivalent ways and then, in
Theorem~\ref{thm:K_inverse_family}, prove that they are indeed inverses of the Kasteleyn operator~$\K$.

\begin{defi}\label{def:operator_A}
For every $u_0$ in $D$, we define the linear operator $\A^{u_0}$
mapping functions on white vertices (with finite support for
definiteness) to functions on black vertices
by its entries: for every pair $(\bs,\ws)$ of black and white vertices of $\Gs$, let
\begin{equation}\label{eq:coeff_Kinv_u0}
    \A^{u_0}_{\bs,\ws} = \frac{i\theta'(0)}{2\pi} \int_{\Cs_{\bs,\ws}^{u_0}} g_{\bs,\ws}(u)du,
\end{equation}
where the path of integration $\Cs_{\bs,\ws}^{u_0}$ is defined in Section~\ref{sec:preliminaries_contours} -- recall that there are three different
definitions depending on whether $u_0$ is
on the top boundary of the domain $D$, a point in the interior, or a point
in a connected component of the lower boundary of $D$. 
\end{defi}

\begin{rem}\label{rem:Au_0} \leavevmode
\begin{enumerate}
\item The operator $\A^{u_0}$ is \emph{local} in the sense that its coefficient $\A^{u_0}_{\bs,\ws}$ is computed using the function $\g_{\bs,\ws}$ which only depends on a path from $\bs$ to $\ws$ in the quad-graph $\GR$ and actually does not depend on the choice of path; $\A^{u_0}_{\bs,\ws}$ only uses \emph{local} information of the graph $\GR$ while one would a priori expect it to use the combinatorics of the whole of the graph $\GR$.
\item  The integrand $\g_{\bs,\ws}$ is meromorphic on the torus $\TT(q)$, so continuously deforming
  the contour of integration $\Cs_{\bs,\ws}^{u_0}$ (while keeping the
  extremities fixed in Case~2) without crossing any poles does not change the
  value of the integral. In particular, in Case~1, all the values of $u_0$ on
  the top boundary of the cylinder $D$ give the same operator. Similarly, in Case~3,
  all the values of $u_0$ in the same connected component of
  $C_0\setminus\{\alpha_T\ ;\ T\in\T\}$ yield the same operator. We can thus
  identify in $D$ all the points on the top boundary, and points in each of the
  connected component of $C_0\setminus\{\alpha_T\ ;\ T\in\T\}$.
\end{enumerate}
\end{rem}

The following lemma gives an alternative, useful way of expressing the coefficients of~$\A^{u_0}$.

\begin{lem}\label{lem:Kinv_alternative}
For every $u_0$ in $D$ and every pair $(\bs,\ws)$ of black and white vertices of $\Gs$, the coefficient $\A^{u_0}_{\bs,\ws}$ of~\eqref{eq:coeff_Kinv_u0}
can be expressed as:
\begin{equation}\label{eq:coeff_Kinv_u0_alt}
\A^{u_0}_{\bs,\ws} = \frac{i\theta'(0)}{2\pi} \oint_{\gamma_{\bs,\ws}^{u_0}} g_{\bs,\ws}(u) H^{u_0}(u) du\,,
\end{equation}
where the function $H^{u_0}$ and the contour
$\gamma_{\bs,\ws}^{u_0}$ are described at the end of Section~\ref{sec:preliminaries_contours}.
\end{lem}

\begin{rem}
The explicit definition of the function $H^{u_0}$ is postponed to Section~\ref{sec:def_Hu0}. Indeed,
at this point, only its qualitative behavior is needed. The explicit form of $H^{u_0}$ is
used when computing edge-probabilities for the corresponding Gibbs measures, see
Section~\ref{sec:Gibbsnonperio}. 
\end{rem}

\begin{proof}[Proof of Lemma~\ref{lem:Kinv_alternative}]
  In each of the three cases, the family of
  contours $\gamma_{\bs,\ws}^{u_0}$ is
  homologous,
  inside the complement of 
  the poles of~$g_{\bs,\ws}H^{u_0}$ in~$\TT(q)\setminus\Cs_{\bs,\ws}^{u_0}$,
  to the family of contours given by the (clockwise oriented) boundary of a small
  bicollar neighborhood of $\Cs_{\bs,\ws}^{u_0}$.
  The contribution of the integrand on both sides of
  $\Cs_{\bs,\ws}^{u_0}$ are on different sides of the cut for $H^{u_0}$ and thus
  differ by $-1$. Recombining these two contributions as a single integral along
  $\Cs_{\bs,\ws}^{u_0}$ yields Equation~\eqref{eq:coeff_Kinv_u0}.
\end{proof}

We now state the main theorem of this section.

\begin{thm}\label{thm:K_inverse_family}
For every $u_0$ in $D$, $\A^{u_0}$ is an inverse of the Kasteleyn operator $\K$.
\end{thm}

\begin{proof}
  We need to check that we have~$\sum_{\bs} \K_{\ws,\bs} \A^{u_0}_{\bs,\ws'} = \delta_{\ws,\ws'}$ for every pair of white vertices $\ws,\ws'$,
  and~$\sum_{\ws} \A^{u_0}_{\bs',\ws} \K_{\ws,\bs} = \delta_{\bs,\bs'}$
  for any pair of black vertices $\bs,\bs'$.
  We only give the proof of the second identity, the other being proved in a
  similar way. The idea of the argument follows~\cite{Kenyon:crit}, see also~\cite{BeaCed:isogen,BdTR1}. 
  If $\bs\neq \bs'$, we use the main definition~\eqref{eq:coeff_Kinv_u0} of
  the coefficients of $\A^{u_0}$.
 {By Lemma~\ref{lem:inters_sectors},}  the intersection of the complements of the sectors $(s_{\bs',\ws})_{\ws\sim \bs}$
  is non-empty.
  It
  is therefore possible to continuously deform all
  the contours~$(\Cs^{u_0}_{\bs',\ws})_{\ws\sim \bs}$ into a
  common contour $\Cs^{u_0}$. By Proposition~\ref{prop:ker} and Remark~\ref{rem:kernel_gen}, we then have:
  \begin{equation*}
    \sum_{\ws:\,\ws\sim \bs} \A^{u_0}_{\bs',\ws} \K_{\ws,\bs} = \frac{i\theta'(0)}{2\pi}
    \int_{\Cs^{u_0}} \underbrace{\sum_{\ws:\,\ws\sim \bs} g_{\bs',\ws}(u) \K_{\ws,\bs}}_{=0} du  =0.
  \end{equation*}
  If $\bs=\bs'$, the points of intersection of the paths/contours $\Cs_{\bs,\ws}^{u_0}$ with the real locus~$C_0$ of the torus
  wind around $C_0$ as $\ws$ runs through the neighbors of $\bs$. We cannot apply
  Proposition~\ref{prop:ker} anymore, but can resort to explicit residue
  computations using the alternative expression~\eqref{eq:coeff_Kinv_u0_alt} for the coefficients of $\A^{u_0}$. We need to compute:
  \begin{equation*}
    \frac{i\theta'(0)}{2\pi} \sum_{\ws:\,\ws\sim \bs}
    \K_{\ws,\bs}\oint_{\gamma^{u_0}_{\bs,\ws}} g_{\bs,\ws}(u)H^{u_0}(u)  du.
  \end{equation*}
  By the residue theorem, each of these integrals is equal to the sum of the
  residues at the poles of
  $g_{\bs,\ws}(u)H^{u_0}(u)$ inside the contour. The poles are of two kinds: first,
  the possible pole(s)
  of $H^{u_0}$, which do not depend on $\bs$, $\ws$, yielding the evaluation of $g_{\bs,\ws}$
  (or its derivatives in case of higher order poles) at some value of $u$, which are in the kernel of $\K$ and
  thus will contribute zero when summing over $\ws$. Second,
  the poles at~$\alpha,\beta$ of~$\g_{\bs,\ws}$, see Remark~\ref{rem:g_wb}, where~$\alpha,\beta$ are the
  parameters of the train-tracks crossing the edge~$\ws\bs$. An explicit evaluation gives
  \begin{align*}
    \res_{\alpha} g_{\bs,\ws}(u)H^{u_0}(u) &=
    \frac{%
      \theta(\alpha-t-\mapd(\bs))\theta(\alpha+t+\mapd(\ws))
    }{%
      \theta'(0) \theta(\alpha-\beta)
    } H^{u_0}(\alpha), \\
    \res_{\beta} g_{\bs,\ws}(u) H^{u_0}(u) &=
    \frac{%
      \theta(\beta-t-\mapd(\bs))\theta(\beta+t+\mapd(\ws))
    }{%
      \theta'(0) \theta(\beta-\alpha)
    } H^{u_0}(\beta).
  \end{align*}
  Using the fact that
  $\mapd(\ws)+\alpha=\mapd(\bs)-\beta=\mapd(\fs')$ and
  $\mapd(\ws)+\beta=\mapd(\bs)-\alpha=\mapd(\fs)$, and recalling the definition of
  $\K_{\ws,\bs}$, we obtain that for every edge $\ws\bs$ of $\Gs$,
  \begin{equation*}
    (\res_\alpha g_{\bs,\ws}(u)H^{u_0}(u) +\res_\beta g_{\bs,\ws}(u)H^{u_0}(u))\K_{\ws,\bs} =
    \frac{1}{\theta'(0)}(H^{u_0}(\alpha)-H^{u_0}(\beta)).
  \end{equation*}
  When summing over white vertices incident to the vertex $\bs$ of degree $d$, surrounded by train-tracks
  with half-angles $\alpha_1=\alpha,\alpha_2,\ldots,\alpha_d,\alpha_{d+1}=\alpha+\pi$, the
  increments of $H^{u_0}$ sum to
  $H^{u_0}(\alpha_1)-H^{u_0}(\alpha_{d+1})=H^{u_0}(\alpha)-H^{u_0}(\alpha+\pi)=-1$
  by construction.
  Therefore, we get:
  \begin{equation*}
    \sum_{\ws:\,\ws\sim \bs} \A^{u_0}_{\bs,\ws} \K_{\ws,\bs} = 2i\pi
    \frac{i\theta'(0)}{2\pi}\frac{H^{u_0}(\alpha)-H^{u_0}(\alpha+\pi)}{\theta'(0)}=1.
  \qedhere
  \end{equation*}
\end{proof}

\begin{rem}
  \label{rem:frozen}
Let us note that in Case 3,
we can work directly with residues on the expression~\eqref{eq:coeff_Kinv_u0}
since $\Cs_{\bs,\ws}^{u_0}$ in this case is a trivial contour.
Indeed, label by~$\alpha_1,\dots,\alpha_d$ the
half-angles of the train-tracks surrounding the vertex $\bs$ so that $\alpha_1$ is the first half-angle on the
right of $u_0$ and $\alpha_d$ is the last angle on the left; denote by $\bs\ws_j$ the edge with train-track angles
$\alpha_j,\alpha_{j+1}$.
We refer to Figure~\ref{fig:case3} for a representation of
  the neighborhood of $\bs$ in the minimal
  immersion~\cite{BCdT:immersion} of $\Gs$ defined by the map $\mapalpha$.
\begin{figure}
\centering
\begin{overpic}[width=8cm]{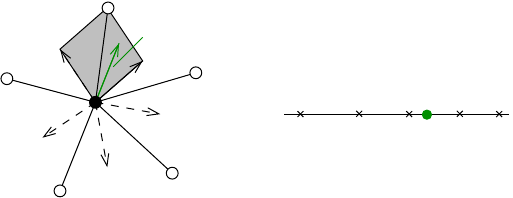}
   \put(20,13.5){\scriptsize $\bs$}
   \put(23,36){\scriptsize $\ws_d$}
   \put(-5,22){\scriptsize $\ws_1$}
   \put(10,-3){\scriptsize $\ws_2$}
   \put(3,28){\scriptsize $e^{2i\alpha_1}$}
   \put(29,26){\scriptsize $e^{2i\alpha_d}$}
   \put(3,8){\scriptsize $e^{2i\alpha_2}$}
   \put(27,32){\scriptsize $e^{2iu_0}$}
   \put(82,12){\scriptsize $u_0$}
   \put(88,19){\scriptsize $\alpha_1$}
   \put(78,19){\scriptsize $\alpha_d$}
   \put(95,19){\scriptsize $\alpha_2$}
\end{overpic}
\vspace{0.4cm}
\caption{Proof of Case 3: the coefficient $\A_{\bs,\ws_d}$ is non-zero if and
  only if the rhombus corresponding to the edge $\bs\ws_d$
  in the minimal immersion of $\Gs$ given by $\mapalpha$
  contains the vector $e^{2iu_0}$. 
}
\label{fig:case3}
\end{figure}
Then by definition, for every $j\neq d$, the contour $\Cs_{\bs,\ws_j}^{u_0}$ either contains both poles of $g_{\bs,\ws_j}$
or none of them.
In both cases this gives $\A_{\bs,\ws_j}^{u_0}=0$. When $j=d$, the contour $\Cs_{\bs,\ws_d}^{u_0}$ contains the pole $\alpha_1$ of
$g_{\bs,\ws_d}$ but not $\alpha_d$; therefore
\[
\A_{\bs,\ws_d}\!=-2i\pi\frac{i\theta'(0)}{2\pi} \frac{\theta(t+\mapd(\ws)+\alpha_1)\theta(t+\mapd(\ws)+\alpha_d)}{\theta(\alpha_1-\alpha_d)\theta'(0)}=
\frac{\theta(t+\mapd(b)-\alpha_d)\theta(t+\mapd(b)-\alpha_1)}{\theta(\alpha_1-\alpha_d)}
\]
and we conclude in particular that
\begin{equation*}
\sum_{\ws:\,\ws\sim \bs}\A_{\bs,\ws}\Ks_{\ws,\bs}=\A_{\bs,\ws_d}\Ks_{\ws_d,\bs}=1,
\end{equation*}
providing an alternative proof in Case 3. We refer to Section~\ref{sec:gibbs_non_perio}
for a probabilistic interpretation of this computation.
\end{rem}

\begin{rem}
  \label{rem:more_gen_inverses}
  Although the computations above rely in an essential way on the
  graph being minimal and the half-angle map belonging to~$X_{\Gs}$,
  they do not use the fact that~$\tau$ {belongs to $i\RR_{>0}$} and that the endpoints of $\Cs_{\bs,\ws}$ are conjugate of each
  other. Hence, this recipe to construct inverses of $\Ks$ is slightly more general
  than what is described here, and two such inverses differ by a function
  of the form described in Remark~\ref{rem:kernel_gen}. However, for
  probabilistic aspects described in the sequel, we restrict ourselves
  to this setting.
\end{rem}

\subsection{Definition of the function \texorpdfstring{$H^{u_0}$}{Hu₀}}\label{sec:def_Hu0}

Recall that $H^{u_0}$
is a meromorphic function on
$\TT(q)\setminus\Cs_{\bs,\ws}^{u_0}$ with a discontinuity jump of~$+1$ when crossing
$\Cs_{\bs,\ws}^{u_0}$ from right to left.
We now give expressions of functions satisfying this property, depending on the
location of $u_0$. Adding any $\Lambda$-elliptic function to these expressions
gives new candidates for $H^{u_0}$ with different poles and residues, but this
has no effect on the resulting value of $\As^{u_0}$.

{In Case 1 (resp. in Case~2), the function $H^{u_0}$ for given
$\bs$ and $\ws$ should be thought of as a
particular determination (depending on $\bs$ and $\ws$) of a multivalued meromorphic
function on~$\TT(q)$ (resp.~$\TT(q)\setminus{\{u_0,\overline{u}_0\}}$) given by the projection of a meromorphic function on an the infinite cyclic
cover of this surface determined by~$\Cs_{\bs,\ws}^{u_0}$.
In any case, even though the function~$H^{u_0}$ depends on~$\bs$ and~$\ws$, it can be chosen so that its poles
and residues do not depend on these vertices, hence their absence in the notation.}
In particular, this multivalued meromorphic function, which is a periodic analogue of the
complex logarithm~\cite{Kenyon:crit},  satisfies~$H^{u_0}(\alpha+\pi)-H^{u_0}(\alpha)=+1$
since the contour $\Cs_{\bs,\ws}^{u_0}$ intersects $C_0$ only
once positively in these cases.

\underline{\textbf{Case 1}}. The function $H^{u_0}$ has a period +1 when winding horizontally around the torus.
It has been explicitly constructed
with a slightly different normalization
in~\cite{BdTR1}, and is given by:
\begin{equation}\label{equ:H0gas}
  H^{u_0}(u):=\frac{K'}{\pi}\tilde{Z}(u)+\frac{u}{\pi}, 
\end{equation}
where $\tilde{Z}(u)=Z(\frac{2K}{\pi}u|k)=\frac{\pi}{2K}\frac{\theta_4'(u)}{\theta_4(u)}$,
$Z$ is 
the \emph{Jacobi zeta function}, see for example~\cite[(3.6.1)]{Lawden}; $k$ is related to $q$ by the relation
$k=\frac{\theta_2^2(0|q)}{\theta_3^2(0|q)}$, $K=\frac{\pi}{2}\theta_3^2(0)$, and $iK'=\tau K$. The function $H^{u_0}$ has a single pole at $\frac{\pi}{2}\tau$ on the torus $\TT(q)$. The function $\tilde{Z}$ has no horizontal period, but it has a vertical period, see~\cite[(3.6.22)]{Lawden}:
\[\textstyle 
\tilde{Z}(u+\pi\tau)=Z(\frac{2K}{\pi}u+2K\tau)
=Z(\frac{2K}{\pi}u+2iK')=-i\frac{\pi}{K}+\tilde{Z}(u),
\]
implying that~$H^{u_0}(u+\pi)=H^{u_0}(u)+1$ and $H^{u_0}(u+\pi\tau)=H^{u_0}(u)$.

Note that~$H^{u_0}(u)=H(\frac{4K}{\pi}u|k)$, where the function $H$ is defined in~\cite[Equation (9)]{BdTR1}.
The above properties are proved in more detail in~\cite[Lemma 45]{BdTR1}, see also~\cite[Appendix A.2]{BdTR2}.

\underline{\textbf{Case 2}}. 
The function $H^{u_0}$ has a period +1 when winding horizontally around the torus at a
 height between $\bar{u}_0$ and $u_0$. It
is given by the following explicit expression:
\begin{equation}\label{equ:H0liq}
  H^{u_0}(u):=
  \frac{1}{2\pi i}\log\frac{\theta(u-u_0)}{\theta(u-\bar{u}_0)}
  -\frac{iK}{\pi^2}(u_0-\bar{u}_0)\tilde{Z}(u),
\end{equation}
and has a single pole at $\frac{\pi}{2}\tau$ on $\TT(q)$. 

Indeed, the function $\log\frac{\theta(u-u_0)}{\theta(u-\bar{u}_0)}$ has the
correct horizontal period as the
function~$\frac{\theta(u-u_0)}{\theta(u-\bar{u}_0)}$ has a zero at $u_0$ and a pole at $\bar{u}_0$. It has vertical period:
\[
\log \frac{\theta(u+\pi\tau-u_0)}{\theta(u+\pi\tau-\bar{u}_0)}=
2i(u_0-\bar{u}_0)+
\log \frac{\theta(u-u_0)}{\theta(u-\bar{u}_0)},
\]
using that $\theta(u+\pi\tau)=(-qe^{2iu})^{-1}\theta(u)$. Then, as $\tilde{Z}$ has no horizontal period and vertical period $-i\frac{\pi}{K}$,
the two vertical periods cancel out and the horizontal period between levels $\bar{u}_0$ and $u_0$ remains.

\underline{\textbf{Case 3}}. The function~$H^{u_0}$ can be chosen to be constant equal to 1 inside $\Cs_{\bs,\ws}^{u_0}$,
and 0 outside. 
Note that for this particular choice of $H^{u_0}$, one sees immediately
that contributions of pieces of $\gamma_{\bs,\ws}^{u_0}$ outside of
$\Cs_{\bs,\ws}^{u_0}$ is zero, and that the
part of $\gamma_{\bs,\ws}^{u_0}$ which is inside can be deformed to become very
close to $\Cs_{\bs,\ws}^{u_0}$. Therefore, the
expressions~\eqref{eq:coeff_Kinv_u0} and \eqref{eq:coeff_Kinv_u0_alt} in this
case are trivially identical.

\medskip

Let us conclude this section with one last remark.
The short proof of Theorem~\ref{thm:K_inverse_family} given in the previous
section, as in the original work of Kenyon~\cite{Kenyon:crit}, does not
explain where this integral formula for $\A^{u_0}$ comes from. The connection to the usual expression obtained by Fourier
transform in the periodic case is explained in
Section~\ref{sec:Gibbs_measures}.

\section{The periodic case}
\label{sec:periodic}

This section deals with the special case where the bipartite planar graph~$\Gs$ is~$\ZZ^2$-periodic.
We start in Section~\ref{sec:tt-per} by explaining the additional features of train-tracks and
half-angle maps in the periodic case. In Section~\ref{sec:Kast-per}, we determine for which half-angle maps
the corresponding elliptic Kasteleyn operator~$\K[t]$ is~$\ZZ^2$-periodic. In Section~\ref{sec:KOS},
we recall standard tools used in the study of the periodic bipartite dimer model, in particular the spectral curve.
In Section~\ref{sec:parametrization}, we use the functions~$g^{(t)}_{\xs,\ys}$ defined in Section~\ref{sec:null_functions} to give an explicit parameterization of the spectral curve for the periodic dimer model corresponding to
the operator~$\K[t]$. Finally,  we describe the set of ergodic Gibbs measures of this model in Section~\ref{sec:Gibbs_measures}, and give an explicit expression for the corresponding slopes in Section~\ref{sec:slope}.

Throughout this section, we fix the parameter~$t$ in~$\RR+\frac{\pi}{2}\tau$, {where $\tau\in i\RR_{>0}$,} and once again omit
the superscript~$(t)$ in the notation of $\K[t],\A[t]$ and $\g[t]$.

\subsection{Train-tracks and monotone angle maps in the periodic case}
\label{sec:tt-per}

In the whole of this section, we assume that the bipartite planar graph~$\Gs$
is~$\ZZ^2$-periodic, \emph{i.e.}, that~$\ZZ^2$ acts freely on colored vertices,
edges and faces by translation.
A basis of~$\ZZ^2$ has been chosen, allowing to identify a \emph{horizontal}
direction (along the first vector~$(1,0)$ of the basis) and a \emph{vertical}
direction (along the second vector~$(0,1)$).
The action of~$\ZZ^2$ is denoted additively: for example, if~$\xs$ is a vertex
and~$(m,n)$ belongs to~$\ZZ^2$, then $\xs+(m,n)$ is the copy of~$\xs$ obtained by
translating it~$m$ times along the horizontal direction and~$n$ times along the vertical
one.

The graph~$\Gs$ has a natural toroidal exhaustion~$(\Gs_n)_{n\geq 1}$, where~$\Gs_n:=\Gs/n\ZZ^2$.
The graph~$\Gs_1$ is a bipartite graph on the torus known as the \emph{fundamental domain}.
We use similar notation for the toroidal exhaustions of the dual graph~$\Gs^*$, of the quad-graph~$\GR$, and of the train-tracks~$\T$.

Fix a face~$\fs$ of~$\Gs$ and draw two simple dual paths in the plane,
denoted by~$\gamma_x$ and~$\gamma_y$, joining~$\fs$ to~$\fs+(1,0)$ and~$\fs+(0,1)$ respectively,
intersecting only at $\fs$.
They project
onto the torus to two simple closed loops on~$\Gs^*_1$, also denoted by~$\gamma_x$ and~$\gamma_y$, winding around the torus and intersecting only at~$\fs$.
Their homology classes~$[\gamma_x]$ and~$[\gamma_y]$ form a basis of the first
homology group of the torus~$H_1(\TT;\ZZ)$ and allow its identification with~$\ZZ^2$.

Every train-track~$T\in\T$ projects to an oriented closed curve on the torus.
Therefore, the corresponding
homology class~$[T]\in H_1(\TT,\ZZ)$ can be written as~$[T]=h_T[\gamma_x]+v_T[\gamma_y]$,
with~$h_T$ and~$v_T$ coprime integers.
This allows to define a partial cyclic order on~$\T$ by using the natural cyclic order
of coprime elements of~$\ZZ^2$ around the origin. As one easily checks,
this coincides with the partial cyclic order on~$\T$ defined in Section~\ref{sec:tt-def}.
By construction, this cyclic order induces a cyclic order on~$\T_1=\T/\ZZ^2$. 
Note also that two oriented train-tracks~$T,T'\in\T$ are parallel (resp.\@ anti-parallel) as defined
in Section~\ref{sec:tt-def} if and only if~$[T]=[T']$ (resp.\@~$[T]=-[T']$).

Recall that~$X_\Gs$ denotes the set of maps~$\mapalpha\colon\T\to\RR/\pi\ZZ$ that are monotone with respect to the cyclic orders on~$\T$ and~$\RR/\pi\ZZ$, and that map pairs of intersecting or anti-parallel train-tracks to distinct half-angles.
We shall denote by~$X^{\mathit{per}}_\Gs$ the set of~$\ZZ^2$-periodic elements of~$X_\Gs$, \emph{i.e.},
\[
X^\mathit{per}_\Gs=\{\mapalpha\in X_\Gs\,|\,\alpha_{T+(m,n)}=\alpha_T \text{ for all }T\in\T\text{ and }
(m,n)\in\ZZ^2\}\,.
\]

Since disjoint curves on the torus have either identical or opposite homology classes,
this space can be described more concretely as
\[
X^\mathit{per}_\Gs=\{\mapalpha\colon\T_1\to\RR/\pi\ZZ\,|\,\mapalpha\text{ is monotone and } \alpha_T\neq\alpha_{T'}\text{ for }[T]\neq[T']\}\,.
\]

By the results of~\cite{BCdT:immersion}, if~$\Gs$ is minimal, then
any~$\mapalpha\in X^\mathit{per}_\Gs$ defines a~$\ZZ^2$-periodic minimal immersion of~$\Gs$
(see Section~\ref{sec:tt-def} for definition),
and every such immersion is obtained in this way.

Recall that since~$\Gs$ is bipartite, the train-tracks in~$\T$ are consistently
oriented, clockwise around black vertices and counterclockwise around white ones.
Therefore, the sum of all oriented closed curves~$T\in\T_1$ bounds a~$2$-chain in the torus.
In particular, its homology class vanishes, so we have~$\sum_{T\in\T_1}[T]=0$.
As a consequence, the collection of vectors~$([T])_{T\in\T_1}$ in~$\ZZ^2$, ordered cyclically,
and drawn so that the initial point of a vector~$[T]$ is the
end point of the previous vector, defines a convex polygon (up to translations).
Note that since its coordinates are coprime integers, the vector~$[T]$ only meets~$\ZZ^2$ at
its end points. 
This polygon is referred to as the
\emph{geometric Newton polygon} of~$\Gs$~\cite{GK} and denoted by~$N(\Gs)$,
{see Figure~\ref{fig:amoeba} for an example.}
The space~$X^\mathit{per}_\Gs$ can now be described combinatorially as the set of order-preserving
maps from oriented boundary edges of~$N(\Gs)$ to~$\RR/\pi\ZZ$
mapping distinct vectors to distinct angles. 

In~\cite[Theorem~2.5]{GK}, see also~\cite{Gulotta,Postnikov}, Goncharov and Kenyon build on earlier work of Thurston~\cite{Thurston}
(the article appeared in 2017, but the original preprint dates back to 2004) 
 to show that for
any convex envelop~$N$ of a finite set of points in~$\ZZ^2$,
there exists a minimal~$\ZZ^2$-periodic graph~$\Gs$ such that~$N(\Gs)=N$.
Moreover, if~$\Gs$ and~$\Gs'$ are
two minimal graphs such that~$N(\Gs)=N(\Gs')$, then they are related by
elementary local moves called \emph{spider moves} and \emph{shrinking/expanding~$2$-valent vertices},
see Figure~\ref{fig:elementary}. We study the effect of these moves on the operator~$\K$
in Section~\ref{sec:spider}.

\subsection{Periodicity of the Kasteleyn operator}
\label{sec:Kast-per}

From now on, we assume that the graph~$\Gs$ is minimal (and~$\ZZ^2$-periodic).
We further suppose that~$\Gs$ is \emph{non-degenerate}, in the sense that its geometric Newton polygon~$N(\Gs)$ has positive area.
The aim of this section is to understand for which half-angle maps~$\mapalpha\in X^\mathit{per}_\Gs$ the corresponding elliptic Kasteleyn operator~$\K$ defined in Equation~\eqref{def:Kast_elliptic} is periodic.

Note that the periodicity of~$\Gs$ and of~$\mapalpha$ is not
sufficient to ensure the periodicity of the operator~$\K$.
Indeed, this operator makes use of the~$\RR/\pi\ZZ$-valued discrete Abel map~$\mapd$ defined in Section~\ref{sec:abel_map}
which might have horizontal and vertical \emph{periods}.
More precisely, and using the notation of Section~\ref{sec:tt-per}, we have that for every vertex~$\xs$ of~$\GR$
and~$(m,n)\in\ZZ^2$, the equality
\begin{equation}
\label{eq:per-eta}
\mapd(\xs+(m,n))=\mapd(\xs)+m\sum_{T\in\T_1} \alpha_T v_T-n\sum_{T\in\T_1} \alpha_T h_T
\end{equation}
holds in~$\RR/\pi\ZZ$. {Note that we have $(v_T,-h_T)$ because when moving in the horizontal (resp. vertical) direction, a cycle with homology class $(h_T,v_T)$ is intersected algebraically~$v_T$ (resp.~$-h_T$) times.}

Motivated by this observation, consider the map
\[
\varphi\colon X_\Gs^\mathit{per}\longrightarrow\RR^2
\]
defined as follows.
Let us enumerate by~$T_1,\dots,T_r$ the elements of~$\T_1$ respecting the
cyclic order, and let~$P_1,\dots,P_r$ denote the integer points on the boundary of~$N(\Gs)$
numbered so that~$P_{j+1}-P_j=[T_j]$ (where~$P_{r+1}$ stands for~$P_1$).
Given a half-angle map~$\mapalpha\in X^\mathit{per}_\Gs$, let us write~$\alpha_j:=\alpha_{T_j}$ and
denote by~$\widetilde{\alpha_j-\alpha_{j-1}}$ the unique lift in~$[0,\pi)$ of~$\alpha_j-\alpha_{j-1}\in\RR/\pi\ZZ$ (where~$\alpha_0$ stands for~$\alpha_r$).
For~$\mapalpha\in X_\Gs^\mathit{per}$, set
\begin{equation}
  \varphi(\mapalpha)=\sum_{j=1}^r \frac{\widetilde{\alpha_j-\alpha_{j-1}}}{\pi}\cdot P_j\in\RR^2\,.
  \label{equ:map_vphi}
\end{equation}
Recall that the geometric Newton polygon is defined up to translation of an element of $\ZZ^2$. When defining $\vphi$ above, we are fixing the integer boundary points $P_1,\dots,P_r$ of $N(\Gs)$, thus an anchoring. The following proposition nevertheless holds for all choices of anchoring, and answers the problem raised at the beginning of the section.
{We refer the reader to Figure~\ref{fig:amoeba} for an illustrated example.}

\begin{prop}
\label{prop:angles_perio}
The image of the map~$\varphi\colon X_\Gs^\mathit{per}\to\RR^2$ is equal to the interior of
the geometric Newton polygon~$N(\Gs)$ of~$\Gs$. Moreover, a periodic half-angle
map~$\mapalpha\in X_\Gs^\mathit{per}$ induces a periodic 
elliptic Kasteleyn operator~$\K$ if and only if~$\varphi(\mapalpha)$ lies in~$\ZZ^2$.
\end{prop}

\begin{proof}
Let us fix~$\mapalpha\in X_\Gs^\mathit{per}$ and consider its image by~$\varphi$.
First observe that since~$\mapalpha$ is monotone, we have~$\sum_{j=1}^r\widetilde{\alpha_j-\alpha_{j-1}}=\pi$. Therefore,
$\varphi(\mapalpha)$ is a convex combination of the vertices~$P_1,\dots,P_r$,
and hence an element of the convex hull~$N(\Gs)$ of these vertices.

To analyse~$\varphi(X_\Gs^\mathit{per})$ more precisely, let us write~$\overline{X}_\Gs^\mathit{per}$ for
the set of monotone half-angle maps~$\mapalpha\colon\T_1\to\RR/\pi\ZZ$ ($\overline{X}_\Gs^\mathit{per}$ is the set $X_\Gs^{\mathit{per}}$ without the condition that train-tracks with different homology classes need to have distinct half-angles), and denote by~$\Delta=\{\mapbeta=(\beta_j)_j\in[0,1]^r\,|\,\sum_{j=1}^r\beta_j=1\}$ the standard simplex of dimension~$r-1$. Observe that~$\varphi$ can be
described as the restriction to~$X_\Gs^\mathit{per}$ of the composition
\[
\overline{X}_\Gs^\mathit{per}\stackrel{\delta}{\longrightarrow}\Delta\stackrel{p}{\longrightarrow}N(\Gs)\,,
\]
with~$\delta(\mapalpha)=(\frac{\widetilde{\alpha_j-\alpha_{j-1}}}{\pi})_j$ and~$p(\mapbeta)=\sum_j\beta_j P_j$.
This composition~$p\circ\delta\colon\overline{X}_\Gs^\mathit{per}\to N(\Gs)$ is clearly surjective, but we now need to
understand how the condition~$\alpha_T\neq\alpha_{T'}$ for~$[T]\neq[T']$ defining the space~$X_\Gs^\mathit{per}\subset\overline{X}_\Gs^\mathit{per}$ affects the image of~$p\circ\delta$ inside~$N(\Gs)$.

Since~$p$ is an affine surjective map, any point in the interior of~$N(\Gs)$
is the image under~$p$ of an element of the interior of~$\Delta$, \emph{i.e.},\@ an element~$\mapbeta\in\Delta$ with no
vanishing coordinate. Therefore, we have
\[
\delta^{-1}(p^{-1}(\mathrm{int}\,N(\Gs)))\subset \delta^{-1}(\mathrm{int}\,\Delta)=\{\mapalpha\in \overline{X}_\Gs^\mathit{per}\,|\,\mapalpha\text{ injective}\}\subset X_\Gs^\mathit{per}\,,
\]
thus checking the inclusion of the interior of~$N(\Gs)$ into~$\varphi(X_\Gs^\mathit{per})$.

To prove the opposite inclusion, consider an arbitrary element~$x$ of~$N(\Gs)\setminus\mathrm{int}\,N(\Gs)$,
and let us write~$F$ for the biggest face of~$N(\Gs)$ containing~$x$ in its interior. (Concretely,~$F=x$ if~$x$ is
a vertex of~$N(\Gs)$, and~$F$ is the boundary edge of~$N(\Gs)$ containing~$x$ otherwise.)
By definition, we have~$p^{-1}(x)=\{\mapbeta\in\Delta\,|\,\sum_j\beta_jP_j=x\}$.
Fix a reference frame for~$\RR^2$ with origin at~$x$ and first coordinate axis orthogonal to~$F$.
Then, the first coordinate of the equation~$\sum_j\beta_jP_j=x$ leads to~$\beta_j=0$ for all~$j$ such that~$P_j$ does
not belong to~$F$. Since~$N(\Gs)$ has positive area, we have~$\beta_j=0$ for some vertex~$P_j$ of~$N(\Gs)$. Such an element of~$\Delta$
can only be realized as~$\delta(\mapalpha)$ with~$\alpha_j=\alpha_{j-1}$. Since~$P_j$ is a vertex of~$N(\Gs)$,
we have~$[T_j]\neq[T_{j-1}]$, so~$\mapalpha$ does not belong to~$X_\Gs^\mathit{per}$.
This shows the inclusion of~$\varphi(X_\Gs^\mathit{per})$ into the interior of~$N(\Gs)$, and thus the equality of these
two sets.

Finally, by~$\pi$-(anti)periodicity of the theta function~$\theta$, the operator~$\K$ is periodic
if and only if the~$\RR/\pi\ZZ$-valued discrete Abel map~$\mapd$ is
{periodic}. By Equation~\eqref{eq:per-eta}, this
is the case if and only if
\[
\sum_{T\in\T_1} \alpha_T [T]=
\sum_{T\in\T_1} \alpha_T\left(\begin{smallmatrix} h_T \\ v_T \end{smallmatrix}\right)= \left(\begin{smallmatrix} 0 \\ 0 \end{smallmatrix}\right)\in(\RR/\pi\ZZ)^2\,.
\]

Fixing arbitrary lifts~$\widetilde{\alpha}_j\in\RR$ of~$\alpha_j\in\RR/\pi\ZZ$, this is equivalent to requiring that
the following element of~$\RR^2$ belongs to~$\ZZ^2$:
\begin{equation}
\label{equ:diff}
\frac{1}{\pi}\sum_{j=1}^r \widetilde{\alpha}_j [T_j]=\frac{1}{\pi}\sum_{j=1}^r \widetilde{\alpha}_j (P_{j+1}-P_j)
=-\sum_{j=1}^r\frac{\widetilde{\alpha}_j-\widetilde{\alpha}_{j-1}}{\pi}P_j=\sum_{j=1}^rn_j P_j - \varphi(\mapalpha)\,,
\end{equation}
with~$n_j=\frac{\widetilde{\alpha_j-\alpha_{j-1}}}{\pi}-\frac{\widetilde{\alpha}_j-\widetilde{\alpha}_{j-1}}{\pi}$. Since~$n_j$ is an integer and~$P_j$ an element of~$\ZZ^2$ for all~$j$, this is equivalent to requiring that~$\varphi(\mapalpha)$ belongs to~$\ZZ^2$. This concludes the proof.
\end{proof}

In the remainder of this section, we suppose that the minimal graph~$\Gs$ is
endowed with~$\mapalpha\in X_\Gs^\mathit{per}$ so that
the corresponding elliptic Kasteleyn operator~$\K$ is periodic, \emph{i.e.},
so that~$\varphi(\mapalpha)$ is an interior lattice point of~$N(\Gs)$.

\begin{rem}\label{rem:NG} \leavevmode
\begin{enumerate}
\item Some minimal periodic graphs have a too small geometric Newton polygon to admit
  such an integer point~$\varphi(\mapalpha)$ in their interior. This is the case for
  the square and hexagonal lattices with their smallest
  fundamental domain composed of one vertex of each color.
    For these graphs, the rest of the discussion
  in this section is void.
\item 
In Section~\ref{sec:slope} we use the following version of
Equation~\eqref{equ:diff}.
Fix a lift
$(\widetilde{\alpha}_j)$
of the half-angles $(\alpha_j)$ in $\RR/\pi\ZZ$ such that,
$\widetilde{\alpha}_1\le\widetilde{\alpha}_2\le\cdots\le \widetilde{\alpha}_r
<\widetilde{\alpha}_1+\pi$,
and consider the element
$-\frac{1}{\pi}\sum_{T\in\T_1}\widetilde{\alpha}_T\binom{h_T}{v_T}$.
Then it is equal to the difference of~$\varphi(\mapalpha)$ with $P_1$,
as in this case, all $n_j$'s are equal to 0
except $n_1$ which is equal to 1.
\end{enumerate}
\end{rem}

\subsection{Characteristic polynomial, spectral curve, amoeba}\label{sec:charact_polynom}
\label{sec:KOS}

We now recall some key tools used in studying the dimer model
on a~$\mathbb{Z}^2$-periodic, planar, bipartite graph~$\Gs$
  with periodic weights, see for example~\cite{KOS}.
For $(z,w)\in(\CC^*)^2$, a function $f$ is said to be \emph{$(z,w)$-quasiperiodic}, if
\[
\forall \xs\in \Vs,\, \forall\,(m,n)\in\ZZ^2,\quad f(\xs+(m,n))=z^{m}w^{n}f(\xs).
\]
We let $\CC^\Vs_{(z,w)}$ denote the space of $(z,w)$-quasiperiodic
functions. Such functions are completely determined by their value in $\Vs_1$.
By a slight abuse of notation, we will identify
  a $(z,w)$-quasiperiodic function with
its restriction to $\Vs_1$, where 
this identification depends on the choice of $\gamma_x$ and $\gamma_y$.
With this convention in mind, a natural basis for
$\CC^\Vs_{(z,w)}$ is given by $(\delta_\xs(z,w))_{\xs\in \Vs_1}$. Similarly,
we let $\CC^\Bs_{(z,w)}$ and $\CC^\Ws_{(z,w)}$ be the set of 
$(z,w)$-quasiperiodic functions defined on black vertices, and on white
vertices respectively.

The periodic operator $\Ks$ maps the vector space
$\CC^\Bs_{(z,w)}$ into $\CC^\Ws_{(z,w)}$ and we let $\Ks(z,w)$ be the matrix
of the restriction of $\Ks$ to these spaces written in their natural respective
bases. Alternatively,
the matrix $\Ks(z,w)$ is the matrix of the Kasteleyn operator of the fundamental domain $\Gs_1$ where edge weights
are multiplied
by $w^{-1}$ or $w$, resp. $z$ or
$z^{-1}$ each time the corresponding edge oriented from the white to the black vertex
crosses the curve $\gamma_x$, resp. $\gamma_y$, left-to-right or right-to-left.
The \emph{characteristic polynomial} $P(z,w)$ is the determinant of
the matrix $\Ks(z,w)$.
The \emph{Newton polygon} of $P$, denoted by $N(P)$ is the convex hull of
lattice points $(i,j)$ such that $z^i w^j$ arises as a non-zero monomial in $P$.
By \cite[Theorem~3.12]{GK}, $N(P)$ is a lattice translate of $N(\Gs)$.

The \emph{spectral curve} $\C$
is the zero locus of the characteristic polynomial:
\[
  \C=\{(z,w)\in(\CC^*)^2:\, P(z,w)=0\}.
\]
In other words, it corresponds to the values of $z$ and $w$ for which we can
find a non-zero $(z,w)$-quasiperiodic function $f$ on black vertices such that
$\Ks f=0$.

The \emph{amoeba} $\Ascr$ of the curve $\C$ is the image
of~$\C$ through the map $(z,w)\!\mapsto\! (-\!\log\!|w|,\log\!|z|)$,
see Figure~\ref{fig:amoeba} for an example.
Note that the present definition of the amoeba differs from that
of~\cite{Gelfand-Kapranov-Zelevinsky,KOS} by a rotation by 90 degrees. This is handy when describing the phase diagram, see Section~\ref{sec:Gibbs_measures}.
By~\cite{KO:Harnack,KOS}, we know that the spectral curve $\C$ is a \emph{Harnack
curve}, which is
equivalent to saying that the map from the curve to its amoeba is at most
2-to-1~\cite{Mikhalkin1}.
(To be precise, such curves are referred to as \emph{simple
Harnack curves} in real algebraic geometry.)
The \emph{real locus} of the curve $\C$ consists of the set of points that are invariant under complex conjugation:
\[
  \{(z,w)\in\C:\,(z,w)=(\bar{z},\bar{w})\}=\{(x,y)\in(\RR^*)^2:\,P(x,y)=0\}
\]
deprived from its isolated singularities, which are the only singularities a
Harnack curve can admit.

\begin{figure}
  \centering
  \def\svgwidth{55mm}
  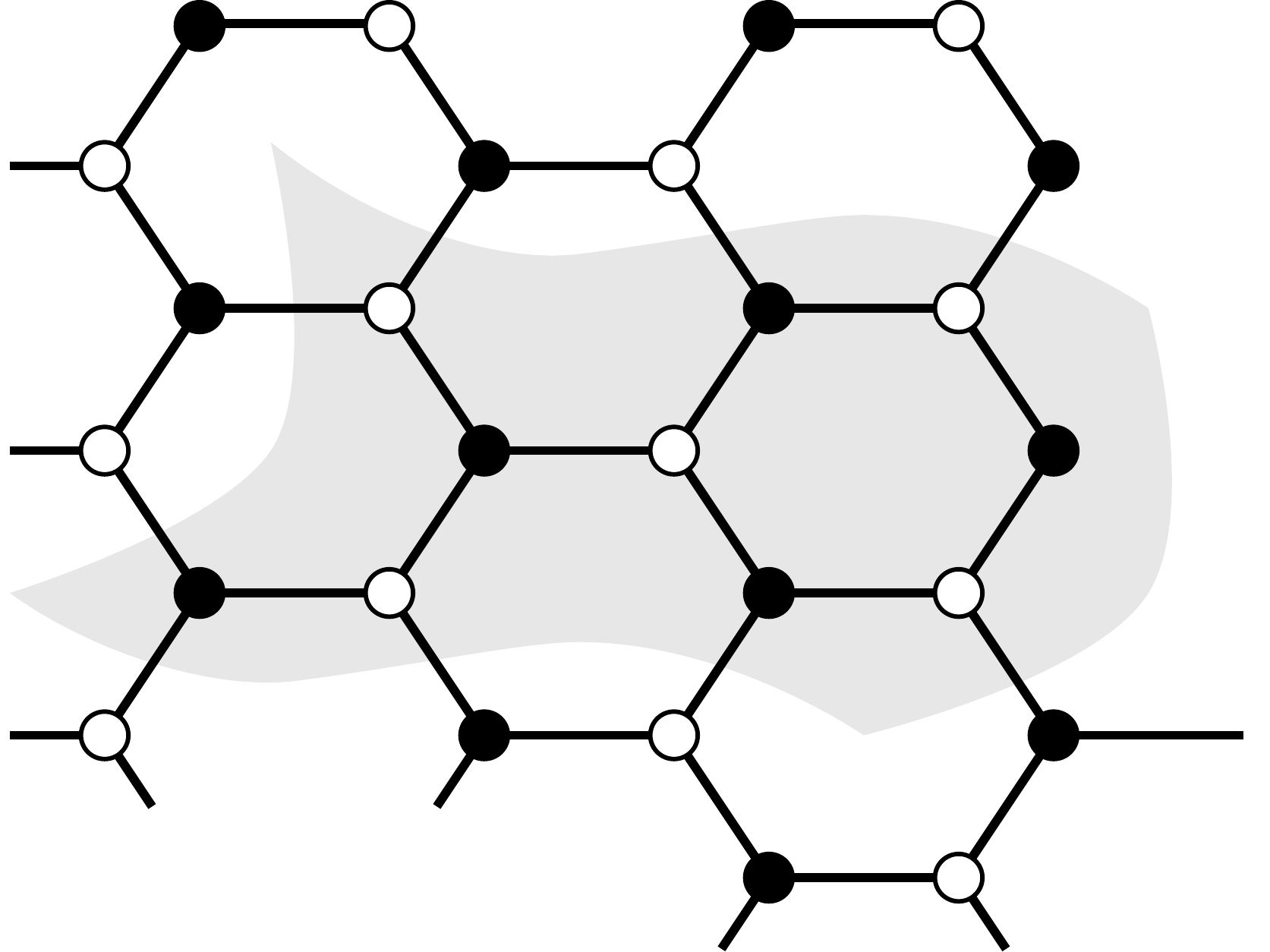
  \hspace{2mm}
  \def\svgwidth{35mm}
  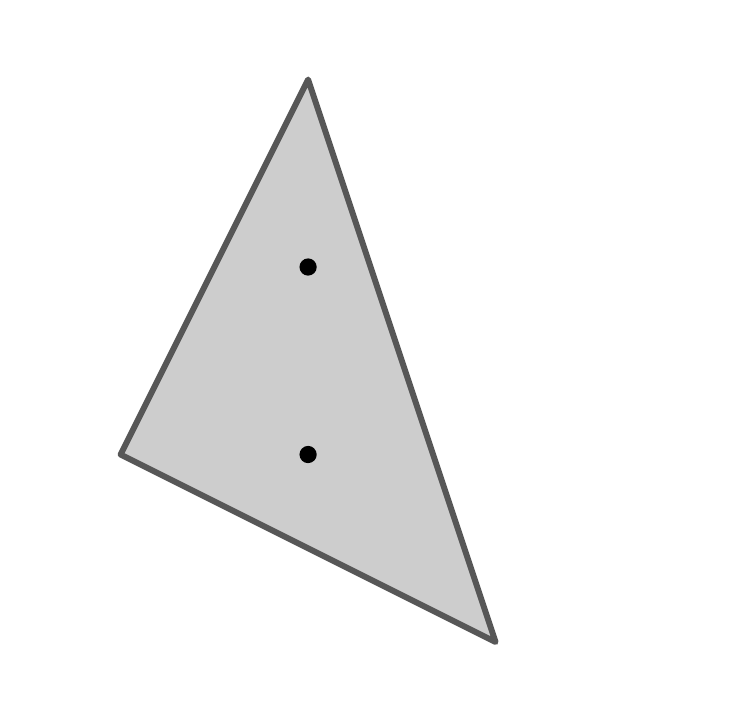
  \def\svgwidth{45mm}
  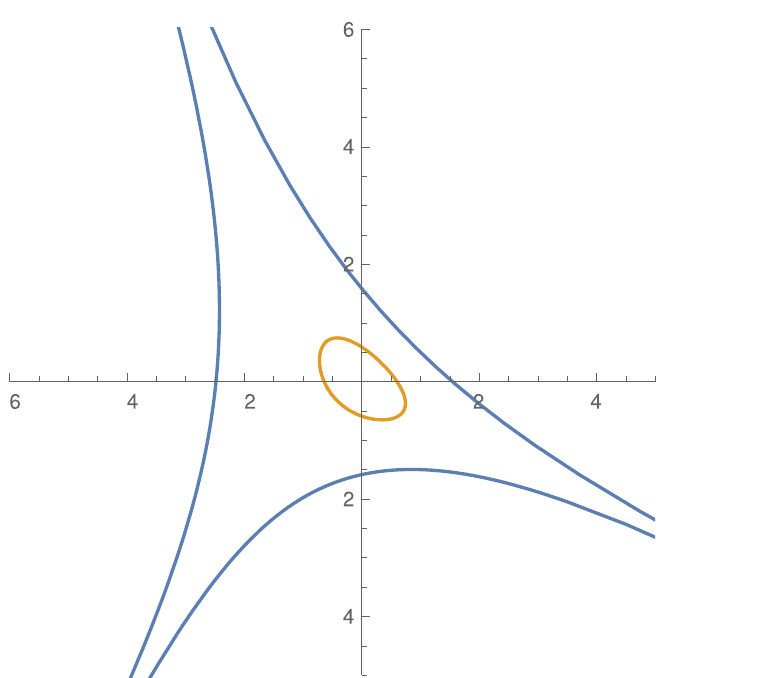
  \caption{Left: the hexagonal lattice with a shaded fundamental domain
    as an example of periodic minimal graph $\Gs$, with
    the trace of the three train-tracks of $\Gs_1$ (green, blue, red) of
    respective homology classes $(-1,3)$, $(-1,-2)$ and $(2,-1)$.
    Middle: the corresponding Newton polygon, whose boundary vectors are given
    by the homology classes of the train-tracks;
    {by Proposition~\ref{prop:angles_perio}, the interior of~$N(\Gs)$ is the image of the map~$\varphi\colon X_\Gs^\mathit{per}\to\RR^2$.}
    Right: amoeba of the associated
    spectral curve when the
half-angles are chosen to be $-\frac{\pi}{5}$, $0$, $\frac{2\pi}{5}$ respectively, so
that the point $\varphi(\mapalpha)$ is the topmost inner point of $N(\Gs)$.
Because of our convention, the asymptotes of the tentacles are parallel (instead of normal classically)
to the corresponding sides of~$N(\Gs)$.
}
\label{fig:amoeba}
\end{figure}

Note that the matrix~$\Ks(z,w)$ and the characteristic polynomial~$P(z,w)$ both
depend on the
particular gauge choice for the edge weights
and on the explicit choices of $\gamma_x$ and $\gamma_y$.
Also, the Newton polygon~$N(P)$ undergoes a translation when these curves
are deformed, \emph{i.e.}, modified within fixed homology classes.
On the other hand,
the spectral curve (and its associated amoeba) only depend on the face weights of the periodic dimer model
(with fixed homology classes for~$\gamma_x$ and~$\gamma_y$).
Replacing this positive basis of~$H_1(\TT;\ZZ)$ by another one amounts to transforming the amoeba by the linear action
of the element of~$\mathit{SL}_2(\ZZ)$ corresponding to this base change.

\subsection{Explicit parameterization of the spectral curve}\label{sec:parametrization}

Throughout this section, we assume that the~$\ZZ^2$-periodic minimal graph~$\Gs$ is
endowed with a half-angle map~$\mapalpha\in X^\mathit{per}_\Gs$ such that the Kasteleyn operator $\Ks$ is periodic. This ensures that the function $g$
{defined in Section~\ref{sec:null_functions}}
satisfies the
equality
\begin{equation}
\label{equ:g-per}
g_{\xs+(m,n),\ys+(m,n)}(u)=g_{\xs,\ys}(u),
\end{equation}
for all~$u\in\CC$, vertices~$\xs,\ys\in\Vs$ and integers~$m,n\in\ZZ$.

Fix a complex number~$u$, and a base vertex~$\xs_0\in \Vs$.
Because of its product structure, the function~$\xs\mapsto g_{\xs,\xs_0}(u)$
is~$(z,w)$-quasiperiodic, with~$(z,w)=(z(u),w(u))$ given by
\begin{equation*}
  z(u)=g_{\xs_0+(1,0),\xs_0}(u),\quad w(u)=g_{\xs_0+(0,1),\xs_0}(u).
\end{equation*}
These quantities are easily seen not to depend on the choice of~$\xs_0$.
Note also that, since~$\Ks g_{\cdot,\xs_0}(u)$ vanishes by Proposition~\ref{prop:ker},
the complex pair~$(z(u),w(u))$ belongs to the spectral curve~$\C$ for all~$u$, a fact already noted in~\cite{Fock}.

The quantities $z(u)$ and $w(u)$ can be expressed explicitly in terms of the 
half-angles and homology
classes of train-tracks of~$\Gs_1$, as follows.
Using its definition, 
compute~$z(u)=g_{\xs_0+(1,0),\xs_0}(u)$
as a product of local
contributions along an edge path of $\GR_1$
winding once horizontally from right to left.
Terms of the form~$\theta(u+t+\mapd(\ws))$ or~$\theta(u-t-\mapd(\bs))$
coming from contributions of edges arriving to or leaving from white vertices~$\ws$
or black vertices~$\bs$ along this path cancel out,
leaving only a factor
of the form
\[
\frac{\theta(u+t+\mapd(\xs_0))}{\theta(u+t+\mapd(\xs_0+(1,0)))}=
\frac{\theta(u+t+\mapd(\xs_0))}{\theta(u+t+\mapd(\xs_0)+\sum_T\alpha_T v_T)}=(-1)^{\frac{1}{\pi}\sum_T\alpha_T v_T}\,,
\]

where we have used Equation~\eqref{eq:per-eta}, the fact that $\sum_{T\in\T_1}v_T
\alpha_T\equiv 0 \mod \pi$, and assumed without loss of generality that~$\xs_0$ is a white vertex.
The remaining factors
$\theta(u-\alpha)$ can be grouped together according to the train-tracks in
$\T_1$ they are associated to. For a train-track $T\in\T_1$, the exponent of
$\theta(u-\alpha_T)$ is the algebraic number of times a copy of $T$ crosses
the path, which is, with our convention, exactly
minus the vertical component of its homology class $[T]$.
One can reason similarly for $w(u)$, giving the following
expressions:
{\small \begin{equation}
  z(u)=(-1)^{\frac{1}{\pi}\sum_T\alpha_T v_T}\prod_{T\in \T_1} \theta(u-\alpha_T)^{-v_T},\quad
  w(u)=(-1)^{\frac{1}{\pi}\sum_T\alpha_T h_T}\prod_{T\in\T_1} \theta(u-\alpha_T)^{h_T}\,.
  \label{eq:expr_zu_wu}
\end{equation}}

By Remark~\ref{rem:g_wb}, the functions $z(u)$ and $w(u)$ are meromorphic functions of $u$ on 
$\TT(q)$ and they are well defined when the half-angle map $\mapalpha$ takes values in $\RR/\pi\ZZ$. Moreover, since 
the $\alpha_T$'s are real, they commute with complex conjugation:
\begin{equation*}
  \forall\ u\in\TT(q),\quad z(\bar{u})=\overline{z(u)}
  \quad \text{and}\quad
  w(\bar{u})=\overline{w(u)}.
\end{equation*}

Define the following map $\psi$:
\begin{align*}
  \psi\colon \TT(q)&\rightarrow \C\\
  u&\mapsto \psi(u)=(z(u),w(u)).
\end{align*}

Then we have the following result, which is illustrated in Figure~\ref{fig:psi}.

\begin{prop}~\label{prop:param_curve}
The map~$\psi$ is an explicit birational parameterization of the spectral curve~$\C$, implying that this Harnack curve has geometric genus~$1$.
Moreover, the real locus of~$\C$ is the image under~$\psi$ of the set~$\RR/\pi\ZZ \times \{0,\frac{\pi}{2}\tau\}$, and the
connected component with ordinate~$\frac{\pi}{2}\tau$ is bounded
(\emph{i.e.},\@ an \emph{oval}) while the other one is not.
\end{prop}
\begin{proof}
  The map $\psi$ is meromorphic so it parameterizes (an open set of) an
  irreducible component of $\C$. But since $\C$ is Harnack~\cite{KOS}, it is in
  particular irreducible. Therefore, it is a parameterization of the
  whole spectral curve. Commutation of $\psi$ with complex conjugation implies that $\psi$ maps 
  the real components $C_0$ and $C_1$ of $\TT(q)$ to those of $\C$. Since $z(u)$ and $w(u)$
  have zeros and poles on $C_0=\RR/\pi\ZZ$, this real component of $\TT(q)$ is mapped to the
  one corresponding to the unbounded component of the real locus of $\C$,
  {and the remaining component~$C_1$ to the oval(s) of~$\C$ or to its isolated real nodes.
  This latter case is impossible, since it would imply that the holomorphic maps~$z$ and~$w$ 
  are constant along~$C_1$, and hence constant.
  Finally, since~$\mapalpha$ belongs to~${X}^\mathit{per}_\Gs$,
  the cyclic ordering of~$\mapalpha(\T)\subset C_0$ coincides with the cyclic ordering of the tentacles of~$\C$.
  We are now in the setting of~\cite[Theorem~10]{Erwan}, which implies that~$\psi\colon\TT(q)\to\C$ is a
  birational parameterization of the spectral curve~$\C$.
 In particular, the geometric genus of $\C$ is equal to 1.}
\qedhere
\end{proof}

\begin{figure}[htb]
    \centering
    \hspace{0.1cm}
\begin{overpic}[width=0.65\textwidth]{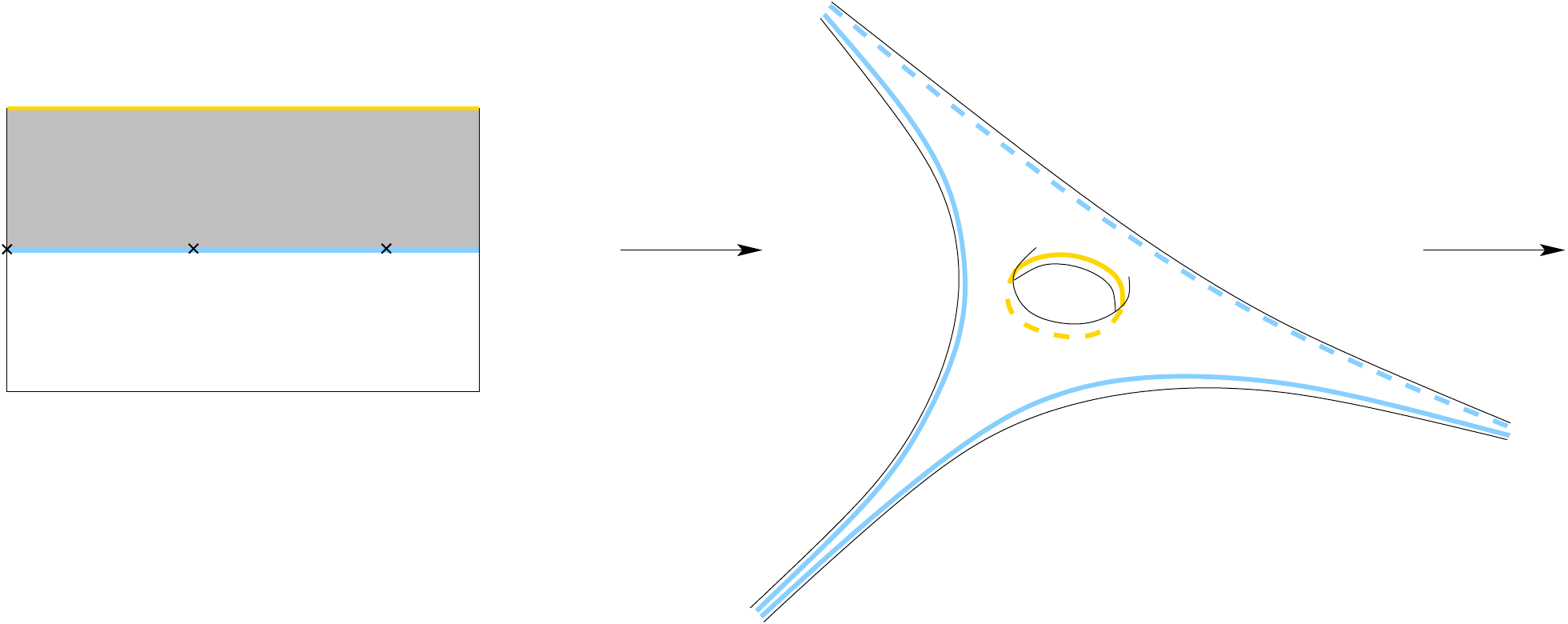}
    \put(-7,14.5){\scriptsize $-\frac{\pi}{2}\tau$}
    \put(-5,32){\scriptsize $\frac{\pi}{2}\tau$}
    \put(5,27){\scriptsize $D$}
    \put(43,26){$\psi$}
    \put(32,23){\textcolor{cyan}{\scriptsize $C_0$}}
    \put(32,32){\textcolor{orange}{\scriptsize $C_1$}}
    \put(13,6){$\TT(q)$}
    \put(72,5){$\C$}
\end{overpic}
    \hspace{0.2cm}
\begin{overpic}[width=0.3\textwidth]{amoeba_simple.pdf}
\put(55,5){$\Ascr$}
\end{overpic}
    \caption{The parameterization~$\psi\colon\TT(q)\to\C$ of the spectral curve, followed by projection~$\C\to\Ascr$ onto the amoeba. {Topologically, the curve~$\C$ is a punctured torus, and the dashed lines represent the parts of~$\psi(C_0\cup C_1)$ that lie behind the
      visible portion of~$\C$.}}
    \label{fig:psi}
  \end{figure}

In the framework of the correspondence between algebraic curves with marked
points and minimal dimer models modulo gauge transformations studied by
Goncharov and Kenyon~\cite{GK}, the converse question has been answered by
Fock~\cite{Fock} for arbitrary smooth, complex algebraic curves without hypothesis on
positivity: given a smooth curve~$\C$, he uses Riemann theta functions to construct a dimer model
whose spectral curve is~$\C$.
We can now give a modified version of his proof in the situation where the
algebraic curve has geometric genus 1 and is Harnack. 

\begin{thm}\label{thm:Harnack}
  Let $\C$ be a Harnack curve in $(\mathbb{C}^*)^2$ having geometric genus 1. Up
  to a scale change $(z,w)\mapsto (\lambda z, \mu w)$, with
  $\lambda,\mu\in\RR^*$, which in particular
  brings the origin inside the hole of the amoeba of the
  rescaled curve, there exists a $\ZZ^2$-periodic minimal graph $\Gs$ and a half-angle map $\mapalpha\in{X}^\mathit{per}_\Gs$ such that, for every $t$ in $\RR+\frac{\pi}{2}\tau$, Fock's elliptic Kasteleyn operator is periodic, and such that the spectral curve of the corresponding dimer model is $\C$.
  This family of Kasteleyn operators, indexed by $t\in \RR/\pi\ZZ+\frac{\pi}{2}\tau$, is in bijection with the
  bounded real component of $\C$.
\end{thm}

\begin{proof}
  The curve $\C$ has geometric genus 1, with two real components, otherwise stated it is \emph{maximal}. As a consequence, there exists
 a birational map $\psi$ from a rectangular torus to $\C$. Up to a global
  scaling, we can assume that this torus is one of the tori~$\TT(q)$, for some
  value of $q\in(0,1)$, thus fixing the parameter of the theta function.
  Let $\TT(q)\ni u\mapsto\psi(u)=(z(u),w(u))\in \C$ be such a map. The projections
  on the first and second coordinates are elliptic functions on $\TT(q)$. They
  thus have the form:
  \begin{equation*}
    z(u)=\lambda \prod_{j=1}^r \theta(u-\alpha_j)^{-b_j},
    \quad
    w(u)=\mu \prod_{j=1}^r \theta(u-\alpha_j)^{a_j},\quad
  \end{equation*}
  where, for all $j\in\{1,\dots,r\}$, $a_j,b_j$ are coprime integers, and $\alpha_j$ belongs to~$\mathbb{R}/\pi\mathbb{Z}$.
  Since the curve $\C$ is Harnack, it is real algebraic, implying that $\lambda,\mu\in \RR^*$. 
  After a possible scaling, we can assume that $\lambda$ and $\mu$ are equal to $\pm 1$. 
The fact that both projections are elliptic implies that they are periodic functions on the torus $\TT(q)$. Quasiperiodicity of the function $\theta$, see Equation~\eqref{equ:quasi_period_theta}, thus implies that $(\alpha_j), (a_j), (b_j)$ satisfy the relations
  \begin{align*}
 &\sum_{j=1}^r a_j=0,\quad \sum_{j=1}^r b_j=0,\quad\\   
 &\sum_{j=1}^r \alpha_j a_j=0\ \mathrm{mod}\ \pi,\quad   \sum_{j=1}^r \alpha_j b_j=0 \ \mathrm{mod}\ \pi.
  \end{align*}
 By~\cite[Theorem~2.5]{GK} (see also~\cite{Gulotta}), the first two equalities allow to construct a bipartite minimal graph $\Gs_1\subset\TT$ 
 with train-tracks $T_1,\ldots,T_r$ satisfying~$[T_j]=(a_j,b_j)\in H_1(\TT;\ZZ)$ for all~$j$.
Let $(z_0,w_0)$ be a marked point on the bounded real component of $\C$. It
corresponds to a unique value $t\in C_1$ such that $z_0=z(t)$ and $w_0=w(t)$.
By Equation~\eqref{eq:per-eta}, the last two equalities displayed above ensure that Fock's elliptic adjacency operator~$\K$ corresponding
 to~$\Gs_1$, $(\alpha_j)$ and~$t$ is indeed periodic.

Since $\C$ is a Harnack curve, the cyclic order of the
  tentacles of the amoeba (which is related to the cyclic order of the homology
  classes of the train-tracks, forming the boundary of the Newton polygon) coincides with
  the cyclic order of the half-angles~$\alpha_T\in\mathbb{R}/\pi\mathbb{Z}$. 
By Proposition~\ref{prop:kastorient}, this implies that Fock's elliptic adjacency operator $\K$ is Kasteleyn.
By the construction above, the curve~$\C$ exactly consists of the points~$(z,w)$ of~$(\mathbb{C}^2)^*$
where the kernel of~$\Ks(z,w)$ is non-trivial. It is therefore the spectral curve.
\end{proof}

\begin{rem}\leavevmode
  \begin{enumerate}
    \item The construction of $\Gs_1$ from the curve $\C$ is not unique. However, if $\Gs'_1$
      is another minimal graph on the torus satisfying the same constraints, then
      $\Gs_1$ and $\Gs'_1$ are related by a sequence of spider moves and
      shrinking/expanding of 2-valent vertices~\cite[Theorem~2.5]{GK}.
    \item A consequence 
      of~\cite[Theorem~7.3]{GK} combined with Theorem~\ref{thm:Harnack} is the following:
      every periodic Kasteleyn operator on $\Gs$ 
      {with spectral curve~$\C$ is gauge-equivalent to Fock's elliptic
    Kasteleyn operators for some $\mapalpha\in X^\mathit{per}_\Gs$ and some $t\in\RR/\pi\ZZ+\frac{\pi}{2}\tau$.
    Furthermore, two periodic dimer models on the same minimal graph~$\Gs$ coming from 
    the same Harnack curve~$\C$, the same angle map~$\mapalpha$ and
    parameters~$t,t'\in\RR/\pi\ZZ+\frac{\pi}{2}\tau$ are gauge equivalent if and only if~$t$ and~$t'$ coincide.}
    \item
  Consider two genus 1 Harnack curves with the same Newton polygon~$N$.
  If the two 
  ovals correspond to the same point~$P\in\ZZ^2$ in the interior of~$N$, then the dimer model
  provides continuous families of genus 1 Harnack curves interpolating between these
  two given curves. For example, one can first continuously deform the half-angle maps within the
  connected space~$\varphi^{-1}(P)\subset X^\mathit{per}_\Gs$ (recall Section~\ref{sec:Kast-per}), and
  then the parameter~$\tau$. One the other hand, if the two ovals correspond to different interior
  points of~$N$, then such a continuous deformation cannot be performed within the realm of elliptic curves.
  One first needs to shrink the oval of one curve into an isolated
  singular double point, so that the curve becomes rational, then take another singular point and transform it into
  an oval. In the dimer picture, this first step corresponds to taking a
  limit as~$|\tau|$ tends to infinity,
  so that the weights become trigonometric (see Section~\ref{sec:rational}).
  At that particular point, there is no constraint imposed for the periodicity of
  weights, and there is more freedom to continuously deform the corresponding spectral curves.
  \end{enumerate}
\end{rem}

\subsection{Gibbs measures}\label{sec:Gibbs_measures}

The main result of this section proves explicit expressions, in our setting, for the two parameter family of ergodic Gibbs measures of~\cite{KOS}, which have the remarkable property of only depending on the local geometry of the graph. Before stating our results we recall required facts from~\cite{KOS}.

\paragraph{Classification of ergodic Gibbs measures~\cite{KOS}.}
Consider a~$\ZZ^2$-periodic, bipartite graph~$\Gs$ (not necessarily minimal)
with periodic weights.
Gibbs measures on dimer configurations of $\Gs$
which are invariant and ergodic under the action of~$\ZZ^2$
are characterized by their
\emph{slope} $(s,t)$~\cite{Sheffield},
which is (up to a fixed arbitrary additive
constant), the expected algebraic number of dimers
crossing $\gamma_x$ and $\gamma_y$. Its precise definition is recalled in
Section~\ref{sec:slope}.

These measures are constructed explicitly in~\cite{KOS} as
limits of (unconditioned) Boltzmann measures on $\Gs_n$ with \emph{magnetically
modified weights}. The magnetic field $B=(B_x, B_y)$ appearing in the weight
modification is the Legendre dual 
of the slope $(s,t)$. It is used
to parameterize the set of possible ergodic Gibbs
measures, whose phase diagram in the plane $(B_x, B_y)$
is given by the amoeba
$\Ascr$ of the characteristic polynomial $P$. 
More precisely, the Boltzmann measure $\PP_n^{B}$ on $\Gs_n$ with magnetic field $B$ is obtained by multiplying
the weight of an edge each time it crosses a copy of $\gamma_x$, resp. $\gamma_y$, by $e^{\pm B_x}$, resp. $e^{\pm B_y}$, if the edge has a white vertex on the left of $\gamma_x$, resp. $\gamma_y$. Note that the convention differs from that used in defining $\Ks(z,w)$ with weights $w^{\mp 1},z^{\pm 1}$, by what can be understood as the action of the intersection form.
Note also that the corresponding Kasteleyn matrix has the same face weights (recall definition~\eqref{eq:altproduct}) for all values of $B$, \emph{i.e.}, the corresponding dimer models are gauge equivalent, but the Boltzmann measures on the torus differ, thus exhibiting a different behavior than in the case of finite graphs embedded in the plane. One way of seeing this is that closed 1-forms on the torus are not necessarily
exact, see \emph{e.g.}~\cite[Section 2.3]{KOS}. Then in~\cite[Theorem~4.3]{KOS}, see also~\cite{CKP}, the authors prove that for every value of $B$, taking the weak limit of the Boltzmann measures~$\PP_n^{B}$ gives an ergodic Gibbs measure $\PP^{B}$, such that the probability of occurrence of a
subset of $k$ distinct edges $\{\es_1=\ws_1\bs_1,\ldots,\es_k=\ws_k\bs_k\}$ is explicitly given by
\begin{equation}\label{eq:Gibbs_1}
  \PP^{B} (\es_1,\ldots,\es_k) =
  \Bigl(\prod_{j=1}^k \Ks_{\ws_j,\bs_j}\Bigr)\times
  \det_{1\leq i,j\leq k} \left(\As^B_{\bs_i,\ws_j}\right)\,,
\end{equation}
where $\As^B$ is defined as follows: if~$\ws$ and $\bs$ are in
the same fundamental domain and~$(m,n)$ belongs to~$\ZZ^2$, then
{\small \begin{equation}\label{eq:inverseKOS}
  \As^B_{\bs+(m,n),\ws}
  = \iint_{\TT_B} \Ks(z,w)^{-1}_{\bs,\ws} z^m w^n 
  \frac{dw}{2i\pi w}
  \frac{dz}{2i\pi z}
  = \iint_{\TT_B} \frac{Q(z,w)_{\bs,\ws}}{P(z,w)} z^m w^n 
  \frac{dw}{2i\pi w}
  \frac{dz}{2i\pi z}\,,
\end{equation}}
where $Q(z,w)$ is the adjugate matrix of $\Ks(z,w)$, $P(z,w)$ is the characteristic polynomial, and $\TT_B=\{(z,w)\in(\CC^2)^*\ ;\
|z|=e^{B_y}, |w|=e^{-B_x}\}$.

\paragraph{Local expressions for ergodic Gibbs measures in the elliptic minimal case.} We consider the case where the graph $\Gs$ is minimal, where Fock's elliptic adjacency operator $\Ks$ is Kasteleyn, and where the half-angle map $\mapalpha\in X_\Gs^{per}$ is such that $\Ks$ is periodic. 
Recall from Section~\ref{sec:inverses} that we have a parameter $u_0\in D$
indexing a family of inverses $(\A^{u_0})_{u_0\in D}$ of the elliptic
Kasteleyn operator $\Ks$. We prove that the set of ergodic Gibbs measures are explicitly written using the operators $(\A^{u_0})_{u_0\in D}$, and that~$\{u_0\in D\}$ gives an alternative parameterization of the set of ergodic Gibbs measures.

\begin{thm}\label{prop:Gibbs_measures}
  For any $B=(B_x,B_y)$ in $\mathbb{R}^2$, there exists a 
   value of the parameter
  $u_0\in D$ such that $\As^B=\As^{u_0}$. When $B$ is inside the amoeba $\Ascr$
  (resp.\@ in the bounded, resp.\@ in an unbounded connected component of its
  complement), $u_0$ is in the interior (resp.\@ on the top boundary, resp.\@ in an
  interval of the bottom boundary) of the domain $D$.
\end{thm}

This theorem combined with the phase diagram of~\cite[Theorem~4.1]{KOS}
yields the following immediate corollary.

\begin{cor}\label{thm:Gibbs_measures}
Consider the dimer model on a~$\ZZ^2$-periodic, bipartite, minimal graph $\Gs$, with a periodic elliptic Kasteleyn operator $\Ks$. Then, the set of ergodic Gibbs measures is the set of measures $(\PP^{u_0})_{u_0\in D}$ whose expression on cylinder sets is explicitly given by, for every $u_0\in D$ and every subset of distinct edges $\{\es_1=\ws_1 \bs_1,\ldots,\es_k=\ws_k \bs_k\}$ of $\Gs$,
\begin{align}\label{equ:mesures_Gibbs}
\PP^{u_0}(\es_1,\ldots,\es_k)&=
\Bigl(\prod_{j=1}^k
\Ks_{\ws_j,\bs_j}\Bigr)\times 
\det_{1\leq i,j\leq k} \Bigl(\A^{u_0}_{\bs_i,\ws_j}\Bigr),
\end{align}
where 
$\As^{u_0}$ is the inverse operator of
Definition~\ref{def:operator_A}.

The domain $D$ gives an alternative phase diagram of the model: when $u_0$ is on the top boundary of $D$, the dimer model
is gaseous; when $u_0$ is in the interior of the set $D$, the model is liquid; when $u_0$ is a
point corresponding to one of the connected components of the lower boundary of $D$, the model is solid.
\end{cor}

\begin{rem}\leavevmode
\begin{enumerate}
\item
The Gibbs measures $(\PP^{u_0})_{u_0\in D}$ are \emph{local}, a property
inherited from that of the inverse operators $(A^{u_0})_{u_0\in D}$, see Point 1 of Remark~\ref{rem:Au_0}. For example, this means that the probability of occurrence of a subset of edges can be computed using only the geometry of paths in $\GR$ joining vertices of these edges, and that it is actually independent of the choice of paths. This remarkable property cannot be seen from the Fourier type expression~\eqref{eq:inverseKOS}. As an illustration, single-edge probabilities are computed in Section~\ref{sec:Gibbsnonperio}.
\item Such type of expressions were already know in the trigonometric
  case~\cite{Kenyon:crit}, corresponding to genus 0 Harnack curves
  (see Section~\ref{sec:rational}),
  and in two
  specific genus~1 cases~\cite{BdTR2,dT_Mass_Dirac} (see
  Section~\ref{sec:previous_genus1}).
  However, let us emphasize that these papers only
  considered the maximal entropy Gibbs measure, corresponding to the weak limit of the toroidal Boltzmann measures with $(0,0)$ magnetic field, and did not
  handle the \emph{two parameter} family of ergodic Gibbs measures.
\item Such local expressions give the right framework to obtain Gibbs measures in the case of (possibly) non-periodic graphs, see Section~\ref{sec:Gibbsnonperio}. Although periodicity is lost, meaning that there is no associated amoeba $\Ascr$, the phase diagram can still be described by the domain $D$. 
Such expressions are also very handy to derive asymptotics, see Section~\ref{sec:asymptotics}. 
\end{enumerate} 
\end{rem}

\begin{proof}[Proof of Theorem~\ref{prop:Gibbs_measures}]
One way of proving equality between $\A^B$ and $\A^{u_0}$ is to use a uniqueness argument for 
the inverse with given asymptotic growth based on Fourier analysis as in~\cite{Boutillier_deTiliere:iso_perio}.
Instead we here choose to do an explicit computation in the spirit of~\cite[Section 5.5.1]{BdTR1} because there are surprising and interesting simplifications which deserve to be made explicit. Note that there are additional difficulties due to the following facts:
we integrate over tori $\TT_B$ of different sizes; in Lemma~\ref{lem:forme_holom}, we explicitly compute the Jacobian of a change of variable from the spectral curve to $\TT(q)$ instead of the abstract argument 
of~\cite{BdTR1}, which required
a combinatorial control of the Newton polygon, gave the result only up to an
unknown multiplicative constant, and did not generalize to higher genus.

Consider $\bs\in\Bs_1$, $\ws\in\Ws_1$, $(m,n)\in\ZZ^2$, and the coefficient  $\As^B_{\bs+(m,n),\ws}$ of~\eqref{eq:inverseKOS}. Up to a change of basis~$([\gamma_x],[\gamma_y])$ of~$H_1(\TT;\ZZ)$,
and possibly deforming $\gamma_x$ and $\gamma_y$ around vertices,
we can assume without loss of generality that~$n\ge 1$
  and that the lowest degree in $z$ (resp. in $w$) of monomials in $P$ is 0.
Recall however that such a base change has the effect of transforming the amoeba
by a linear transformation in~$\mathit{SL}_2(\ZZ)$.
This has to be kept in mind when defining the path of integration $\Cs^{u_0}$
below. Indeed, it might seem that $\Cs^{u_0}$ does not depend on $\bs+(m,n),\ws$
as it should according to the definition of $\As^B_{\bs+(m,n),\ws}$. It actually
does, since
transforming
the amoeba as above has the effect of moving the path of
integration.

For a fixed $z$ such that $|z|=e^{B_y}$, let us compute the integral over $w$ by residues.
Let us denote by~$\{w_j(z)\}_{j=1}^{d_B}$ the
zeros of $P(z,\cdot)$ in the disk of radius $e^{-B_x}$,
which are simple for almost all $z$ on the circle of radius
$e^{B_y}$.
Then, by the residue theorem,
\[
\int_{|w|=e^{-B_x}} \frac{Q(z,w)_{\bs,\ws}}{P(z,w)}w^{n-1}\frac{dw}{2i\pi}=
\sum_{j=1}^{d_B} \frac{Q(z,w_j(z))_{\bs,\ws}}{\partial_w P(z,w_j(z))}w_j(z)^{n-1},
\]
where $\partial_w$ denotes the partial derivative with respect to the second variable.
Indeed, the possibility of a pole of the integrand at $w=0$ is excluded
by the assumptions on $\gamma_x$ and~$\gamma_y$.

To compute the remaining integral over $z$, we perform the change of variable
from $z$ to $u\in\TT(q)$. The set on which we integrate is
\[
  \bigcup_{j=1}^{d_B}\{(z,w_j(z))\in\C:\, |z|=e^{B_y},\, |w_j(z)|\leq e^{-B_x}\}.
\]
In order to identify the preimage of this set under $\psi$,
it is useful to first look at its projection onto the amoeba $\Ascr$, then to lift it
to the curve $\C$ and to take its preimage by~$\psi$.
On the amoeba $\Ascr$, we are looking at its
intersection with the half-line at ordinate $B_y$, extending to the right of
$B_x$.
This intersection consists of a finite number (possibly
zero) of intervals. All these intervals
have their two extremities on the (compactified) boundary of the amoeba, except maybe one,
denoted by $I$. This happens when $B$ is in
the interior of the amoeba, and $B$ is the extremity of $I$ not on the
boundary.
Using the property that the map from the spectral curve to its amoeba is 2-to-1
on the interior of the amoeba, and that the boundary of the amoeba is the image
of the real locus of the spectral curve, the union of intervals can be lifted
to the spectral curve as a collection of paths: intervals joining two points of
the unbounded component of the amoeba are lifted to trivial loops surrounding
points at infinity of the spectral curve. To complete the picture, we
distinguish three cases depending on the position of~$B$ with respect to the
amoeba:

\underline{\textbf{Case 1: gaseous phase}}. $B$ is in the closure of the bounded connected
component of the complement of the amoeba. Then one of the intervals connects
the two components of the boundary of the amoeba, which lifts in $\C$ to a non-trivial
loop winding ``vertically".

\underline{\textbf{Case 2: liquid phase}}. $B$ is in the interior of the amoeba. The
interval $I$ lifts in $\C$ to a curve joining
$(z_B,w_B)$ to its complex conjugate $(\overline{z_B},\overline{w_B})$, where
$|z_B|=e^{B_y}$ and $|w_B|=e^{-B_x}$.

\underline{\textbf{Case 3: solid phases}}. $B$ is in the closure of one of the unbounded
connected components of the complement of the amoeba, and the
(possibly empty) corresponding collection of paths in $\C$ consists of trivial loops
surrounding points at infinity.

In the three cases, the collection of paths on the spectral curves are lifted
back on $\TT(q)$ by $\psi$, and can be deformed to one single path $\Cs^{u_0}$,
which depends on $\bs+(m,n)$ and~$\ws$,
as
described in Section~\ref{sec:preliminaries_contours}, with $u_0$ in $C_1$,
in the interior of $D$, or in~$C_0$, for Cases 1, 2, 3 respectively.

Performing the change of variable from $z=z(u)$ on the collection of paths
on the spectral curve $\C$ to $u\in\Cs^{u_0}$, we can therefore write
\[
\As^B_{\bs+(m,n),\ws}=\frac{1}{2\pi i}\int_{\Cs^{u_0}}
\frac{Q(z(u),w(u))_{\bs,\ws}}{z(u)w(u)\partial_w P(z(u),w(u))}z(u)^{m}w(u)^{n} z'(u)du,
\]
where $u_0$ is such that $|z(u_0)|=|z_B|=e^{B_y}$ and $|w(u_0)|=|w_B|=e^{-B_x}$.

We now need the following lemma, whose proof is deferred until the end of this
one.

\begin{lem}\label{lem:forme_holom}
  There exists a meromorphic function $f$ on $\TT(q)$ such that:
\begin{enumerate}
 \item 
$\forall\, u\in \TT(q), \forall\,\bs\in \Bs_1,\forall\,\ws\in \Ws_1, \quad
Q(z(u),w(u))_{\bs,\ws}=f(u) g_{\bs,\ws}(u),$
 \item
   $\forall\, u\in\TT(q),\quad \displaystyle\frac{f(u)}{z(u)w(u)\partial_w P(z(u),w(u))}z'(u)=-\theta'(0)$.
\end{enumerate}
\end{lem}

Using Lemma~\ref{lem:forme_holom}, we obtain:
\begin{align*}
  \As^B_{\bs+(m,n),\ws}&=\frac{-\theta'(0)}{2\pi i}\int_{\Cs^{u_0}}
z(u)^{m}w(u)^{n} g_{\bs,\ws}(u)du\\
&=\frac{i\theta'(0)}{2\pi}\int_{\Cs^{u_0}} g_{\bs+(m,n),\ws}(u)
du=\A^{u_0}_{\bs+(m,n),\ws},
\end{align*}
where in the second equality, we used the fact that $g_{\cdot,\ws}(u)$
is~$(z(u),w(u))$-quasiperiodic.
Let us emphasize that connected components of the complement of
$\Ascr$ correspond bijectively to connected components of $C_1\cup
C_0\setminus\{\alpha_T\ ;\ T\in\T\}$ on $\TT(q)$, and coefficients of $\As^B$ (resp. 
$\As^{u_0}$) do not change when $B$ (resp. $u_0$) varies while staying in the same
connected component, due to the nature of path integration of meromorphic
functions.
\end{proof}

We now prove Lemma~\ref{lem:forme_holom}.

\begin{proof}[Proof of Lemma~\ref{lem:forme_holom}]
The existence of a meromorphic function~$f$ on~$\TT(q)$ satisfying
Point~1 follows from a straightforward adaptation of Lemmas 29 and
30 of~\cite{BdTR1}, and is based on the fact that on the spectral curve, the
adjugate matrix $Q(z,w)$ is of rank (at most) 1.
To prove Point 2, we now perform a direct computation,
thus using
a different argument than the one of~\cite{BdTR1}: indeed, the proof of $~\cite{BdTR1}$
uses general facts about
holomorphic differential forms on genus 1 surfaces which do not transfer to our setting in a straightforward way,
and which give the result up to a multiplicative constant only.

Let us first show the following:
\begin{equation}\label{equ:expression_f(u)}
f(u)=\frac{\partial_w P(z(u),w(u))}{\sum_{\ws,\bs}
\partial_w \Ks(z(u),w(u))_{\ws,\bs}g_{\bs,\ws}(u)},
\end{equation}
where $\ws$ and $\bs$ run through all black and white vertices of the fundamental
domain $\Gs_1$, respectively.
Starting from the relation satisfied by the adjugate matrix:
\[
Q(z,w)\Ks(z,w)=P(z,w)\Id,
\]
and differentiating with respect to $w$, one can show, as in the proof
of~\cite[Theorem~4.5]{KOS} (see also \cite[Lemma~1]{Bout:patterns} for more details),
that
\[
\operatorname{tr}\left(Q(z,w) \partial_w \Ks(z,w)\right)=\partial_w P(z,w)
\]
for $(z,w)$ on the spectral curve. Evaluating this for $(z(u),w(u))$ and replacing $Q(z(u),w(u))$ by its expression obtained in Point~1 yields:
\[
f(u)\sum_{\bs,\ws} g_{\bs,\ws}(u) \partial_w \Ks(z(u),w(u))_{\ws,\bs} = \partial_w P(z(u),w(u))\,,
\]
thus showing~\eqref{equ:expression_f(u)}.
In order to establish Point 2, we are thus left with proving that:
\begin{equation}\label{equ:lemme30}
\theta'(0) w(u)\sum_{\ws,\bs} \partial_w\Ks(z(u),w(u))_{\ws,\bs}g_{\bs,\ws}(u)=-\frac{z'}{z}(u).
\end{equation}

Let us start from the left-hand side of\eqref{equ:lemme30} and consider a pair
$(\ws,\bs)$ such that $\ws$ and $\bs$ are not connected in $\Gs_1$ by an edge
crossing $\gamma_x$. Then $\partial_w(\Ks(z(u),w(u))_{\ws,\bs}$ is zero and the
pair does not contribute to the sum. If on the contrary there is an edge $\ws\bs$ crossing~$\gamma_x$. Then,
\begin{equation}
w(u)\partial_w\Ks(z(u),w(u))_{\ws,\bs}g_{\bs,\ws}(u)=
\begin{cases}
\Ks_{\ws,\bs}g_{\bs,\ws}(u) & \text{if $\ws\bs$ has a black vertex on the left}\\
-\Ks_{\ws,\bs}g_{\bs,\ws}(u) & \text{if $\ws\bs$ has a white vertex on the left}.
\end{cases}
\label{equ:lemme30b}
\end{equation}
Now, using the telescopic version of Fay's identity~\eqref{eq:telefay}
as in the proof of Proposition~\ref{prop:ker} and the notation of Figure~\ref{fig:around_rhombus},
we obtain:
\begin{align*}
\theta'(0)\Ks_{\ws,\bs}g_{\bs,\ws}(u)&=\theta'(0)[F^{(t+\mapd(\bs))}(u;\beta)-F^{(t+\mapd(\bs))}(u;\alpha)]\\
&=\Bigl(\frac{\theta'}{\theta}(t+\mapd(\bs)-\beta)-\frac{\theta'}{\theta}(u-\beta)\Bigr)-
\Bigl(\frac{\theta'}{\theta}(t+\mapd(\bs)-\alpha)-\frac{\theta'}{\theta}(u-\alpha)\Bigr)\\
&= \Bigl(\frac{\theta'}{\theta}(t+\mapd(\fs))-\frac{\theta'}{\theta}(t+\mapd(\fs'))\Bigr)+
\Bigl(\frac{\theta'}{\theta}(u-\alpha)-\frac{\theta'}{\theta}(u-\beta)\Bigr),
\end{align*}
When summing these contributions over edges crossing the path $\gamma_x$ with a minus
sign for edges having a white vertex on the left of $\gamma_x$,
the terms~$\frac{\theta'}{\theta}(t+\mapd(\fs))$
and~$\frac{\theta'}{\theta}(t+\mapd(\fs'))$ cancel out in a telescopic way, and we therefore obtain:
{\small \begin{align*}
\theta'(0)w(u)\sum_{\ws,\bs}
\partial_w\Ks(z(u),w(u))_{\ws,\bs}g_{\bs,\ws}(u)&=
\Bigl(\sum_{\substack{\ws\bs\cap \gamma_x\neq\emptyset \\ \text{$\bs$ left of $\gamma_x$}}}
-\sum_{\substack{\ws\bs\cap \gamma_x\neq\emptyset \\ \text{$\ws$ left of $\gamma_x$}}}\Bigr)
\Bigl(\frac{\theta'}{\theta}(u-\alpha)-\frac{\theta'}{\theta}(u-\beta)\Bigr)\\
&=-\frac{d}{du}
\log \Bigl(\prod_{\substack{\ws\bs\cap \gamma_x\neq\emptyset \\ \text{$\bs$ left of $\gamma_x$}}}
\frac{\theta(u-\beta)}{\theta(u-\alpha)}
\prod_{\substack{\ws\bs\cap \gamma_x\neq\emptyset \\ \text{$\ws$ left of $\gamma_x$}}}
\frac{\theta(u-\alpha)}{\theta(u-\beta)}
\Bigr).
\end{align*}}

We then notice that if $\ws$ is on the left of $\gamma_x$,
as illustrated on Figure~\ref{fig:intersec_tt_edge_gammax},
the train-track with
half-angle $\beta$ crosses $\gamma_x$ from bottom to top, whereas the
train-track with half-angle $\alpha$ crosses it in the other direction.
We now group the factors $\theta(u-\cdot)$ according to their corresponding
train-track in $\T_1$. For a fixed $T\in\T_1$, the factor $\theta(u-\alpha_T)$
will thus appear in the product with an exponent equal to $-v_T$. Comparing with
Equation~\eqref{eq:expr_zu_wu}, we obtain that the
product on the right-hand side is $z(u)$, up to a factor~$\pm 1$ which
plays no role when differentiating,
thus ending the proof of Point 2.\qedhere
\end{proof}

\begin{figure}[htb]
  \centering
  \def\svgwidth{6cm}
  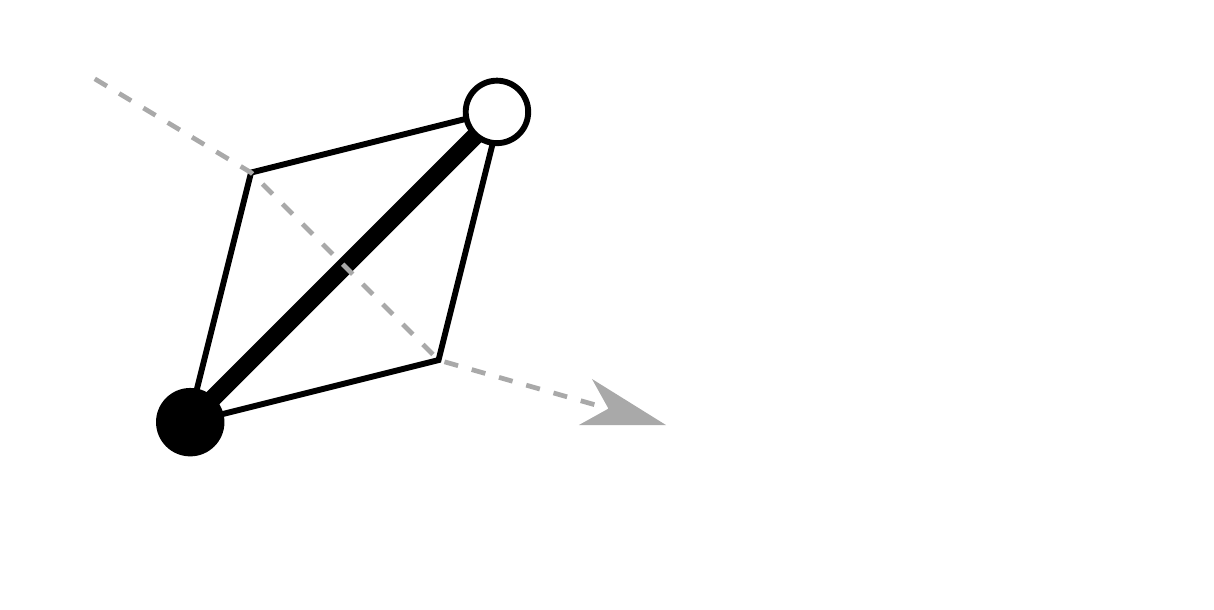
  \caption{Intersection of the two train-tracks corresponding to an edge and
  $\gamma_x$ when the white vertex is on the left of $\gamma_x$.}
  \label{fig:intersec_tt_edge_gammax}
\end{figure}

\subsection{Slope of the Gibbs measures \texorpdfstring{$\mathbb{P}^{u_0}$}{Pu₀}}
\label{sec:slope}

Defining the height function of a dimer configuration $\Ms$ requires to fix a reference dimer configuration
$\Ms_1$.
A natural choice is to fix $u_1\in C_0\setminus\{\alpha_T\ ;\ T\in\T\}$ and consider $\Ms_1$ to be the dimer configuration on which the solid
Gibbs measure $\PP^{u_1}$ is concentrated. Note also that by construction of
$N(\Gs)$, the intervals of $C_0\setminus\{\alpha_T\ ;\ T\in\T\}$  are in bijection
with integer points of $\partial N(\Gs)$.
Let us denote by $P_1$ the one
corresponding to the interval containing $u_1$.

Let $\fs$ and $\fs'$ be two faces of $\Gs$, and $\gamma$ a dual path connecting $\fs'$ and
$\fs$.
Then the height difference (relative to $\Ms_1$) between these two faces
in a dimer configuration $\Ms$ is 
\begin{equation}
  h(\fs)-h(\fs') = \sum_{\es\cap\gamma\neq \emptyset}
  \pm\bigl(\II_{\{\es\in\Ms\}}-\II_{\{\es\in\Ms_1\}}\bigr)
  \label{eq:height_change}
\end{equation}
where the sign $\pm$ is $+$ (resp.\@ $-$) when the white end of $\es$ is on the
left (resp.\@ right) of $\gamma$.

The \emph{slope} $(s^{u_0}, t^{u_0})$  
of the Gibbs measure $\PP^{u_0}$ is the expected horizontal and vertical height
change~\cite{KOS}.
The main result of this section is an explicit expression for the slope of the
ergodic Gibbs
measures $\mathbb{P}^{u_0}$. The content is 
essentially that of Theorem 5.6. of~\cite{KOS} with the additional feature that,
using the explicit parameterization of the spectral curve, we are able to
identify the explicit value of the slope, not only up to a sign and modulo~$\pi$.

Define $\Cs^{u_0}_{u_1}$ to be a contour on $\TT(q)$ from $\bar{u}_0$ to $u_0$,
crossing $C_0$ once at $u_1$, with a positive derivative for the imaginary part
of $u$ at this point.
When $u_0$ is not on $C_0$, this contour can be chosen in such a way that the
imaginary part of $u$ is increasing (and winds once around the circle),
while the real part does not wind around the circle.
When $u_0$ is on~$C_0$, then the contour $\Cs_{u_1}^{u_0}$ can be chosen to be
  a homotopically trivial
  simple closed curve crossing $C_0$ at $u_1$ and $u_0$, oriented so that the
  connected component of~$\TT(q)\setminus\Cs^{u_0}_{u_1}$ to its left (resp.\@ right)
  contains the oriented segment of $C_0$
  from $u_0$ to $u_1$ (resp.\@ from $u_1$ to $u_0$).

\begin{thm}
  \label{thm:slope}
The slope of the Gibbs measure $\PP^{u_0}$ is equal to:
\begin{equation*}
  s^{u_0}=-\frac{1}{2\pi i} \int_{\Cs^{u_0}_{u_1}} \frac{d}{du}(\log  z(u)) du
  ,\quad t^{u_0}=-\frac{1}{2\pi i} \int_{\Cs^{u_0}_{u_1}} \frac{d}{du}(\log w(u)) du,
\end{equation*}
where the path $\Cs^{u_0}_{u_1}$ is defined above.
\end{thm}
\begin{proof}

The quantity $s^{u_0}$ is the expectation under $\PP^{u_0}$
of~\eqref{eq:height_change} with $\gamma=\gamma_x$.
By Corollary~\ref{thm:Gibbs_measures}, 
the expectation under $\PP^{u_0}$ of the term in~\eqref{eq:height_change}
corresponding to an edge $\es=\ws\bs$ is given by:
\begin{multline*}
  \EE^{u_0}[\II_{\{\es\in \Ms\}}-\II_{\{\es\in
  \Ms_1\}}] =\PP^{u_0}(\es) - \II_{\{\es\in\Ms_1\}}=
\PP^{u_0}(\es)-\PP^{u_1}(\es) = \\
  \frac{i\theta'(0)}{2\pi}\int_{\Cs^{u_0}_{\bs,\ws}} \Ks_{\ws\bs}g_{\bs\ws}(u)du
  -  \frac{i\theta'(0)}{2\pi}\int_{\Cs^{u_1}_{\bs,\ws}}
  \Ks_{\ws\bs}g_{\bs\ws}(u)du
   =  \frac{i\theta'(0)}{2\pi}\int_{\Cs^{u_0}_{u_1}}     \Ks_{\ws\bs}g_{\bs\ws}(u)du.
\end{multline*}

Summing over the edges crossing $\gamma_x$, and using
relations~\eqref{equ:lemme30} and~\eqref{equ:lemme30b}, we get
\begin{equation*}
  \sum_{\ws\bs\cap
  \gamma_x}\pm\theta'(0)\Ks_{\ws\bs}g_{\bs\ws}(u)=-\frac{z'}{z}(u)
  =-\frac{d}{du} \log z(u)\,,
\end{equation*}
establishing the first identity. The proof of the second is almost identical.
\end{proof}

The polygon of allowed slopes, after a rotation of
$-90$ degrees and a translation, can be identified with $N(\Gs)$ by \cite{KOS},
see also~\cite{Passare_2016}.
From the computation in the proof of Theorem~\ref{thm:slope},
or simply from Equation~\eqref{eq:height_change}, 
one observes that the slope of the measure~$\mathbb{P}^{u_1}$ is trivially $0$ with our choice
of reference configuration.
Therefore, the exact translation is obtained by anchoring
$N(\Gs)$ in such a way that $P_1$ is at the origin.
This also follows from Corollary~\ref{cor:slope_solid} below which computes the slope for any solid Gibbs measure.

\begin{cor}[Solid phases]\label{cor:slope_solid}
Suppose that $u_0$ belongs to one of the connected components of $C_0\setminus
\{\alpha_T\ ;\ T\in\T_1\}$. Then, 
\begin{equation*}
(s^{u_0},t^{u_0})=
\sum_{T\in\T_1\,:\, u_0<\alpha_T<u_1} \bigl(v_T,-h_T\bigr)\,.
\end{equation*}
In particular, the points~$(-t^{u_0},s^{u_0})$ for $u_0$ in the connected components of
$C_0\setminus\{\alpha_T\ ;\ T\in\T_1\}$ are the integer boundary vertices of the
polygon $N(\Gs)$.
\end{cor}

\begin{proof}
When $u_0$ belongs to~$C_0\setminus\{\alpha_T\ ;\ T\in\T_1\}$, the oriented path of integration
$\Cs^{u_0}_{u_1}$ is the boundary of a domain of the torus $\TT(q)$ containing
the arc of $C_0$ from $u_0$ to $u_1$. Therefore the result is a direct
consequence of the residue theorem, noticing that
$u\mapsto\theta(u-\alpha)$ has a simple zero at $u=\alpha$,
so that the potentially nonzero residues of $\frac{z'}{z}(u)$
(resp.\@ $\frac{w'}{w}(u)$) are at $\alpha_T$ for $T\in\T_1$, with
value $-v_T$ (resp.\@ $h_T$).
\end{proof}

\begin{cor}[Gaseous phase]\label{cor:slope_gas2}
Suppose that $u_0$ belongs to $C_1$. Then,
\begin{equation*}
(t^{u_0},-s^{u_0})=-\frac{1}{\pi}\Bigl(\sum_{T\in\T_1}h_T\tilde{\alpha}_T,\sum_{T\in\T_1}v_T \tilde{\alpha}_T\Bigr)=\vphi(\mapalpha),
\end{equation*}
where the lifts~$(\tilde{\alpha}_T)$ of~$({\alpha}_T)$
are in an interval of length (smaller than) $\pi$ obtained by
cutting $C_0$
at $u_1$.
In the second equality, $\vphi(\mapalpha)$ is given by Equation~\eqref{equ:map_vphi} with
the geometric Newton polygon~$N(\Gs)$ anchored so that~$P_1$ is at the origin.
\end{cor}

\begin{proof}
Let us prove the first equality for the horizontal slope. The point $u_0$ being
on the real connected component $C_1$,
the contour $\Cs^{u_0}_{u_1}$ winds once vertically
  in the positive direction on the torus and passes through
   $u_1\in C_0$. 
  The quasiperiodic property of the theta function implies that for any
  $\alpha\in C_0$,
  \begin{equation*}
    \int_{\Cs^{u_0}_{u_1}}
    \frac{\theta'}{\theta}(u-\alpha) du =
    c + 2i\tilde{\alpha},
  \end{equation*}
  where
  $e^c=-q^{-2}e^{-2i\Re{u_0}}$
  is independent of $\alpha$,
  and the determination
  $\tilde{\alpha}$
  of $\alpha$ lives in an interval of size
  $\pi$ obtained by cutting $C_0$ at
  $u_1$.
  As a consequence, integrating $\frac{z'}{z}(u)$ gives the following expression for the horizontal slope:
  \begin{equation*}
    s^{u_0}=\frac{1}{2\pi i} \sum_{T\in\T_1}  v_T
    \int_{\Cs^{u_0}_{u_1}}
    \frac{\theta'}{\theta}(u-\alpha_T) \, du=
    \frac{1}{\pi}\sum_{T\in\T_1} v_T
    \tilde{\alpha}_T,
  \end{equation*}
where we have also used the fact that $\sum_{T\in\T_1}v_T=0$.
The argument for the vertical slope is similar.

By the second point of Remark~\ref{rem:NG},
$(t^{u_0},-s^{u_0})$ is equal to $\vphi(\mapalpha)$
minus $P_1$, which gives the second equality and concludes the proof.
\end{proof}

\section{Beyond the periodic case}
\label{sec:gibbs_non_perio}

In this section we consider the general case where the graph $\Gs$ has bounded faces, is minimal,
but \emph{not necessarily periodic}. We let $\mapalpha\in X_\Gs$ be a half-angle map as defined in Section~\ref{sec:tt-def}; the domain $D$ is that of Section~\ref{sec:preliminaries_contours}. Again, we place ourselves in the context where Fock's elliptic adjacency operator is Kasteleyn,
meaning that we suppose that $\tau$ lies {in~$i\RR_{>0}$}, and that the parameter $t$ belongs to $\RR+\frac{\pi}{2}\tau$. We omit the superscript~$(t)$ from the notation.

In Section~\ref{sec:Gibbsnonperio},
using the construction of the operators $(\A^{u_0})_{u_0\in D}$ of Section~\ref{sec:inverses}, we define a two parameter family of Gibbs measures $(\PP^{u_0})_{u_0\in D}$,
and see that the three phases occurring in the
periodic situation arise here too, depending on the position of $u_0$ in $D$.
The domain $D$ in this context
plays the role of the amoeba
in the periodic case, and describes the
phase diagram of the model. As an illustration of the locality property of the Gibbs measures, we explicitly compute single-edge probabilities in the gaseous, liquid and solid phases. In Section~\ref{sec:asymptotics}, we compute asymptotics of the inverse operators $(\A^{u_0})_{u_0\in D}$, depending on the position of $u_0$. 

\subsection{Construction of Gibbs measures}
\label{sec:Gibbsnonperio}

In order to state our result for Gibbs measures, we need the 
following technical assumption {on the minimal graph~$\Gs$ and angle map~$\mapalpha\in X_\Gs$}:

$(*)$ Every finite, simply connected subgraph~$\Gs_0$
of $\Gs$ can be embedded in a \emph{periodic} minimal graph $\Gs'$,
{with the restriction of~$\mapalpha$ to the train-tracks of~$\Gs_0$ extending to an element~$\mapalpha'$ of~$X_{\Gs'}$.}

\begin{rem} 
By~\cite[Proposition~4.1]{BeaQuad} this condition is true for isoradial embeddings, if we do not impose
bipartitedness. The proof consists in enlarging the finite, simply connected subgraph 
of the diamond graph to a finite rhombus graph 
that can tile the plane in a periodic fashion. This proof does not
extend trivially when we require bipartitedness and relax constraints on the
train-tracks (to go from isoradial graphs to minimal ones).
We nevertheless believe this
condition to hold for all minimal graphs {$\Gs$ and angle maps~$\mapalpha\in X_\Gs$}.
\end{rem}

\begin{thm}\label{thm:gibbs_non_perio}
Assume hypothesis $(*)$ and
consider the dimer model on $\Gs$ with Fock's elliptic weights, and corresponding Kasteleyn operator $\Ks$. 
Then for every $u_0\in D$, the
operator $\As^{u_0}$ defines a Gibbs measure $\mathbb{P}^{u_0}$ on dimer configurations of $\Gs$, whose expression on cylinder sets is explicitly given by, for every subset of distinct edges $\{\es_1=\ws_1 \bs_1,\ldots,\es_k=\ws_k \bs_k\}$ of $\Gs$,
  \begin{equation*}
    \mathbb{P}^{u_0}(\es_1,\ldots,\es_k) = \Bigl(\prod_{j=1}^k \Ks_{\ws_j,\bs_j}\Bigr)\times
    \det_{1\leq i,j\leq k} (\As^{u_0}_{\bs_i,\ws_j}).
  \end{equation*}
The set $D$ gives the phase diagram of the model: when $u_0$ is on the top boundary of $D$, the dimer model
is gaseous; when $u_0$ is in the interior of the set $D$, the model is liquid; when $u_0$ lies on the lower boundary of $D$, the model is solid.
\end{thm}

\begin{proof}
The argument is similar to that of~\cite{BeaQuad}, see also~\cite{BeaCed:isogen,BdTR2}. The idea is to use the
  determinantal structure to show that the expressions on the right-hand side of the equality displayed above form a compatible
  family of marginals for a measure, and to conclude with Kolmogorov's extension
  theorem. A technical but crucial point is to show that these expressions are
  indeed probabilities. This can be checked using hypothesis~$(*)$ and locality of
  $\As^{u_0}$, as follows. Locality implies that the formula on the cylinder set~$\{\es_1,\ldots,\es_k\}$
  only depends on some finite simply connected subgraph~$\Gs_0$ of~$\Gs$ containing these edges,
  and on the associated half-angles.
  Using $(*)$, we can change the graph outside of~$\Gs_0$ to obtain a periodic minimal graph $\Gs'$.
  By Corollary~\ref{thm:Gibbs_measures}, the formula for cylinder sets coincides with that obtained as the weak limit of the Boltzmann measures on the toroidal exhaustion of the periodic minimal graph $\Gs'$. We can then use this to show that it indeed defines a probability measure,
  following the argument outlined above
  (see~\cite[Section 4.4]{BeaQuad} for details).

  Note however that the present setting brings an additional subtlety.
  Indeed, we need to take the arbitrary half-angle map~$\mapalpha\in X_\Gs$,
  consider its restriction~$\mapalpha_0$ to the train-tracks of~$\Gs_0$,
   and be able to extend it to a map~$\mapalpha'\in X_{\Gs'}^{per}$ such that the associated elliptic Kasteleyn operator is periodic.
   {However, the condition~$(*)$ only ensures that~$\mapalpha_0$ extends to
an element~$\mapalpha'$ of~$X_{\Gs'}$.}
This issue can be solved as follows.
The map~$\mapalpha'$ does not
belong to~$X_{\Gs'}^{per}$ \emph{a priori}, but it is~$(n\ZZ)^2$-periodic
for~$n$ big enough. Therefore, one can enlarge the fundamental domain of~$\Gs'$
by a factor~$n^2$ (an operation under which the set~$X_{\Gs'}$ is easily seen to behave well), and obtain an element~$\mapalpha'$ lying in~$X_{\Gs'}^{per}$. Finally, the associated
Kasteleyn operator~$\Ks$ is not periodic \emph{a priori}. However, enlarging
the fundamental domain of~$\Gs'$ once again (by a factor~$2$)
produces additional free parameters for~$\mapalpha'$ to ensure
that~$\vphi(\mapalpha')$ is an integer point. By Proposition~\ref{prop:angles_perio}, this implies that~$\Ks$  is periodic.

The fact that the set $D$ gives the phase diagram of the model relies on the forthcoming asymptotic computations of Section~\ref{sec:asymptotics}.
\end{proof}

\paragraph{Single-edge probabilities.}

As an illustration of the locality property of the Gibbs measures, we explicitly compute the probability of occurrence of a single edge $\es=\ws\bs$ in a dimer configuration of the graph $\Gs$. Recall that $\alpha,\beta$ denote the half-angles of the train-tracks of the edge $\ws\bs$, see Figure~\ref{fig:around_rhombus}. Recall also the definition of Jabobi's zeta function
$\tilde{Z}(u)=Z(\frac{2K}{\pi}u)=\frac{\pi}{2K}\frac{\theta_4'(u)}{\theta_4(u)}$
of Section~\ref{sec:def_Hu0}, see also~\cite[3.6.1]{Lawden},
and the constants~$K=\frac{\pi}{2}\theta_3^2(0)$ and~$K'=-i\tau K$.
In order to state
our result, let us introduce the
notation~$t_\bs:=t-\frac{\pi}{2}\tau+\mapd(\bs)=\Re(t+\mapd(\bs))$.

\begin{prop}\label{lem:single_edges}
Consider an edge $\ws\bs$ of $\Gs$ with train-track half-angles $\alpha,\beta$. Then, the probability that this edge occurs in a dimer configuration of $\Gs$ chosen with respect to the Gibbs measure $\PP^{u_0}$ is explicitly given by the following.
\begin{enumerate}
 \item \textbf{Gaseous phase:} if $u_0$ belongs to the component $C_1$ of the domain~$D$, then
 \begin{align*}
 \PP^{u_0}(\ws\bs)&=\frac{\beta-\alpha}{\pi}+\frac{K'}{\pi}[\tilde{Z}(\beta-t_\bs)-\tilde{Z}(\alpha-t_\bs)]=H^{u_0}(\beta-t_\bs)-H^{u_0}(\alpha-t_\bs),
 \end{align*}
where $H^{u_0}$ is given by~\eqref{equ:H0gas}.
\item \textbf{Liquid phase:} if $u_0$ belongs to the interior of $D$, then
\begin{equation*}
\PP^{u_0}(\ws\bs)=
\frac{1}{2\pi i}\log \frac{\theta(\beta-u_0)\theta(\alpha-\bar{u}_0)}{\theta(\beta-\bar{u}_0)\theta(\alpha-u_0)}-\frac{iK}{\pi^2}(u_0-\bar{u}_0)[\tilde{Z}(\beta-t_\bs)-\tilde{Z}(\alpha-t_\bs)].
\end{equation*}
\item \textbf{Solid phases:} if $u_0$ belongs to the component $C_0\setminus\{\alpha_T\ ;\ T\in\T\}$ of $D$, then
 \[
 \PP^{u_0}(\ws\bs)=\II_{\{\alpha<u_0<\beta\ \text{on $C_0$}\}}.
 \]
 Moreover, around every black vertex $\bs$, there is exactly one edge $\ws\bs$ incident to $\bs$ such that $\PP^{u_0}(\ws\bs)=1$, so that the dimer model is deterministic.
\end{enumerate}
\end{prop}
\begin{proof}
By Theorem~\ref{thm:gibbs_non_perio}, we have $\PP^{u_0}(\ws\bs)=\Ks_{\ws,\bs}A^{u_0}_{\bs,\ws}$.
The proof in the solid phases is a direct consequence of Remark~\ref{rem:frozen}. In the liquid and gaseous phases, part of the computations we need are
already dealt with in the proof of Theorem~\ref{thm:K_inverse_family} when showing that $\sum_{\bs}\Ks_{\ws,\bs}\A^{u_0}_{\bs,\ws}=1$. Since we were summing over black vertices incident to $\ws$, we did not need to handle the residue of $H^{u_0}$ at $\frac{\pi}{2}\tau$ because these contributions were cancelling; we need to consider this residue now. Returning to the proof, we immediately obtain
\begin{equation}\label{eq:proba_1}\textstyle
\PP^{u_0}(\ws\bs)=H^{u_0}(\beta)-H^{u_0}(\alpha)-\theta'(0)\Ks_{\ws,\bs}g_{\bs,\ws}(\frac{\pi}{2}\tau)\mathrm{Res}_{\frac{\pi}{2}\tau}H^{u_0}(u).
\end{equation}
Using Corollary~\ref{cor:Faystrisecant} as in the proof of Proposition~\ref{prop:ker}, we have
\begin{equation}\label{eq:proba_2}\textstyle
\theta'(0)\Ks_{\ws,\bs}g_{\bs,\ws}(\frac{\pi}{2}\tau)=
\theta'(0)[F^{(t+\mapd(\bs))}(\frac{\pi}{2}\tau;\beta)-F^{(t+\mapd(\bs))}(\frac{\pi}{2}\tau;\alpha)].
\end{equation}
Using Equation \eqref{eq:theta14} to express $F^{(s)}(u;a)$ using the function $\theta_4$, we obtain
\begin{align*}
\theta'(0)F^{(s)}(u;a)&=\frac{\theta_4'(s-\frac{\pi}{2}\tau-a)}{\theta_4(s-\frac{\pi}{2}\tau-a)}-
\frac{\theta_4'(u-\frac{\pi}{2}\tau-a)}{\theta_4(u-\frac{\pi}{2}\tau-a)}\\
&\textstyle=\frac{2K}{\pi}[\tilde{Z}(s-\frac{\pi}{2}\tau-a)-\tilde{Z}(u-\frac{\pi}{2}\tau-a)],
\end{align*}
where in the last line we used the definition of Jacobi's zeta function.
Plugging this into~\eqref{eq:proba_1} and~\eqref{eq:proba_2}, using the notation~$t_\bs=
t+\mapd(\bs)-\frac{\pi}{2}\tau$ and the fact that~$\tilde{Z}$ is odd, yields
\begin{equation*}\textstyle
\PP^{u_0}(\ws\bs)=H^{u_0}(\beta)-H^{u_0}(\alpha)-\frac{2K}{\pi}
\mathrm{Res}_{\frac{\pi}{2}\tau}H^{u_0}(u)
[\tilde{Z}(t_\bs-\beta)+\tilde{Z}(\beta)
-\tilde{Z}(t_\bs-\alpha)-\tilde{Z}(\alpha)].
\end{equation*}
From~\cite[Lemma 45]{BdTR1}, we know that the residue of the function $\tilde{Z}$
at the pole~$\frac{\pi}{2}\tau$ is equal to $\frac{\pi}{2K}$.
Returning to the explicit definition of the function $H^{u_0}$ in
the gaseous and liquid phases, see Equations~\eqref{equ:H0gas} and~\eqref{equ:H0liq},
we deduce that the residue of the function~$H^{u_0}$ at the pole $\frac{\pi}{2}\tau$
is equal to $\frac{K'}{2K}=\frac{\tau}{2i}$ in the gaseous phase,
and to $\frac{u_0-\bar{u}_0}{2i\pi}$ in the liquid phase.
Using again the explicit form of the functions $H^{u_0}$ gives that in both cases, the terms involving $\tilde{Z}(\beta),\tilde{Z}(\alpha)$ cancel out in the above equation.
Using once more that $\tilde{Z}$ is odd proves the result.
\end{proof}

\subsection{Asymptotics}\label{sec:asymptotics}

As in the periodic case, the measures constructed in Theorem~\ref{thm:gibbs_non_perio} have different
behaviors depending on the position of $u_0$ in the domain $D$. We now compute
precise asymptotics for $\As^{u_0}_{\bs,\ws}$ as the distance between $\bs$ and
$\ws$ gets large, depending on the position of $u_0$ in~$D$.
As a consequence,
we obtain that the phase diagram described in Corollary~\ref{thm:Gibbs_measures}
in terms of $u_0$ is still valid in the non-periodic case.

The universal behaviors we now describe require some regularity
for the map $\mapalpha$.
The following technical condition
{using the intersection number introduced in Section~\ref{sec:poles_domainD}.}
is therefore assumed to hold in this section:

$(\diamondsuit)$ There exists $\delta>0$ such that for every infinite
{simple}
path $\Pi$ {in $\GR$},
the distance between the sets
$\{ \alpha_{T}\ ;\ \Pi\wedge T=+1\}$
and
$\{ \alpha_{T'}\ ;\ \Pi\wedge T'=-1\}$
is larger than $\delta$.

It is meant to forbid {geodesics in $\GR$}
with too many steps
resembling ``back-and-forth movements" in the corresponding minimal
immersion.

\begin{rem}\leavevmode
  \begin{enumerate}
\item
Condition $(\diamondsuit)$ can be equivalently reformulated in terms
of the zeros and poles of the functions $g$ as follows: for every sequence of
vertices $(\xs_n)$ and $(\ys_n)$ of $\GR$ such that the graph distance between
$\xs_n$ and $\ys_n$ goes to infinity with $n$, the distance between the zeros
and poles of $g_{\xs_n,\ys_n}$ stays bounded from below by this $\delta>0$.
This condition prevents accumulation of zeros and poles of the functions $g$
appearing in the definition of $\As^{u_0}$.
\item
Condition~$(\diamondsuit)$ is automatically satisfied in the \emph{quasicrystalline}
setting, where only a finite number of values for the half-angles $\alpha_T$ are
allowed. In particular it holds in the periodic case. Note that
Condition $(\diamondsuit)$ is strictly stronger than forcing the half-angles of the
rhombi in the corresponding minimal immersion to be in $[\delta,\pi-\delta]$,
unless all train-tracks with distinct half-angles intersect.
\end{enumerate}
\end{rem}

We thus assume that $\Gs$ satisfies Condition~$(\diamondsuit)$, and investigate the
behavior of $\As^{u_0}_{\bs,\ws}$ when $\bs$ and $\ws$ are far apart in the
three cases.

\paragraph{Case 1: gaseous phase.} Let $u_0$ be on the component $C_1$ of $D$. The integration contour $\Cs_{\bs,\ws}^{u_0}$ is then a closed
  loop. As noted in the proof of Theorem~\ref{prop:Gibbs_measures},
  moving $u_0$ in~$C_1$ corresponds to a continuous deformation of the
  initial contour, avoiding all poles of the integrand, implying that the integral $A^{u_0}_{\bs,\ws}$ does not change.

  Let $\bs$ and $\ws$ be a black and white vertex of $\Gs$ respectively. Denote by
  $N=\operatorname{dist}_{\diamond}(\bs,\ws)$ the graph distance between $\bs$ and $\ws$
  in $\GR$, which we assume to be large.
  Consider the functions
  \begin{align*}
    G_{\bs,\ws}(v)&=\theta_4(v+t+\mapd(\ws))\theta_4(v-t-\mapd(\bs)),\\
    F_{\bs,\ws}(v)&
= \frac{1}{N}\log \frac{g_{\bs,\ws}(v+\frac{\pi
\tau}{2})}{G_{\bs,\ws}(v)}=
    \sum_{\substack{\text{$T$ train-tracks}\\\text{separating $\bs$ and $\ws$}}}
    \frac{\epsilon_T}{N} \log \theta_4 (v-\alpha_T),
  \end{align*}
  where $\epsilon_T$ is
  {the intersection number of $T$ with a simple path}
  from $\bs$ to $\ws$.
  Note that all the exponential factors appearing when expressing $g_{\bs\ws}(v+\frac{\pi\tau}{2})$ in terms of $\theta_4$ (see Relation~\eqref{eq:theta14}) cancel out, thus explaining the second equality of $F_{\bs,\ws}$.
  Therefore, we have the equality
  \begin{equation*}
    G_{\bs,\ws}(v)e^{NF_{\bs,\ws}(v)} = g_{\bs,\ws}(v+\frac{\pi\tau}{2}).
  \end{equation*}

  The function $F_{\bs,\ws}$ is not an
  elliptic function, but it is meromorphic on the cylinder~$\TT(q)\setminus C_1$.
  Furthermore, it is real on $C_0$.

  Condition~$(\diamondsuit)$ implies that:
  \begin{enumerate}
    \item the global minimum of $F_{\bs,\ws}$ on $C_1$ is bounded from above by
      a strictly negative constant, uniformly in $\bs$ and $\ws$;
    \item at the point $v_0$ where it is reached, the second
  derivative is bounded from below by a positive constant, uniformly in $\bs$
  and $\ws$.
\end{enumerate}
This analytic control on $F$ allows to obtain the following:
  
\begin{prop}[Gaseous phase]\label{prop:asympt_gas}
  When $u_0$ belongs to the component $C_1$ of $D$, and when the distance $N$ between $\bs$ and $\ws$ is large, the following asymptotics for $A^{u_0}_{\bs,\ws}$ hold:
  \begin{align*}
    A^{u_0}_{\bs,\ws} &=
    -\frac{\theta'(0)}{\sqrt{2\pi N F_{\bs,\ws}''(v_0)}}
    G_{\bs,\ws}(v_0) e^{N F_{\bs,\ws}(v_0)}(1+o(1)) \\
              &= -\frac{\theta'(0)}{\sqrt{2\pi N F_{\bs,\ws}''(v_0)}} g_{\bs,\ws}(v_0+\frac{\pi\tau}{2}) (1+o(1)).
  \end{align*}
\end{prop}

\begin{proof}
  The proof is very similar to the asymptotics of the Green function of the
  $Z$-invariant massive Laplacian~\cite{BdTR1}, and is based on the steepest descent
  method. First, continuously deform the contour $\Cs_{\bs,\ws}^{u_0}$ so that it
  crosses $C_1$ vertically at $v_0+\frac{\pi\tau}{2}$, and so that
 $\log|g_{\bs,\ws}|$ is smaller on the rest of $\Cs_{\bs,\ws}^{u_0}$ than on a neighborhood of $v_0+\frac{\pi\tau}{2}$;
  the existence of such a deformation is
    guaranteed by the maximum principle for the harmonic function
  $\log|g_{\bs,\ws}|$,
  and the existence of zeros of $g_{\bs,\ws}$ in the interval of $C_0$
where the integration contour is allowed to pass.
 Therefore, the integral along the contour
$\Cs_{\bs,\ws}^{u_0}$  in a ball $B$ of radius $N^{-r}$ with
$\frac{1}{3}<r<\frac{1}{2}$ centered at
    $v_0+\frac{\pi\tau}{2}$ can be approximated by a Gaussian integral: writing
    $u=v_0+\frac{\pi\tau}{2}+is$, with $s$ in a small interval $I$ around 0,
    we have
    \begin{align*}
      \int_{\Cs_{\bs,\ws}^{u_0}\cap B} g_{\bs,\ws}(u)du
      &= \int_I(G_{\bs,\ws}(v_0)+o(1)) e^{N(F_{\bs,\ws}(v_0)-
        \frac{s^2}{2}
      F_{\bs,\ws}''(v_0) + O(s^3))}i\, ds \\
      &=i\sqrt{\frac{2\pi}{NF_{\bs,\ws}''(v_0)}}G_{\bs,\ws}(v_0)(1+o(1)) e^{NF_{\bs,\ws}(v_0)}.
    \end{align*}
The integral over the rest of $\Cs_{\bs,\ws}^{u_0}$ is negligible when compared to
the contribution above. Multiplying the quantity above by $\frac{i\theta'(0)}{2\pi}$
yields the result.
\end{proof}

\begin{rem} Here are a few remarks about this statement:
  \begin{enumerate}
    \item As $q$ tends to 0 (\textit{i.e.}, as $\Im\tau$ goes to $\infty$)
      while $qN$ goes to $\infty$, we have
      \[
	NF_{\bs,\ws}(v)=-2q\sum_T \epsilon_T\left(\cos(v-\alpha_T)
	+O(q^4)\right)
	= 2q\langle e^{2iv},\vec{\bs\ws}\rangle + O(q^4 N)\,,
      \]
      where $\vec{\bs\ws}=-\sum_T
      \varepsilon_T e^{2i\alpha_T}$ is the vector from $\bs$ to $\ws$
    in the minimal immersion of $\Gs$ defined by $\mapalpha$, and
    $\langle\cdot,\cdot\rangle$ is the usual Euclidean scalar product in
    $\mathbb{R}^2$. The unit vector
    $e^{2iv_0}$ associated to the location $v_0$ of the minimum of $F_{\bs,\ws}$ gets
      closer and closer to the direction of $-\vec{\bs\ws}$.
           Hence, in this regime, the exponential decay rate becomes isotropic.
    \item 
	By the first consequence of Condition~($\diamondsuit$) described above,
	the coefficients of this inverse tend exponentially fast to
      0 as $N$ goes to infinity. This allows to apply the operator $\As^{u_0}$
      not only to functions with finite support, as indicated in
      Definition~\ref{def:operator_A}, but to any bounded function (or with
      subexponential growth).
\end{enumerate}
\end{rem}

The last point in the remark above is the key for the following analogue of a
maximum principle for the Kasteleyn operator $\Ks$:
\begin{prop}
  The only function of subexponential growth in the right (resp.\@ left) kernel of
  $\Ks$ is the function identically equal to 0. As a consequence, $\A^{u_0}$ is
  the only inverse of~$\Ks$ with bounded coefficients (or even with
  subexponential growth in module).
\end{prop}

\begin{proof}
  Let $f$ be a function of subexponential growth in the right kernel of $\Ks$.
  Then we can write
  \begin{equation*}
    0=\A^{u_0} 0 = \A^{u_0}(\Ks f) = (\A^{u_0} \Ks) f = f\,,
  \end{equation*}
  as all the sums involved are absolutely convergent. The proof for the left
  kernel is similar. The last point follows by considering the functions $f$
  given by the columns of $\A^{u_0}-\mathsf{B}$, where $\mathsf{B}$ is
  an inverse of~$\Ks$ with subexponentially growing coefficients.
\end{proof}

\paragraph{Case 2: liquid phase.} Let $u_0$ be an interior point of $D$.
Recall from Section~\ref{sec:circlepatterns} the definition of the {t-realization} $\Psi_{u_0}$  of the dual graph $\Gs^*$, extended to $\GR$
with a bounded function $\Xi$.
Condition$(*)$ on $\Gs$ implies by~\cite[{Theorem~10}]{KLRR} that $\Psi_{u_0}$
defines a convex embedding of $\Gs^*$. Moreover, Condition~$(\diamondsuit)$
implies that the inner and outer diameter of the faces of $\Gs^*$ in this
embedding are bounded away from 0 and infinity. As a consequence, the distances
measured in the graph $\GR$ or in the {realization} $\Psi_{u_0}$ are
quasi-isometric.

Define the unit complex number $\zeta=\zeta_{\bs,\ws}=\exp(i\mapd(\bs)+i\mapd(\ws)+2i\Re{t})$.
We now have all the ingredients to express the asymptotics of $\A^{u_0}_{\bs,\ws}$
when $\bs$ and $\ws$ are far apart. 

\begin{prop}[Liquid phase]
  \label{prop:asymp_liq1}
  When
  $u_0$ belongs to the interior of the domain $D$, and when the distance $N$ between the black
  vertex $\bs$ and the white vertex $\ws$ is
  large, the following asymptotics for $\A^{u_0}_{\bs,\ws}$ holds:
  \begin{equation*}
    \A^{u_0}_{\bs,\ws} =
    -\frac{\hcancel{\theta'(0)}{1}}{\pi}\Im\left(%
    \frac{%
  g_{\bs,\ws}(u_0)
  }{%
  \Psi_{u_0}(\bs)-\Psi_{u_0}(\ws)
  }
\zeta\right)\bar\zeta
    +O\left(\frac{%
      |g_{\bs,\ws}(u_0)|
}{%
N^{3/2}
}\right).
  \end{equation*}
\end{prop}

This form of asymptotics is very similar to the one found for the inverse
Kasteleyn operator in \cite{KOS} in
the liquid phase for general periodic, bipartite planar graphs, and in \cite{Kenyon:crit}
for the critical dimer model on isoradial graphs. The numerator in the imaginary
part has an oscillating phase, and a modulus growing like $|g_{\bs,\ws}(u_0)|$ which
can be absorbed via gauge
transformations.
The module of the denominator grows linearly with the 
graph distance $N$ between
$\bs$ and $\ws$.

\begin{proof}
  The proof is close in spirit to that of~\cite{Kenyon:crit}.
 Recall that $s_{\bs,\ws}$ is the interval of~$C_0=\mathbb{R}/\pi\mathbb{Z}$ given by
  Lemma~\ref{lem:sep_zeros_poles} containing all the poles of $g_{\bs,\ws}$. Its
  complement contains at least one zero, as $\bs$ and $\ws$ are supposed far enough,
  in particular not neighbors.

  We claim that the contour $\Cs_{\bs,\ws}^{u_0}$ connecting $\bar{u}_0$ to $u_0$
  across $C_0\setminus s_{\bs,\ws}$ can be deformed continuously into a
  contour obtained by concatenating the following three paths:
  \begin{itemize}
    \item a first path from $\bar{u}_0$ to $\bar{u}_1$ for some~$u_1$
    such that $|u_1-u_0|=O(N^{-1/2})$, along which~$|g_{\bs,\ws}(u)|$
    decreases at exponential speed in $N$;
    \item a second path from $\bar{u}_1$ to $u_1$,
    crossing $C_0$,
 along which $|g_{\bs,\ws}(u)|<|g_{\bs,\ws}(u_1)|$;
    \item a third path from $u_1$ to $u_0$ such that $|g_{\bs,\ws}(u)|$ increases (also
      at exponential speed).
  \end{itemize}

  The existence of the third path in $D$ is guaranteed by the maximum
  principle applied to the bounded harmonic function
  $u\mapsto\log|g_{\bs,\ws}(u)|$ on compact sets of the interior of~$D$,
  and by the fact that $g_{\bs,\ws}$ is the product
  of order $N$ terms which can be controlled by constants independent of $N$
  near $\bar{u}_0$  (resp.\@ $u_0$).
Then, the first  path can be chosen as the complex conjugate of the third one.
Finally, to justify the existence of the second path, one can use the fact that
  $g_{\bs,\ws}$ has at least one zero in
  the interval $C_0\setminus s_{\bs\ws}$,
  so that~$\log|g_{\bs,\ws}(u)|$  can be taken sufficiently
  negative in a neighborhood of the intersection between the
  integration contour and $C_0$.

  Following the same steps as in~\cite{Kenyon:crit}, we now estimate the
  contributions of these three pieces to the integral,
  denoted respectively by
  $I_1$, $I_2$ and $I_3$.
  The integral $I_3$ from $u_1$ to $u_0$ along the third path can be
  estimated by writing locally
  $g_{\bs,\ws}(u)=\exp(h(u))$ near $u_0$. The derivative $h'$ does not vanish in a neighborhood of
  $u_0$ and is the sum of $O(N)$ terms which are controlled, so that
  $h'(u)\asymp N$, $h''(u)=O(N)$, uniformly along this path. Integrating by parts, we obtain
  \begin{multline*}
    \frac{i\theta'(0)}{2\pi}\int_{u_1}^{u_0} g_{\bs,\ws}(u)du =
    \frac{i\theta'(0)}{2\pi}\int_{u_1}^{u_0} h'(u) \exp(h(u))\frac{du}{h'(u)}
    =\\
    \frac{i\theta'(0)}{2\pi}\left. \frac{\exp(h(u))}{h'(u)}\right|_{u_1}^{u_0}
      + \int_{u_1}^{u_0} \exp(h(u)) \frac{h''(u)}{h'(u)^2} du.
  \end{multline*}
  In the rightmost integral, the integrand is bounded by
  $\frac{C|g_{\bs,\ws}(u_0)|N}{N^2}$ for some constant $C>0$, so the integral is
  $O(|g_{\bs,\ws}(u_0)|\frac{|u_1-u_0|}{N})$. The evaluation of the integrated term
  at $u_0$ gives the main contribution and the evaluation at $u_1$ is
  negligible as $g_{\bs,\ws}(u_1)$ is exponentially small when compared to
  $g_{\bs,\ws}(u_0)$, yielding:
  \begin{equation*}
    I_3=\frac{i\theta'(0)}{2\pi} \frac{\exp(h(u_0))}{h'(u_0)}
    +O\left(\frac{1}{N^{3/2}}\right)=
  \frac{i\theta'(0)}{2\pi}\frac{g_{\bs,\ws}(u_0)}{g'_{\bs,\ws}(u_0)/g_{\bs,\ws} (u_0)}
  +O\left(\frac{1}{N^{3/2}}\right).
  \end{equation*}
  Using the product form of $g_{\bs,\ws}(u_0)$, and recalling the definition
  and Property~\eqref{eq:psi_quasi_isom} of~$\Psi_{u_0}$ from Section~\ref{sec:circlepatterns} with a bounded $\Xi$, one
  has 
\begin{equation*}
    I_3=
    \frac{i\hcancel{\theta'(0)}}{2\pi}
    \frac{g_{\bs,\ws}(u_0)}{\Psi_{u_0}(\bs)-\Psi_{u_0}(\ws)}\left(1+O\left(\frac{1}{N^{1/2}}\right)\right).
\end{equation*}

The contribution $I_1$ from $\bar{u}_0$ to $\bar{u}_1$ is computed in the same
way:
\begin{equation*}
  I_1 = 
  -
  \frac{i\hcancel{\theta'(0)}}{2\pi}
  \frac{g_{\bs,\ws}(\bar{u}_0)}{\Psi_{\bar{u}_0}(\bs)-\Psi_{\bar{u}_0}(\ws)}\left(1+O\left(\frac{1}{N^{1/2}}\right)\right).
\end{equation*}
By the choice of our contours, the integral from $\bar{u}_1$ to $u_1$ along 
the second piece is negligible when compared to
$I_1$ and $I_3$.

  All the factors $\theta(u-\alpha)$ with real $\alpha$ appearing in $g_{\bs,\ws}$
  satisfy $\theta(\bar{u}_0-\alpha)=\overline{\theta(u_0-\alpha)}$.
  They are all of this form, except two of them (corresponding to the first and
  last step of the path from $\bs$ to $\ws$). We then have that
\begin{align*}
\Psi_{\bar{u}_0}(\bs)-\Psi_{\bar{u}_0}(\ws)&={\frac{1}{\theta'(0)}}
  \frac{g'_{\bs,\ws}(\bar{u}_0)}{g_{\bs,\ws}(\bar{u}_0)}+O(1)= {\frac{1}{\theta'(0)}}\overline{%
    \left(%
      \frac{g'_{\bs,\ws}(u_0)}{g_{\bs,\ws}(u_0)}
    \right)
  } + O(1)\\
  &=
  \overline{\Psi_{u_0}(\bs)-\Psi_{u_0}(\ws)}+O(1).
\end{align*}

However, because of these two factors with non-real parameters involving $t$,
and the fact that $\theta$ is
only quasiperiodic in the vertical direction, it is not generally true
that $g_{\bs,\ws}(\bar{u}_0)=\overline{g_{\bs,\ws}(u_0)}$.
More precisely, using the fact that the imaginary part of $t$ is $\frac{\pi|\tau|}{2}$,
and thus $\bar{t}=t-\pi\tau$, we have
\begin{align*}
  \overline{\theta(u_0+t+\mapd(\bs))}&=\theta(\overline{u_0+t+\mapd(\bs)})=\theta(\bar{u}_0+t+\mapd(\bs)-\pi\tau)
  \\
  &=-q^{-1}e^{2i(\bar{u}_0+t+\mapd(\bs))}\theta(\bar{u}_0+t+\mapd(\bs))\\
  &=-e^{2i(\bar{u}_0+\Re(t)+\mapd(\bs))}\theta(\bar{u}_0+t+\mapd(\bs))
\end{align*}
and similarly
\begin{align*}
  \overline{\theta(u_0-t-\mapd(\ws))}
  &=-q^{-1}e^{-2i(\bar{u}_0-t-\mapd(\ws))}\theta(\bar{u}_0-t-\mapd(\ws))\\
  &=-e^{2i(-\bar{u}_0+\Re(t)+\mapd(\ws))}\theta(\bar{u}_0-t-\mapd(\ws))\,,
\end{align*}
so that
\begin{equation*}
  \frac{\overline{g_{\bs,\ws}(u_0)}}{g_{\bs,\ws}(\bar{u}_0)}=
  \frac{\overline{\theta(u_0+t+\mapd(\bs))}\cdot\overline{\theta(u_0-t-\mapd(\ws))}
  }{\theta(\bar{u}_0+t+\mapd(\bs))\cdot\theta(\bar{u}_0-t-\mapd(\ws))}
  =\exp(2i(\mapd(\ws)+\mapd(\bs)+2i\Re(t))=:\zeta^2\,.
\end{equation*}
This gives the following expression for the asymptotics:
\begin{equation*}
  \A^{u_0}_{\bs,\ws} =
   \frac{i\hcancel{\theta'(0)}}{2\pi}
  \left(%
    \left(\frac{g_{\bs,\ws}(u_0)}{\Psi_{u_0}(\bs)-\Psi_{u_0}(\ws)}\right)
    -\overline{%
  \left(\frac{g_{\bs,\ws}(u_0)}{\Psi_{u_0}(\bs)-\Psi_{u_0}(\ws)}
\right)}
  \zeta^{-2}
\right)
+O\left(\frac{1}{N^{3/2}}\right).
\end{equation*}

One obtains the final expression of Proposition~\ref{prop:asymp_liq1} by noticing
that $|\zeta|=1$ and factoring out $\bar{\zeta}$.
\end{proof}

Using now standard arguments~\cite{KenyonGFF,BeaGFF,KOS,Russkikh}, one readily
obtains the following result for the fluctuations of the height function.

\begin{cor}
If $\Psi$ is used
to draw the graph on the plane, then the centered height function
converges weakly in distribution to $1/\sqrt{\pi}$ times the Gaussian Free Field, with standard
covariance structure given by the full plane Green function.
\end{cor}

\paragraph{Case 3: solid phases.}
  According to the single-edge probability computation for solid phases in
  Proposition~\ref{lem:single_edges}, one sees that the state of every edge is
  deterministic: each edge is either a dimer almost surely, or vacant almost
  surely. In particular, the corresponding Gibbs measure is concentrated on a
  single configuration.

\section{Invariance under elementary transformations}
\label{sec:spider}

{The dimer model on a bipartite graph behaves in a controlled way with respect to elementary transformations/moves of the underlying graph, see for example~\cite{Thurston,Postnikov}. In this section, we focus on two such moves, 
\emph{shrinking/expanding of a~$2$-valent vertex} and \emph{spider move}, that are relevant for minimal graphs, see Figure~\ref{fig:elementary}. These two moves were first considered by Kuperberg~\cite{Kuperberg}, then Propp~\cite{Propp} who coined the name \emph{urban renewal}. As mentioned in Section~\ref{sec:tt-per},
Goncharov and Kenyon~\cite{GK}, relying on the work of Thurston~\cite{Thurston}, proved that two minimal graphs on a torus
have the same Newton polygon if and only if they
are related by a sequence of these two moves.}

\begin{figure}[htb]
    \centering
    \includegraphics[width=0.8\linewidth]{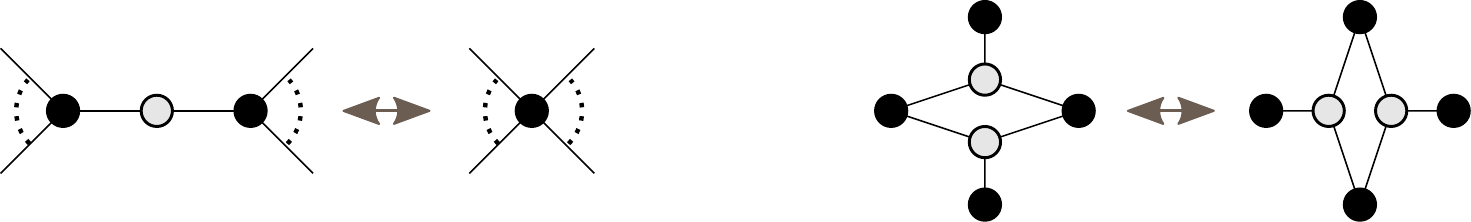}
    \caption{Shrinking/expanding of a~$2$-valent (white) vertex, and spider move (with black boundary vertices).}
    \label{fig:elementary}
  \end{figure}

The aim of this section is to study which dimer models are invariant under these elementary transformations.
{%
  Before doing so, let us make a preliminary remark:
throughout this section, and unlike in the rest of this article, the symbol~$\Gs$ denotes a \emph{finite}
bipartite planar graph, and~$\Ks$ an arbitrary Kasteleyn operator, not necessarily Fock's elliptic operator.
Also, it is convenient to undertake this study in the more general context of ``dimer models'' with
arbitrary (non-necessarily positive) weights.

Recall that a dimer model on a finite, bipartite, planar graph~$\Gs$ comes with an edge weight function,
defining a partition function which can be expressed using face weights as described in Equation~\eqref{eq:Z-face}.
Also, the choice of a Kasteleyn orientation on~$\Gs$ leads to an associated Kasteleyn matrix which defines an
operator~$\K\colon\CC^{\Bs}\to\CC^{\Ws}$, where~$\Bs$ and~$\Ws$ denote the set
of black and white vertices of~$\Gs$, respectively.
Note that any such Kasteleyn operator~$\Ks$ allows to recover the face weights in the obvious way.
Therefore, it makes sense to study the invariance of the face-weight partition function via conditions on the coefficients of~$\Ks$, and this is what we will do.
}

This invariance is with respect to \emph{local moves} of the underlying graph, so the  
{%
Kasteleyn coefficients (or equivalently, the weights) of edges not affected by these moves should remain unchanged, and those of edges affected by the move should be chosen so that the face-weight
partition function remains unchanged. As a consequence,
in the probabilistic setting (i.e.\ when the edge weights are positive), the Boltzmann probability of events not affected by these moves are invariant.  Examples are explicitly carried out in~\cite[Section 11]{Postnikov}, see also~\cite[Theorem 4.7]{GK} for the urban renewal move.

In a nutshell, we will show that, assuming that the Kasteleyn coefficients are functions of the train-track angles,
this invariance is equivalent to the corresponding Kasteleyn coefficients
being antisymmetric and satisfying Fay's identity in the form of Corollary~\ref{cor:FayFock} (see Corollary~\ref{cor:2-valent} and Theorem~\ref{thm:spider} below for the precise statements). In particular, the dimer model with weights given by~\eqref{def:Kast_elliptic} is invariant under these moves, a fact already obtained in another form by Fock (see~\cite[Proposition~1]{Fock}).

After this lengthy introduction, we now begin our study.
}

The following result is fairly straightforward,
and was already observed (at least partially) in~\cite[p.8]{KLRR}.
We include the proof for completeness.

\begin{prop}
\label{prop:2-valent}
{Given a finite, bipartite, planar graph~$\Gs$,} the following conditions are equivalent:
\begin{enumerate}
\item The {face-weight partition function} is invariant under shrinking/expanding of~$2$-valent vertices.
\item Given any white (resp.\@ black)~$2$-valent vertex~$\ws$ (resp.\@~$\bs$) with
  adjacent vertices~$\bs_1,\bs_2$ (resp.~$\ws_1,\ws_2$), we have the
  equalities~$\K_{\ws,\bs_1}+\K_{\ws,\bs_2}=\K_{\ws_1,\bs}+\K_{\ws_2,\bs}=0$.
\end{enumerate}
 In such a case, the kernel of~$\K\colon\CC^\Bs\to\CC^\Ws$ is invariant under shrinking/expanding of~$2$-valent vertices, up to canonical isomorphism.
\end{prop}

\begin{proof}
Let us consider a {finite, bipartite, planar} graph~$\Gs$ with a~$2$-valent white vertex~$\ws$ and
corresponding Kasteleyn coefficients~$\K_{\ws,\bs_1}$ and~$\K_{\ws,\bs_2}$, and
let~$\Gs'$ be the graph obtained by shrinking~$\ws$. The weights of all the
faces of~$\Gs$ and~$\Gs'$ coincide, apart from the two faces adjacent to~$\ws$;
taking into account the Kasteleyn phase, these two face weights get multiplied
by~$-\frac{\K_{\ws,\bs_1}}{\K_{\ws,\bs_2}}$
and~$-\frac{\K_{\ws,\bs_2}}{\K_{\ws,\bs_1}}$, respectively. Therefore, the
partition function is invariant if and only if the
equation~$\K_{\ws,\bs_1}+\K_{\ws,\bs_2}=0$ holds.
The case of a~$2$-valent black vertex is treated in an analogous way, and the equivalence of Conditions~$1$ and~$2$ is checked.

Let us now consider a~$2$-valent white vertex, with notations as above, and
assume that~$\K_{\ws,\bs_1}+\K_{\ws,\bs_2}=0$. The
operators~$\K\colon\CC^B\to\CC^W$ and~$\K'\colon\CC^{\Bs'}\to\CC^{\Ws'}$
for~$\Gs$ and~$\Gs'$ are not defined on the same space, but~$\CC^{\Bs'}$
naturally embeds in~$\CC^\Bs$ as the set of~$g\colon \Bs\to\CC$ such
that~$g(\bs_1)=g(\bs_2)$. By assumption, the operator~$\K$
satisfies~$\K_{\ws,\bs_1}=-\K_{\ws,\bs_2}\neq 0$, so the white vertex~$\ws$
contributes to the kernel of~$\K$ via the equation~$g(\bs_1)=g(\bs_2)$. As a
consequence, the kernel of~$\K$ is included in the
subspace~$\CC^{\Bs'}\subset\CC^\Bs$, and since there is no corresponding white
vertex in~$\Gs'$, it coincides with the kernel of~$\K'$.

Let us finally assume that~$\Gs'$ is obtained from~$\Gs$ by shrinking a black
vertex~$\bs$, with adjacent white vertices~$\ws_1,\ws_2\in \Ws$ merging into a
single vertex~$\ws'\in \Ws'$, and let us assume the
equality~$\K_{\ws_1,\bs}+\K_{\ws_2,\bs}=0$. The vertices~$\ws_1,\ws_2$ induce
two equations of the form~$\sum_{i\in I} k_ig(\bs_i)=-k g(\bs)$ and~$\sum_{i\in J}
k_i g(\bs_i)=kg(\bs)$, with~$k:=\K_{\ws_1,\bs}=-\K_{\ws_2,\bs}\neq 0$. The space of
solutions to these equations is canonically isomorphic to the space of solutions
to the single equation~$\sum_{i\in I\cup J} k_i g(\bs_i)=0$, which is induced by the
vertex~$\ws'\in \Ws'$. This implies that the kernels of~$\K$ and~$\K'$ are canonically isomorphic.
\end{proof}

\begin{cor}
\label{cor:2-valent} {Let $\Gs$ be a finite, bipartite, planar graph, and let us assume that the Kasteleyn
coefficients are functions of train-track angles. Then,
the face-weight} partition function and the kernel of~$\K$ are invariant under shrinking/expanding of~$2$-valent vertices in the following cases:
\begin{enumerate}
\item
{$\K_{\ws,\bs}=e^{2i\beta}-e^{2i\alpha}$, with~$\alpha,\beta$ as in Figure~\ref{fig:around_rhombus},
for any value of~$\mapalpha$.}
\item {$\K^{(t)}_{\ws,\bs}$} defined by~\eqref{def:Kast_elliptic}, for any value of the parameters~$\mapalpha,\tau$ and~$t$.
\end{enumerate}
\end{cor}
\begin{proof}
We only need to check that these weights satisfy Condition~$2$ in
Proposition~\ref{prop:2-valent}. Let us consider the case of a~$2$-valent white
vertex~$\ws$, with adjacent faces~$\fs_1,\fs_2$ in~$G$ and corresponding
faces~$\fs_1',\fs_2'$ in~$\Gs'$.
Finally, let us write~$\alpha,\beta$ for the relevant train-track parameters, as
illustrated below.
\begin{figure}[h]
\centering
\def\svgwidth{0.7\linewidth}
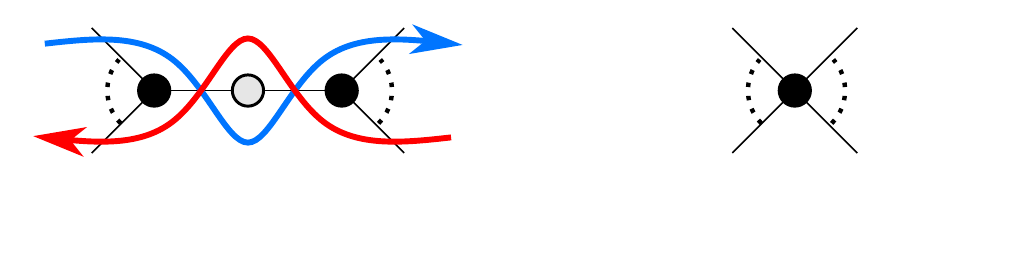
\end{figure}
By definition, the Kasteleyn operator has
coefficients~$\K_{\ws,\bs_1}=e^{2i\beta}-e^{2i\alpha}$,~$\K_{\ws,\bs_2}=e^{2i\alpha}-e^{2i\beta}=-\K_{\ws,\bs_1}$ in the first case, and
\[
  \K^{(t)}_{\ws,\bs_1}=\frac{\theta({\beta-\alpha})}{\theta(t+\eta(\fs_1))\theta(t+\eta(\fs_2))}\,,\quad\K^{(t)}_{\ws,\bs_2}=\frac{\theta({\alpha-\beta})}{\theta(t+\eta(\fs_1))\theta(t+\eta(\fs_2))}=-\K^{(t)}_{\ws,\bs_1}
\]
in the second, so Condition~$2$ is satisfied. The case of a~$2$-valent black vertex is checked in the same way.
\end{proof}

{In particular, if~$\Gs$ is isoradial and the angle map~$\mapalpha$ belongs to the space of isoradial embeddings of~$\Gs$ (see~\cite{KeSchlenk}), then the first point shows that Kenyon's critical weights~$\K^{\mathrm{crit}}_{\ws,\bs}=e^{2i\beta}-e^{2i\alpha}$ define a (probabilistic) model which is invariant under shrinking/expanding of~$2$-valent vertices.
Also, if~$\Gs$ is minimal and the parameters~$\mapalpha,\tau$ and~$t$ satisfy the conditions of Proposition~\ref{prop:kastorient}, then the second point shows that the same can be said of Fock's elliptic weights.}

\begin{wrapfigure}{r}{0.2\textwidth}
  \def\svgwidth{0.2\textwidth}
  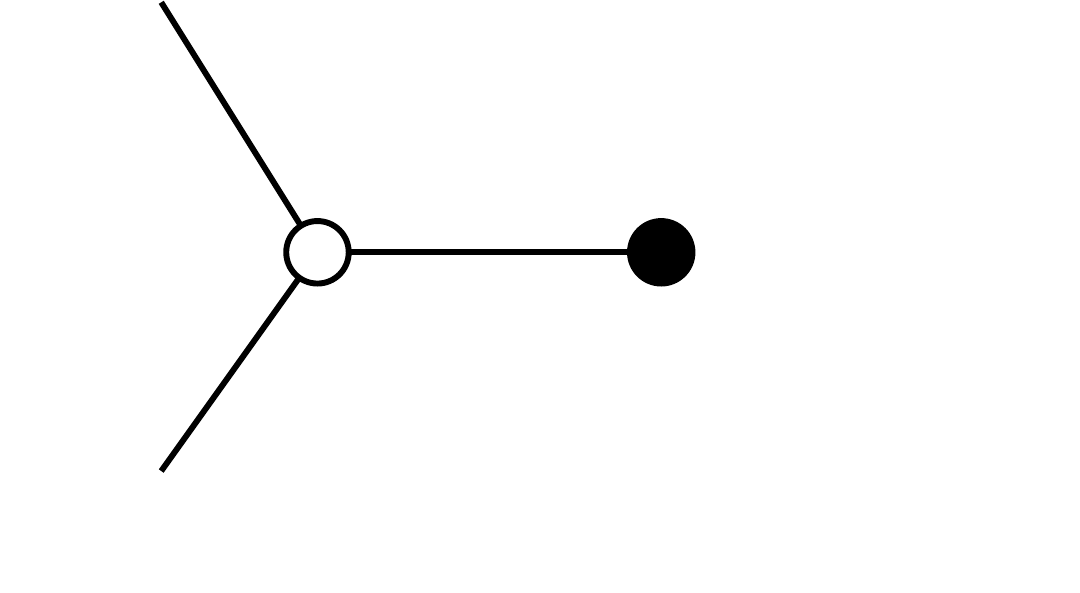
\end{wrapfigure}
Now that the invariance of the model under this first move is {understood}, we turn to the spider move.

First note that any bipartite graph without degree~$1$ vertices can be transformed into a graph with only trivalent white vertices using reduction of~$2$-valent white vertices and expansion of~$2$-valent black vertices. For such graphs, Fock's elliptic weights take a particularly nice form, as explained now.

\begin{lem}
\label{lemma:spider}
Assume that a white vertex~$\ws\in \Ws$ is trivalent, with adjacent train-track parameters~$\alpha,\beta,\gamma$ as illustrated above.
Then the weights defined by~\eqref{def:Kast_elliptic} are gauge equivalent to the weights
\[
\tilde{\Ks}^{(t)}_{\ws,\bs}=
{\theta(\beta-\alpha)\theta(\alpha+\beta-s)=F_{s}(\beta,\alpha)},
\]
where~$s=t+\eta(\ws)+\alpha+\beta+\gamma$, {and $F_s(\beta,\alpha)$ is defined in Corollary~\ref{cor:FayFock}}. Furthermore, this parameter~$s=s(\ws)$ is constant on the four white vertices appearing in any spider move with black boundary vertices.
\end{lem}
\begin{proof}
Let~$\ws\in \Ws$ be any trivalent white vertex, with adjacent
faces~$\fs_1,\fs_2,\fs_3$. Multiplying the weights adjacent to~$\ws$
by~${-}\theta(t+\eta(\fs_1))\theta(t+\eta(\fs_2))\theta(t+\eta(\fs_3))$,
we obtain the gauge equivalent
weights~$\tilde{\Ks}^{(t)}_{\ws,\bs}={-}\theta(\beta-\alpha)\theta(t+\eta(\fs_3))$,
where~$\fs_3$ denotes the face adjacent to~$\ws$ but not to~$\bs$. The
definition of~$\eta$ yields the equality~$\eta(\fs_3)=\eta(\ws)+\gamma$ which, {using also the anti-symmetry of $\theta$}, implies the first statement.
Let us now consider a spider move with black boundary vertices, as illustrated
in Figure~\ref{fig:spider}. If~$\ws_1,\ws_2,\ws_1',\ws_2'$ denote the four white vertices involved in this move,
then by definition of~$\eta$, we have the equalities
\[
\eta(\ws_1)+a+b+d=\eta(\ws_2)+a+b+c=\eta(\ws_1')+a+c+d=\eta(\ws_2')+b+c+d\,.
\]
This implies that~$s=s(\ws)$ is constant on these four vertices.
\end{proof}

The bottom line of this discussion is the following: for the study of the
invariance under spider moves (say, with black boundary vertices)
of {a wide family of weights which includes} Fock's elliptic weights as well as Kenyon's critical weights, it can be assumed that the Kasteleyn
coefficients are of the
form~$\K^{(t)}_{\ws,\bs}
={H_t(\beta,\alpha)}$,
{for some function~$H_t$ of the train-track angles satisfying $H_t(\beta,\alpha)=-H_t(\alpha,\beta)$, and}
with~$t=t(\ws)$
constant on the four white vertices appearing in any such spider move.
Obviously, the study of invariance {of the partition function} under spider moves with white boundary
vertices can be performed in the same way by exchanging the roles of the colors
and making all black vertices trivalent. {However, it does not make sense to
speak about invariance of the kernel of~$\Ks^{(t)}\colon\CC^\Bs\to\CC^\Ws$ under such spider moves,
as the set of black vertices is not preserved.}

We are now ready to state and prove the main result of this section.

\begin{thm}
\label{thm:spider}
Consider a dimer model on a {finite} bipartite, planar graph~$\Gs$, with Kasteleyn coefficients as described above. The following conditions are equivalent:
\begin{enumerate}
\item The {face-weight partition function}, seen as a function of the train-track parameters, is invariant under spider moves with black boundary vertices.
\item For all~${(a,b,c,d)\in\CC^4,\ t\in\CC}$, we have the equality
\begin{equation}
\label{equ:FayFock}
{H_t(a,b)H_t(c,d)+H_t(a,c)H_t(d,b)+H_t(a,d)H_t(b,c)=0\,.}
\end{equation}

\end{enumerate}
In such a case, the kernel of~$\K^{(t)}\colon\CC^\Bs\to\CC^\Ws$ is invariant under spider moves with black boundary vertices.
\end{thm}

\begin{proof}
Let us assume that the bipartite graphs~$\Gs$ and~$\Gs'$ are related by a spider
move with black boundary vertices. We denote by~$\fs,\fs_1,\fs_2,\fs_3,\fs_4$
the five faces of~$\Gs$ involved in this local transformation, and
by~$\fs',\fs'_1,\fs'_2,\fs'_3,\fs'_4$ the corresponding faces in~$\Gs'$, as
illustrated in Figure~\ref{fig:spider}. Finally, let us write~$a,b,c,d$ for the
relevant train-track parameters, also illustrated in this same figure.
\begin{figure}[tb]
    \centering
    \begin{overpic}[width=13cm]{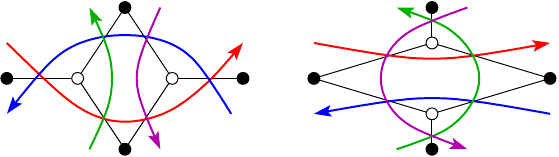}
    \put(22,13){\scriptsize $\fs$}
    \put(76,13){\scriptsize $\fs'$}
    \put(10,22){\scriptsize $\fs_1$}
    \put(10,5){\scriptsize $\fs_2$}
    \put(32,5){\scriptsize $\fs_3$}
    \put(32,22){\scriptsize $\fs_4$}
    \put(44,21){\scriptsize $a$}
    \put(-1,6){\scriptsize $b$}
    \put(28.5,-0.5){\scriptsize $c$}
    \put(15,28){\scriptsize $d$}
    \put(15.5,13.3){\scriptsize $\ws_1$}
    \put(26.2,13.3){\scriptsize $\ws_2$}
    \put(-3,13.2){\scriptsize $\bs_2$}
    \put(21.2,29){\scriptsize $\bs_1$}
    \put(21.2,-2){\scriptsize $\bs_3$}
    \put(45,13.2){\scriptsize $\bs_4$}
    \put(3,26){$\Gs$}
    \put(64,22){\scriptsize $\fs_1'$}
    \put(64,5){\scriptsize $\fs_2'$}
    \put(86,5){\scriptsize $\fs_3'$}
    \put(86,22){\scriptsize $\fs_4'$}
    \put(100,20){\scriptsize $a$}
    \put(54,6){\scriptsize $b$}
    \put(85,0){\scriptsize $c$}
    \put(69,27){\scriptsize $d$}
    \put(73.3,20.9){\scriptsize $\ws'_1$}
    \put(73.5,5.7){\scriptsize $\ws'_2$}
    \put(52,13.2){\scriptsize $\bs_2$}
    \put(76.2,29){\scriptsize $\bs_1$}
    \put(76.2,-2){\scriptsize $\bs_3$}
    \put(100,13.2){\scriptsize $\bs_4$}
    \put(58,26){$\Gs'$}
    \end{overpic}
    \caption{Faces, vertices and train-tracks involved in a spider move.}
    \label{fig:spider}
  \end{figure}
By~\cite[Theorem~4.7]{GK}, {see also~\cite{Postnikov}}, the graphs~$\Gs$ and~$\Gs'$ have identical
partition functions if and only if the
weights~${\Wscr}:={\Wscr}_\fs$,~${\Wscr}':={\Wscr}_{\fs'}$,~${\Wscr}_i:={\Wscr}_{\fs_i}$ and~${\Wscr}_i':={\Wscr}_{\fs_i'}$
satisfy the equalities
\[
{\Wscr}'={\Wscr}^{-1},\quad {\Wscr}'_i={\Wscr}_i(1+{\Wscr})\quad(i=1,3)\quad\text{and}\quad {\Wscr}'_j={\Wscr}_j(1+{\Wscr}^{-1})^{-1}\quad(j=2,4)\,.
\]
Hence, we are left with the analysis of these identities.
(The conventions of~\cite{GK} lead to face weights inverse to ours,
hence the exchange of black and white vertices in our Figure~\ref{fig:spider} 
compared to Figure~21 of~\cite{GK}.)
Computing the face weights of~$\fs$ and~$\fs'$ taking into account the Kasteleyn phase, we obtain
\begin{equation*}
{\Wscr}=-\frac{H_t(b,c)H_t(a,d)}{H_t(a,c)H_t(b,d)}\quad\text{and}\quad
{\Wscr}'=-\frac{H_t(c,a)H_t(d,b)}{H_t(c,b)H_t(d,a)}\,.
\end{equation*}
Therefore, the equality~${\Wscr}'={\Wscr}^{-1}$ follows from the assumption that~$H_t$ is antisymmetric.
We now turn to the identity~${\Wscr}'_1={\Wscr}_1(1+{\Wscr})$. A direct computation of the face weights
of~$\fs_1$ and~$\fs_1'$ yields
\begin{equation*}
{\Wscr}_1=-X\,\frac{H_t(d,b)}{H_t(a,b)}\quad\text{and}\quad {\Wscr}_1'=-X\,\frac{H_t(d,c)}{H_t(a,c)}
\end{equation*}
for some~$X$. Using the equations displayed above, the identity~${\Wscr}'_1={\Wscr}_1(1+{\Wscr})$ gives
\[
\frac{H_t(d,c)}{H_t(a,c)}=\frac{H_t(d,b)}{H_t(a,b)}\left(1-\frac{H_t(b,c)H_t(a,d)}{H_t(a,c)H_t(b,d)}\right)\,,
\]
which, using the antisymmetry of~$H_t$, is immediately seen to be equivalent to Equation~\eqref{equ:FayFock}.
The equality~${\Wscr}'_3={\Wscr}_3(1+{\Wscr})$ gives the same equations, only exchanging the roles
of~$a$ and~$b$, and of~$c$ and~$d$. Since Equation~\eqref{equ:FayFock} is
invariant by these permutations, it is also equivalent to this second equality.
As for the equality~${\Wscr}_2'={\Wscr}_2(1+{\Wscr}^{-1})^{-1}$, it can be restated
as~${\Wscr}_2={\Wscr}_2'(1+{\Wscr}')$. This is nothing but the already analysed
equation~${\Wscr}'_3={\Wscr}_3(1+{\Wscr})$ applied to a~$\pi/2$-rotation of both graphs, which
exchanges the roles of~$\Gs$ and~$\Gs'$. The last equality,
namely~${\Wscr}_4'={\Wscr}_4(1+{\Wscr}^{-1})^{-1}$, can be treated in the same way. This concludes
the proof of the equivalence of Conditions~$1$ and~$2$.

Let us now assume that~$\Gs$ and~$\Gs'$ are related by a spider move with black
boundary vertices, as illustrated in Figure~\ref{fig:spider}. We will assume the
notation of this figure.
The set of black vertices for~$\Gs$ and~$\Gs'$ are identical, and we will
directly show that the kernels of~$\K^{(t)}$ and~$\K'^{(t)}$ coincide
in~$\CC^\Bs=\CC^{\Bs'}$. For both graphs, we have two white vertices defining two
equations that only involve the value of~$g\in\CC^\Bs$ on~$\bs_1,\ldots,\bs_4$, thus
defining a~$4\times 2$-matrix. In the case of~$\Gs$, it is given by
\[
M=\begin{pmatrix}
H_t(b,d) & H_t(a,b) & H_t(d,a) & 0 \\
H_t(c,d) & 0 & H_t(a,c) & H_t(b,a)
\end{pmatrix}
\]
and in the case of~$\Gs'$, by
\[
M'=\begin{pmatrix}
H_t(c,d) & H_t(a,c) & 0 & H_t(d,a)\\
0 & H_t(c,b) & H_t(d,c) & H_t(b,d)
\end{pmatrix}\,.
\]
Now, consider the two vectors
\[
g_1=\begin{pmatrix}
H_t(c,a)\\
H_t(c,d)\\
H_t(c,b)\\
0
\end{pmatrix}
\quad\text{and}\quad
g_2=\begin{pmatrix}
0\\
H_t(a,d)\\
H_t(a,b)\\
H_t(a,c)
\end{pmatrix}\,.
\]
Equation~\eqref{equ:FayFock} easily implies that~$g_1$ and~$g_2$ lie in the kernel
of both~$M$ and~$M'$. Since~$g_1$ and~$g_2$ are linearly independent and the
matrices~$M$ and~$M'$ have rank~$2$, we conclude that the kernels of~$M$
and~$M'$ coincide (with the span of~$g_1$ and~$g_2$). Since each of the other
white vertices define the same equation for~$\Gs$ and~$\Gs'$, this directly implies
that the kernels of~$\K^{(t)}$ and~$\K'^{(t)}$ coincide.
\end{proof}

By direct computation in Case~1 and Fay's identity in the form of
Corollary~\ref{cor:FayFock} in Case~2 {(originally due to Fock~\cite{Fock})}, we immediately obtain
the following result.

\begin{cor}
\label{cor:spider}
{Let $\Gs$ be a finite, bipartite, planar graph.}
The {face-weight} partition function and the kernel of~$\K$ are invariant under
spider moves (with black boundary vertices in the case of~$\K$) in the following
cases:
\begin{enumerate}
\item {$\K_{\ws,\bs}=e^{2i\beta}-e^{2i\alpha}$, with~$\alpha,\beta$ as in Figure~\ref{fig:around_rhombus},
for any value of~$\mapalpha$.}
\item The weights defined by~\eqref{def:Kast_elliptic} for any value of the parameters~$\mapalpha,\tau$ and~$t$.\qed
\end{enumerate}
\end{cor}

{The discussion after Corollary~\ref{cor:2-valent} applies once again:
if~$\Gs$ is isoradial and the angle map defines an isoradial embeddings of~$\Gs$,
then the first point shows that Kenyon's critical weights~$\K^{\mathrm{crit}}_{\ws,\bs}$ define a
model which is invariant under spider move, a fact originally due to Kenyon~\cite{Kenyon:crit}.
Finally, if~$\Gs$ is minimal and the parameters~$\mapalpha,\tau$ and~$t$ satisfy the conditions of Proposition~\ref{prop:kastorient}, then this statement also holds for Fock's elliptic weights.}

\begin{rem}
We conclude this section with three final comments.
\begin{enumerate}
\item
As explicited in Corollaries~\ref{cor:2-valent} and~\ref{cor:spider}, the results of this section
imply that any ``dimer model'' defined by Fock's elliptic weights~\eqref{def:Kast_elliptic} is spider invariant.
More generally, these results actually imply that this holds for Fock's weights~\cite{Fock} defined via any
odd Riemann theta function, whatever the genus of the curve.
This was already observed by Fock in another form, see~\cite[Proposition~1]{Fock},
and will be used in our forthcoming article~\cite{BCdT:genusg}.
\item
With Theorem~\ref{thm:spider} in hand, it is natural to ask for other classes of
weights satisfying Equation~\eqref{equ:FayFock}, and hence giving rise to spider
invariant models. In particular, one might wonder if there are other classes of
weights which, as Kenyon's critical ones, are \emph{local}, in the sense that
the corresponding Kasteleyn coefficients~$\Ks_{\ws,\bs}$ only depend on~$\alpha,\beta$
(and on no additional parameter~$t$). Also, we can further ask these
coefficients to be \emph{rotationally invariant}, \emph{i.e.},\@ to
satisfy~$\Ks_{\ws,\bs}(\alpha+s,\beta+s)=\Ks_{\ws,\bs}(\alpha,\beta)$ for all~$s$. Finally,
it is natural to ask for the corresponding edge weights to be positive. Such a
search is bound to fail. Indeed, it can be shown that any rotationally invariant
local Kasteleyn coefficients satisfying Equation~\eqref{equ:FayFock} and
inducing  positive edge weights are gauge equivalent to Kenyon's critical
weights.
\item
Even relaxing the rotational invariance condition does not help much. It can be proved that for any~$\ZZ^2$-periodic minimal graph,
and any local Kasteleyn coefficients satisfying Equation~\eqref{equ:FayFock} and inducing  positive edge weights,
the corresponding spectral curve is an irreducible rational curve.
Therefore, if one wishes to break free from the rational curves, and this is one of the aims of the present article, then the corresponding edge weights can not be local.
\end{enumerate}
\end{rem}

\section{Connection to known results}
\label{sec:connection}

We now present relations between our work and other dimer models on
isoradial graphs that have already been handled in the literature.
In Section~\ref{sec:rational} we show how Kenyon's critical dimer models~\cite{Kenyon:crit} can be obtained as rational limits of our elliptic setting.
Then, in Section~\ref{sec:previous_genus1}, we explain how
two special classes of bipartite isoradial graphs with local elliptic weights, 
namely the graph $\GQ$~\cite{BdTR2} and the double graph $\GD$~\cite{dT_Mass_Dirac} can be viewed as special cases
of the constructions of this paper.

\subsection{Degeneration to the rational case}
\label{sec:rational}

When $q$ goes to zero, or equivalently when $|\tau|$ goes to $\infty$, the
torus $\TT(q)$ becomes a cylinder, which is mapped to the Riemann sphere
via $u\mapsto\lambda=e^{2iu}$.
In this regime, we have
$\theta(u)=2\sin(u)q^{1/4}(1+O(q))$.
Since the parameter~$t$ belongs to~$\RR+\frac{\pi}{2}\tau$,
the entries of the Kasteleyn operator~$\Ks=\K[t]$
become
\begin{equation*}
  \K_{\ws,\bs}=
  -i q^{3/4} e^{2i(\mapd(\ws)+\Re t)}
  \left(\overline{e^{2i\beta}-e^{2i\alpha}}\right)(1+O(q))\,.
\end{equation*}
Up to a gauge equivalence, this is nothing but
the complex conjugate of
Kenyon's formula
\begin{equation}
  \K^{\text{crit}}_{\ws,\bs}=e^{2i\beta}-e^{2i\alpha}=
  i e^{i(\alpha+\beta)}2\sin(\beta-\alpha)
  \label{eq:Kcrit}
\end{equation}
for critical weights on
isoradial graphs~\cite{Kenyon:crit}.

The author then gives local functions $f_\bs(\lambda)$ in the kernel of $\K^{\text{crit}}$, which we denote by $f_{\bs,\ws}(\lambda)$ to be consistent with our conventions. These functions
can 
also be recovered from rescaled limits of our functions~$g=\g[t]$.
Indeed, when $q\to 0$, for $\bs$ and $\ws$ neighbors in $\Gs$ as in
Figure~\ref{fig:around_rhombus}, one has:
\begin{equation}\label{eq:g_crit}
  g_{\bs,\ws}(u)= -q^{-1}e^{-2i(\mapd(\ws)+\Re(t))}
  \frac{e^{2iu}(1+O(q))}{(e^{2iu}-e^{2i\alpha})(e^{2iu}-e^{2i\beta})}.
\end{equation}
The denominator is exactly what appears when computing the functions
$f_{\bs,\ws}(\lambda)$, the prefactor plays no role, and the $e^{2iu}$ in the
numerator is absorbed
in the Jacobian
when performing the change of variables from $u$
to $\lambda$ 
in Equation~\eqref{eq:Keninv} below.

Using the functions $f_{\bs,\ws}(\lambda)$, Kenyon then obtains a local expression for \emph{one} inverse 
of the Kasteleyn operator~\cite[Theorem 4.2]{Kenyon:crit}, which in the $\ZZ^2$-periodic case, corresponds to $(0,0)$ magnetic field. It is given by
\begin{equation}
\label{eq:Keninv}
  (\Ks^{\text{crit}})^{-1}_{\bs,\ws}=
  \frac{1}{4i\pi^2}
  \oint f_{\bs,\ws}(\lambda)\log(\lambda) d\lambda,
\end{equation}
where the integral is computed along a contour surrounding poles of the integrand and 
avoiding a 
ray from $0$ to $\infty$. We now explain how, using~\eqref{eq:Keninv}, one obtains a local formula 
for a 
\emph{two-parameter} family of inverses while staying in the realm of genus~0, thus recovering the limit of our elliptic operators $\A^{u_0}$ as $q\rightarrow 0$. The main idea is 
to consider gauge equivalent dimer models corresponding to Möbius transformations of 
half-angles assigned to train-tracks. This idea takes its source in~\cite[Section 5.3]{KO:Harnack} which handles the connection between generic rational Harnack curves with triangular Newton polygons and dimers 
on the hexagonal lattice with critical isoradial weights.

Fix $\zeta\in\CC$, $|\zeta|<1$, and consider the Möbius transformation
$U(\lambda)=\frac{\lambda-\zeta}{1-\bar{\zeta}\lambda}$,
preserving the unit circle.
From $U$ construct modified half-angles
$(\tilde{\alpha}_T)_{T\in \T}$
by setting, for every train-track $T$ of $\T$, $e^{2i\tilde{\alpha}_T}=U^{-1}(e^{2i\alpha_T})$. 
Then, an explicit computation yields that non-zero coefficients of the modified 
Kasteleyn operator $\tilde{\Ks}^{\text{crit}}$ are given by
\begin{equation}\label{equ:gauge_K}
\tilde{\Ks}^{\text{crit}}_{\ws,\bs}=e^{2i\tilde{\beta}}-e^{2i\tilde{\alpha}}=
(e^{2i\beta}-e^{2i\alpha})\frac{1-|\zeta|^2}{(e^{2i\beta}\bar{\zeta}+1)(e^{2i\alpha}\bar{\zeta}+1)}=
\frac{1-|\zeta|^2}{\bar{\zeta}^2}f_{\bs,\ws}(-\bar{\zeta}^{-1})\Ks^{\text{crit}}_{\ws,\bs},
\end{equation}
implying that $\Ks^{\text{crit}}$ and $\tilde{\Ks}^{\text{crit}}$ are gauge equivalent.

Now observe that the function $\tilde{f}_{\bs,\ws}(\lambda)$, defined as the function $f_{\bs,\ws}(\lambda)$ with 
modified half-angles $(\tilde{\alpha}_T)_{T\in\T}$, satisfies the relation:
\[
\tilde{f}_{\bs,\ws}(\lambda)=\frac{\bar{\zeta}^2}{f_{\bs,\ws}(-\bar{\zeta}^{-1})(1-\bar{\zeta}\lambda)^2}f_{\bs,\ws}(U(\lambda)).
\]
As a consequence, performing the change of variable $\mu=U(\lambda)$ in 
Kenyon's formula~\eqref{eq:Keninv} for the inverse $(\tilde{\Ks}^{\text{crit}})^{-1}_{\bs,\ws}$ 
of $\tilde{\Ks}^{\text{crit}}$ gives:
{\small
\[
(\tilde{\Ks}^{\text{crit}})^{-1}_{\bs,\ws}=\frac{1}{4i\pi^2}
  \oint \tilde{f}_{\bs,\ws}(\lambda)\log(\lambda) d\lambda=\frac{1}{4i\pi^2}
\frac{\bar{\zeta}^2}{(1-|\zeta|^2)f_{\bs,\ws}(-\bar{\zeta}^{-1})}\oint 
f_{\bs,\ws}(\mu)\log\Bigl(\frac{\mu+\zeta}{1+\bar{\zeta}\mu}\Bigr)d\mu,
\]}
where the second contour of integration avoids a ray from $U^{-1}(0)=\zeta$ to $U^{-1}(\infty)=\bar{\zeta}^{-1}$. 
Finally, using the gauge equivalence relation~\eqref{equ:gauge_K}, we have that for every $\zeta\in\CC$, $|\zeta|<1$, 
\[
\frac{1}{4i\pi^2}\oint f_{\bs,\ws}(\mu)\log\Bigl(\frac{\mu+\zeta}{1+\bar{\zeta}\mu}\Bigr)d\mu=:\A^{\text{crit},\zeta}_{\bs,\ws}
\]
defines an inverse of the operator $\Ks^{\text{crit}}$ with fixed half-angles $(\alpha_T)_{T\in\T}$.
In the periodic case,
the complex number $\zeta$ can be thought of as parameterizing the
``northern hemisphere'' of the spectral curve, or its amoeba, and as such, plays the
role of the magnetic field.

When
performing the additional change of variable $e^{2iu}=\mu$,
the complex plane (deprived from 0) becomes an infinite cylinder, and
$\zeta$ and $\bar{\zeta}^{-1}$
are sent
to some $\bar{u}_0$ and $u_0$. The operator $\As^{\text{crit},\zeta}$ exactly corresponds to the limit of our formula~\eqref{eq:coeff_Kinv_u0_alt} for $\As^{u_0}$ when the
torus $\TT(q)$ degenerates to a cylinder as $q\to 0$.
Taking the limit in Formula~\eqref{eq:coeff_Kinv_u0} instead gives
the alternative expression for $(\Ks^{\text{crit}})^{-1}$ used in the proof of
\cite[Theorem~4.3]{Kenyon:crit} by integrating $f_{\bs,\ws}$ along a ray. 

This critical limit is also interesting from the point of view of the
immersion
of the graph $\GR$ in the plane.
Indeed, a \emph{minimal immersion} of~$\Gs$,
where all the faces of $\GR$ are drawn as
rhombi~\cite{BCdT:immersion}, can be obtained as the limit of the
immersion~$\Psi_{u_0}=\Psi^{(t)}_{u_0}$
described in Section~\ref{sec:circlepatterns},
when both $\tau$ and
$u_0$ go to $i\infty$. 
Indeed, as $q\to 0$ and the imaginary part of $u_0$ goes to infinity, we have
\begin{equation*}
  \frac{\theta'}{\theta}(t+\mapd(\fs))=-i+O(q),\quad
    \frac{\theta'}{\theta}(u_0-\alpha)=-\cot(u_0-\alpha)+O(q)=-i(1+2e^{2iu_0}e^{-2i\alpha})+o(e^{2iu_0}).
\end{equation*}

As a consequence, if we choose $\Xi$ in such way that
$\lim {\theta'(0)}\Xi(\bs) = 0$
for every black vertex $\bs$ and
$\lim {\theta'(0)}\Xi(\ws) = 2i$
for every white vertex $\ws$, we have
\begin{equation*}
  {\theta'(0)}\bigl(\Psi_{u_0}(\bs)-\Psi_{u_0}(\fs)\bigr) \sim -i2e^{2iu_0} e^{-2i\alpha}, \quad
  {\theta'(0)}\bigl(\Psi_{u_0}(\ws)-\Psi_{u_0}(\fs)\bigr) \sim -i2e^{2iu_0} e^{-2i\alpha}.
\end{equation*}

Therefore, the edges of $\GR$ corresponding to a train-track associated with a half-angle
$\alpha$ are drawn as unit vectors $\overline{e^{2i\alpha}}$, up to a global
similarity with a stretch factor $2e^{-2\Im u_0}$ tending to~0.

\subsection{Connection to known elliptic dimer models}
\label{sec:previous_genus1}

In the literature there are two instances of elliptic dimer models where a local formula is proved for an inverse Kasteleyn operator:
the dimer model on a bipartite graph arising from a $Z$-invariant Ising model
on an isoradial graph~\cite{BdTR2}, and the dimer model arising from the $Z$-Dirac operator introduced in~\cite{dT_Mass_Dirac}.
It is however not immediate to see that
these models are specific cases of those of this paper; we now explain these connections.
Note that since the massive Laplacian operator of~\cite{BdTR1}
is related to the $Z$-Dirac operator (see~\cite[Theorem~1]{dT_Mass_Dirac}), it
can also be related to the present work.

Consider an infinite, planar graph $\Gs$, not necessarily bipartite. From $\Gs$ one constructs two bipartite decorated
graphs: the \emph{double graph} $\GD$, see for example~\cite{Kenyon:crit} and the graph $\GQ$, see for example~\cite{WuLin}.
The double $\GD$ is obtained by taking the diagonals of the quad-graph $\Gs^\diamond$ and adding a white vertex at the crossing of the diagonals.
It is bipartite and has two kinds of black vertices corresponding to vertices of $\Gs$ and vertices of $\Gs^*$. The associated quad-graph $(\GD)^{\diamond}$ is obtained by
dividing the quadrangular faces of $\Gs^\diamond$ into four, see Figure~\ref{fig:GD}. 
The graph $\GQ$ is the dual graph of the superimposition of the quad-graph $\Gs^\diamond$ and of the double graph $\GD$. It has three kinds of faces containing in their interior either a vertex of $\Gs$, a vertex of $\Gs^*$, or a white vertex of $\GD$; the latter are
quadrangles whose pair of parallel edges correspond to primal and dual edges of $\Gs$. The quad-graph $(\GQ)^{\diamond}$ is that of $\GD$ with an additional quadrangular face for each edge of $\Gs^\diamond$, which should be thought of
as ``flat'', see Figure~\ref{fig:GQ} and the discussion below.

Each train-track of~$\Gs$ induces two train-tracks that are
 anti-parallel in~$\GD$ and make anti-parallel bigons in~$\GQ$, see Figures~\ref{fig:GD} and~\ref{fig:GQ}.
Therefore, if~$\Gs$ is an isoradial graph, \emph{i.e.}, if its train-tracks do not self-intersect
and two train-tracks never intersect more than once,
then the associated bipartite graphs~$\GD$ and~$\GQ$ are isoradial and minimal, respectively.
Moreover, an isoradial embedding of~$\Gs$ naturally induces an isoradial embedding of~$\GD$ and a
minimal immersion of~$\GQ$, as follows.
As proved by Kenyon-Schlenker~\cite{KeSchlenk}, an isoradial embedding is given by some
half-angle map~$\mapalpha$ on the oriented train-tracks of~$\Gs$, such that the half-angles associated to any given oriented 
train-track and to the same train-track with the opposite orientation differ by~$\frac{\pi}{2}$.
One can then simply define the induced half-angle maps~$\mapalpha^{\scriptscriptstyle{\mathrm{D}}}$
and~$\mapalpha^{\scriptscriptstyle{\mathrm{Q}}}$ by associating to each oriented train-track of~$\GD$ and~$\GQ$
the half-angle of the unique oriented train-track of~$\Gs$ it is parallel to. This is illustrated in Figures~\ref{fig:GD} and~\ref{fig:GQ}.
If~$\mapalpha$ defines an isoradial embedding of~$\Gs$, then so does~$\mapalpha^{\scriptscriptstyle{\mathrm{D}}}$
for~$\GD$ (in particular, it belongs to~$X_{\GD}$), while~$\mapalpha^{\scriptscriptstyle{\mathrm{Q}}}$ belongs to~$X_{\GQ}$
and therefore defines a minimal immersion of the minimal graph~$\GQ$.
In this minimal immersion, the rhombi corresponding to the edges of~$\Gs^\diamond$ are degenerate,
or ``flat": this is the reason why it is not an isoradial embedding.

\begin{figure}[ht]
  \centering
  \begin{overpic}[width=5cm]{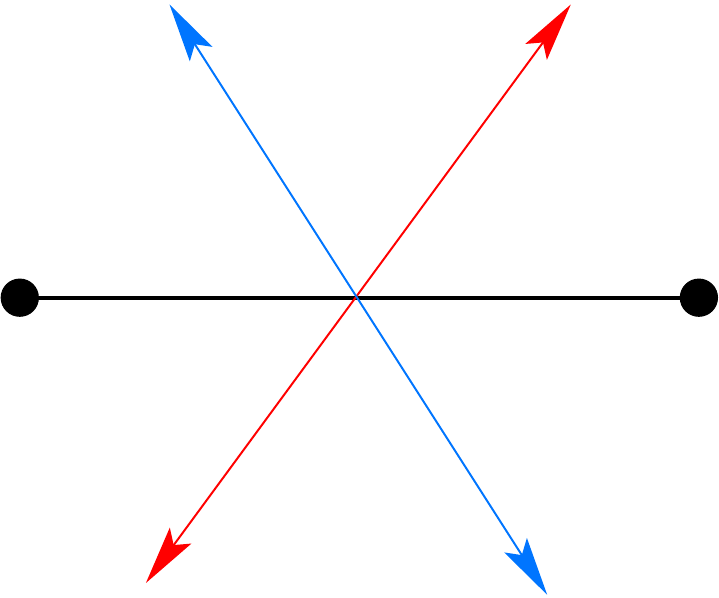}
  \put(-5,40){\scriptsize $\vs$}
    \put(80,85){\scriptsize $\alpha$}
  \put(20,85){\scriptsize $\beta$}
  \put(70,-5){\scriptsize $\beta+\frac{\pi}{2}$}
    \put(10,-5){\scriptsize $\alpha+\frac{\pi}{2}$}
  \end{overpic}
  \hspace{2cm}
  \begin{overpic}[width=7cm]{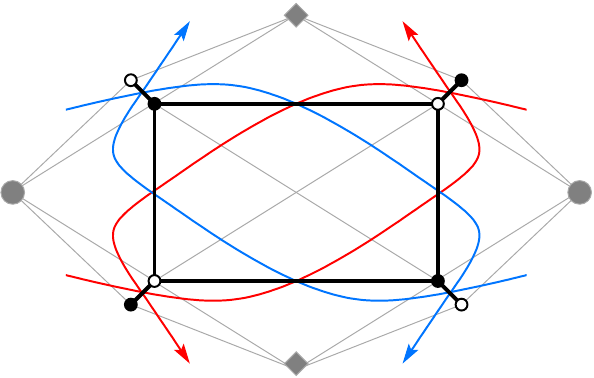}
  \put(-4,30){\scriptsize $\vs$}
  \put(48,64){\scriptsize $\fs$}
  \put(27,18){\scriptsize $\ws$}
  \put(76,18){\scriptsize $\bs_1$}
  \put(27,41){\scriptsize $\bs_2$}
  \put(18,8){\scriptsize $\bs_3$}
  \put(66,64){\scriptsize $\alpha$}
  \put(31,-3){\scriptsize $\alpha+\frac{\pi}{2}$}
  \put(32,64){\scriptsize $\beta$}
  \put(68,-3){\scriptsize $\beta+\frac{\pi}{2}$}
  \end{overpic}
  \medskip
  \caption{Left: an edge of~$\Gs$ with its two adjacent train-tracks.
  Right: the corresponding pieces of~$\GQ$ (black lines and white/black vertices),
  of the quad-graph $(\GQ)^\diamond$ (grey lines and vertices), and of the four
  adjacent  train-tracks (red and blue lines).}
  \label{fig:GQ}
\end{figure}

\paragraph{Connection to the dimer model on the graph $\GQ$ arising from the Ising model}
Consider a half-angle map $\mapalpha^{\scriptscriptstyle{\mathrm{Q}}}\in X_{\GQ}$ as above.
Following Section~\ref{sec:abel_map}, let us compute the discrete Abel map on vertices of the quad-graph $(\GQ)^{\diamond}$,
taking as reference point a vertex $\vs_0$ of $\Gs$;
note that $\vs_0$ is indeed a vertex
of~$(\GQ)^\diamond$, at it should (recall Section~\ref{sec:abel_map}).
Then, using the notation of Figure~\ref{fig:GQ},
and recalling that $\mapd$ is defined as an element of $\RR/\pi\ZZ$, we have
the following equalities modulo~$\pi$
\[\textstyle
\forall\,\vs\in\Vs,\ \mapd(\vs)=0,\quad \forall\,\fs\in\Vs^*,\ \mapd(\fs)=\frac{\pi}{2},\quad
\mapd(\bs_1)=\beta+\frac{\pi}{2},\quad \mapd(\bs_2)=\beta, \quad \mapd(\bs_3)=\alpha\,.
\]
In other words, for every black vertex $\bs$ of $\GQ$, $\mapd(\bs)$ is given by the half-angle of the train-track
crossing the edge $\vs\bs$ of~$(\GQ)^\diamond$,
where $\vs$ is the vertex of $\Gs$ adjacent to $\bs$ in~$(\GQ)^\diamond$.

Returning to Section~\ref{sec:Kast_def}, the weight function $\nu^{(t)}$ of the corresponding elliptic operator~$\K[t]$ is:
\[
\nu^{(t)}_{\ws\xs}=
\begin{cases}
  \hfill
\left|\frac{\theta(\beta-\alpha)}{\theta(t+\frac{\pi}{2}-(\beta-\alpha))\theta(t+\frac{\pi}{2})}\right| 
  \hfill
& \text{if $\xs=\bs_1$}\\
\hfill
\left|\frac{\theta(\alpha+\frac{\pi}{2}-\beta)}{\theta(t)\theta(t+\frac{\pi}{2}+(\beta-\alpha))}
\right|
  \hfill
& \text{if $\xs=\bs_2$}\\
\hfill
\left|
\frac{\theta(\frac{\pi}{2})}{\theta(t)\theta(t-\frac{\pi}{2})}
\right|
  \hfill
& \text{if $\xs=\bs_3$}.
\end{cases}
\]

We now turn to the weight function $\nuQ$ of the dimer model on $\GQ$ arising from the Ising model, see for example~\cite{BdTR2}, which is independent of $t$.
We refer to~\cite[Chap 2]{Lawden} for the definition of Jacobi's elliptic (trigonometric) functions. We here use the functions $\sn$ and $\cn$ which are the elliptic analogues of the trigonometric functions $\sin$ and $\cos$. Let us recall the following relations between the parameters of Jacobi's elliptic and theta functions: $k=\frac{\theta_2^2(0|q)}{\theta_3^2(0|q)},\, k'=\frac{\theta_4^2(0|q)}{\theta_3^2(0|q)},\, K=\frac{\pi}{2}\theta_3^2(0),\, iK'=\tau K$; we simply denote $\sn(v|k)$ as $\sn(v)$ and similarly for $\cn$. Using the notation of Figure~\ref{fig:GQ}, the weight function $\nuQ$ is given by
\[
\nuQ_{\ws\xs}=
\begin{cases}
\sn(\frac{2K}{\pi}(\beta-\alpha))& \text{if $\xs=\bs_1$}\\
\cn(\frac{2K}{\pi}(\beta-\alpha))& \text{if $\xs=\bs_2$}\\
\hfill 1\hfill& \text{if $\xs=\bs_3$}.
\end{cases}
\]
Then, we have:
\begin{prop}
Suppose that $t=\frac{\pi}{2}+\frac{\pi}{2}\tau$, then the elliptic dimer models on the bipartite graph $\GQ$ with weight function $\nu^{(\frac{\pi}{2}+\frac{\pi}{2}\tau)}$ and $\nuQ$ are gauge equivalent.
\end{prop}
\begin{proof}
Let us explicitly compute $\nu^{(\frac{\pi}{2}+\frac{\pi}{2}\tau)}$. Using the following identities, see \cite[1.3.6-13.9]{Lawden} and \cite[2.1.1-2.1.3]{Lawden},
\begin{align*}
&\textstyle |\theta(u+\frac{\pi}{2})|=|\theta_2(u)|,\quad |\theta(u+\frac{\pi}{2}\tau)|=q^{-\frac{1}{4}}|\theta_4(u)|,\quad |\theta(u+\frac{\pi}{2}+\frac{\pi}{2}\tau)|=q^{-\frac{1}{4}}|\theta_3(u)|,\\
&\textstyle\sn(\frac{2K}{\pi}u)=\frac{\theta_3(0)}{\theta_2(0)}\frac{\theta(u)}{\theta_4(u)},\quad 
\cn(\frac{2K}{\pi}u)=\frac{\theta_4(0)}{\theta_2(0)}\frac{\theta_2(u)}{\theta_4(u)},
\end{align*}
we obtain that $\nu^{(\frac{\pi}{2}+\frac{\pi}{2}\tau)}=q^{\frac{1}{2}}\frac{\theta_2(0)}{\theta_4(0)\theta_3(0)}\nuQ$. This immediately shows that the two weight functions are gauge equivalent.\qedhere
\end{proof}

\begin{rem}
By explicitly computing the function $g_{\bs,\ws}^{(\frac{\pi}{2}+\frac{\pi}{2}\tau)}(u)$ in the case of $\GQ$, we recover the local expression for the inverse Kasteleyn operator~\cite[Theorem 37]{BdTR2}. 
\end{rem}

\begin{figure}[ht]
  \centering
   \begin{overpic}[width=4cm]{edge.pdf}
  \put(103,40){\scriptsize $\vs$}
    \put(80,85){\scriptsize $\alpha$}
  \put(18,85){\scriptsize $\beta$}
  \put(70,-5){\scriptsize $\beta+\frac{\pi}{2}$}
    \put(10,-5){\scriptsize $\alpha+\frac{\pi}{2}$}
  \end{overpic}
  \hspace{2cm}
  \begin{overpic}[width=5cm]{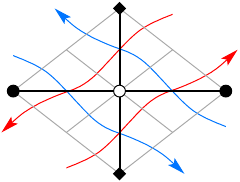}
  \put(100,36){\scriptsize $\vs$}
  \put(52,74){\scriptsize $\fs$}
  \put(54,31){\scriptsize $\ws$}
  \put(102,54){\scriptsize $\alpha$}
  \put(-7,13){\scriptsize $\alpha+\frac{\pi}{2}$}
  \put(18,73){\scriptsize $\beta$}
  \put(76,-3){\scriptsize $\beta+\frac{\pi}{2}$}
  \end{overpic}
  \caption{Left: an edge of~$\Gs$ with its two adjacent train-tracks.
  Right: the corresponding pieces of~$\GD$ (black lines and white/black vertices),
  of the quad-graph $(\GD)^\diamond$ (grey lines), and of the four
  adjacent  train-tracks (red and blue lines).}
  \label{fig:GD}
\end{figure}

\paragraph{Connection to the $Z$-Dirac operator~\cite{dT_Mass_Dirac} on the double graph $\GD$.}
Consider a half-angle map $\mapalpha^{\scriptscriptstyle{\mathrm{D}}}\in X_{\GD}$
as above.
We compute the discrete Abel map $\mapd$ on the quad-graph $(\GD)^\diamond$,
choose once again a vertex $\vs_0$ of $\Gs$ as reference point.
Using the notation of Figure~\ref{fig:GD}, we have the following equalities mod $\pi$,
where~$\WD$ denotes the set of white vertices of~$\GD$:
\[
\forall\,\vs\in\Vs,\ \mapd(\vs)=0, \quad \forall\,\fs\in\Vs^*,\ \mapd(\fs)=\frac{\pi}{2},\quad
\forall\,\ws\in\WD,\ \mapd(\ws)=-(\alpha+\beta)\,.
\]
The weight function $\nu^{(t)}$ of the corresponding elliptic operator~$\K[t]$ is:
\[
\nu^{(t)}_{\ws\xs}=
\begin{cases}
  \hfill
\Bigl|\frac{\theta(\beta-\alpha)}{\theta(t-\beta)\theta(t-\alpha)}\Bigr| 
  \hfill
&
\text{if $\xs=\vs\in\Vs$}\\
  \hfill
\left|\frac{\theta(\alpha+\frac{\pi}{2}-\beta)}{\theta(t-\alpha)\theta(t-\beta+\frac{\pi}{2})}
\right| 
  \hfill
& \text{if $\xs=\fs\in\Vs^*$}.
\end{cases}
\]

We now turn to the $Z$-Dirac operator of~\cite{dT_Mass_Dirac}. We here use the function $\sc$ and $\dn$ which are the elliptic analogues of the trigonometric functions $\tan$ and $1$, see~\cite[Chap 2]{Lawden}.
Using the notation of Figure~\ref{fig:GD} for the half-angles, setting
$u=\frac{2K}{\pi}2s$ in~\cite[Example 14]{dT_Mass_Dirac}, $s\in\RR$, the weight
function $\nuD{s}$ corresponding to the $Z$-Dirac operator on $\GD$ is:
\[
  \nuD{s}_{\ws\xs}=
\begin{cases}\textstyle 
[\sc(\frac{2K}{\pi}(\beta-\alpha))\dn(\frac{2K}{\pi}(s-\alpha))\dn(\frac{2K}{\pi}(s-\beta))]^{\frac{1}{2}} &
\text{if $\xs=\vs\in\Vs$}\\
\textstyle 
k'[\sc(\frac{2K}{\pi}(\alpha+\frac{\pi}{2}-\beta))[\nd(\frac{2K}{\pi}(s-\beta))
\nd(\frac{2K}{\pi}(s-(\alpha+\frac{\pi}{2}))]^{\frac{1}{2}}&\text{if $\xs=\fs\in\Vs^*$}.
\end{cases}
\]

Then we have:

\begin{prop}
Suppose that $t\in \RR+\frac{\pi}{2}\tau$. Then, the
elliptic dimer models on the double graph $\GD$ with weight functions
$\nu^{(t)}$ and $\nuD{t-\frac{\pi}{2}\tau}$ are gauge equivalent.
\end{prop}
\begin{proof}
The proof amounts to showing that the alternate product of the weight functions
$\nu^{(t)}$ and  $\nuD{t-\frac{\pi}{2}\tau}$ are equal for all faces of $\GD$. Faces of $\GD$ are quadrangles consisting of two black vertices corresponding to a primal and a dual vertex of $\Gs$, and two white vertices. The proof is a rather straightforward explicit computation, so that we do not write out the details; the key identities used are
\begin{align*}
\textstyle
\sc\bigl(\frac{2K}{\pi}u\bigr)&=(k')^{-\frac{1}{2}}\frac{\theta(u)}{\theta(u+\frac{\pi}{2})},\quad\cite[\text{2.1.1--2.1.7}]{Lawden}\\
\sc(u+iK')&=i\nd(u),\quad \cite[\text{2.2.17--2.2.19}]{Lawden}.\qedhere
\end{align*}
\end{proof}

\begin{rem}
Note that it is not immediate to see that the local expression for the inverse massive Dirac operator obtained in~\cite[Corollary 27]{dT_Mass_Dirac} and the local expression stemming from Theorem~\ref{thm:K_inverse_family} of this paper are indeed the same. This can be done using identities relating Jacobi's elliptic and theta functions. 
\end{rem}

\bibliographystyle{alpha}
\bibliography{ellipt_iso}

\end{document}

%% file: fig_domaineD_intro.pdf_tex
\begingroup%
  \makeatletter%
  \providecommand\color[2][]{%
    \errmessage{(Inkscape) Color is used for the text in Inkscape, but the package 'color.sty' is not loaded}%
    \renewcommand\color[2][]{}%
  }%
  \providecommand\transparent[1]{%
    \errmessage{(Inkscape) Transparency is used (non-zero) for the text in Inkscape, but the package 'transparent.sty' is not loaded}%
    \renewcommand\transparent[1]{}%
  }%
  \providecommand\rotatebox[2]{#2}%
  \newcommand*\fsize{\dimexpr\f@size pt\relax}%
  \newcommand*\lineheight[1]{\fontsize{\fsize}{#1\fsize}\selectfont}%
  \ifx\svgwidth\undefined%
    \setlength{\unitlength}{292.4147004bp}%
    \ifx\svgscale\undefined%
      \relax%
    \else%
      \setlength{\unitlength}{\unitlength * \real{\svgscale}}%
    \fi%
  \else%
    \setlength{\unitlength}{\svgwidth}%
  \fi%
  \global\let\svgwidth\undefined%
  \global\let\svgscale\undefined%
  \makeatother%
  \begin{picture}(1,0.57831466)%
    \lineheight{1}%
    \setlength\tabcolsep{0pt}%
    \put(0,0){\includegraphics[width=\unitlength,page=1]{fig_domaineD_intro.pdf}}%
    \put(0.7788933,0.53200916){\color[rgb]{0,0,0}\makebox(0,0)[lt]{\lineheight{1.25}\smash{\begin{tabular}[t]{l}\scriptsize{Case 1}\end{tabular}}}}%
    \put(0.27105255,0.39350725){\color[rgb]{0,0,0}\makebox(0,0)[lt]{\lineheight{1.25}\smash{\begin{tabular}[t]{l}\scriptsize{Case 2}\end{tabular}}}}%
    \put(0.71220691,0.15394229){\color[rgb]{0,0,0}\makebox(0,0)[lt]{\lineheight{1.25}\smash{\begin{tabular}[t]{l}\scriptsize{Case 3}\end{tabular}}}}%
    \put(0.70707728,0.3457434){\color[rgb]{0,0,0}\makebox(0,0)[lt]{\lineheight{1.25}\smash{\begin{tabular}[t]{l}$D$\end{tabular}}}}%
    \put(0,0){\includegraphics[width=\unitlength,page=2]{fig_domaineD_intro.pdf}}%
    \put(0.08820111,0.23948691){\color[rgb]{0,0,0}\makebox(0,0)[lt]{\lineheight{1.25}\smash{\begin{tabular}[t]{l}$0$\end{tabular}}}}%
    \put(0.05747763,0.47260571){\color[rgb]{0,0,0}\makebox(0,0)[lt]{\lineheight{1.25}\smash{\begin{tabular}[t]{l}$\frac{\pi\tau}{2}$\end{tabular}}}}%
    \put(0.90382356,0.24461663){\color[rgb]{0,0,0}\makebox(0,0)[lt]{\lineheight{1.25}\smash{\begin{tabular}[t]{l}$\pi$\end{tabular}}}}%
    \put(-0.00151393,0.01093261){\color[rgb]{0,0,0}\makebox(0,0)[lt]{\lineheight{1.25}\smash{\begin{tabular}[t]{l}$-\frac{\pi\tau}{2}$\end{tabular}}}}%
  \end{picture}%
\endgroup%

%% file: fig_graph_vec.pdf_tex
\begingroup%
  \makeatletter%
  \providecommand\color[2][]{%
    \errmessage{(Inkscape) Color is used for the text in Inkscape, but the package 'color.sty' is not loaded}%
    \renewcommand\color[2][]{}%
  }%
  \providecommand\transparent[1]{%
    \errmessage{(Inkscape) Transparency is used (non-zero) for the text in Inkscape, but the package 'transparent.sty' is not loaded}%
    \renewcommand\transparent[1]{}%
  }%
  \providecommand\rotatebox[2]{#2}%
  \newcommand*\fsize{\dimexpr\f@size pt\relax}%
  \newcommand*\lineheight[1]{\fontsize{\fsize}{#1\fsize}\selectfont}%
  \ifx\svgwidth\undefined%
    \setlength{\unitlength}{585.26514519bp}%
    \ifx\svgscale\undefined%
      \relax%
    \else%
      \setlength{\unitlength}{\unitlength * \real{\svgscale}}%
    \fi%
  \else%
    \setlength{\unitlength}{\svgwidth}%
  \fi%
  \global\let\svgwidth\undefined%
  \global\let\svgscale\undefined%
  \makeatother%
  \begin{picture}(1,1.05186181)%
    \lineheight{1}%
    \setlength\tabcolsep{0pt}%
    \put(0,0){\includegraphics[width=\unitlength,page=1]{fig_graph_vec.pdf}}%
    \put(0.74325287,1.01227651){\color[rgb]{0,0,0}\makebox(0,0)[lt]{\lineheight{1.25}\smash{\begin{tabular}[t]{l}$T_5$\end{tabular}}}}%
    \put(0.81964077,0.92196042){\color[rgb]{0,0,0}\makebox(0,0)[lt]{\lineheight{1.25}\smash{\begin{tabular}[t]{l}$T_4$\end{tabular}}}}%
    \put(0.2096608,0.74255456){\color[rgb]{0,0,0}\makebox(0,0)[lt]{\lineheight{1.25}\smash{\begin{tabular}[t]{l}$T_6$\end{tabular}}}}%
    \put(0.06613613,0.5887781){\color[rgb]{0,0,0}\makebox(0,0)[lt]{\lineheight{1.25}\smash{\begin{tabular}[t]{l}$T_1$\end{tabular}}}}%
    \put(0.8201411,0.0896178){\color[rgb]{0,0,0}\makebox(0,0)[lt]{\lineheight{1.25}\smash{\begin{tabular}[t]{l}$T_3$\end{tabular}}}}%
    \put(0.691994,0.01272957){\color[rgb]{0,0,0}\makebox(0,0)[lt]{\lineheight{1.25}\smash{\begin{tabular}[t]{l}$T_2$\end{tabular}}}}%
    \put(0,0){\includegraphics[width=\unitlength,page=2]{fig_graph_vec.pdf}}%
  \end{picture}%
\endgroup%

%% file: fig_notation_tt_face.pdf_tex
\begingroup%
  \makeatletter%
  \providecommand\color[2][]{%
    \errmessage{(Inkscape) Color is used for the text in Inkscape, but the package 'color.sty' is not loaded}%
    \renewcommand\color[2][]{}%
  }%
  \providecommand\transparent[1]{%
    \errmessage{(Inkscape) Transparency is used (non-zero) for the text in Inkscape, but the package 'transparent.sty' is not loaded}%
    \renewcommand\transparent[1]{}%
  }%
  \providecommand\rotatebox[2]{#2}%
  \newcommand*\fsize{\dimexpr\f@size pt\relax}%
  \newcommand*\lineheight[1]{\fontsize{\fsize}{#1\fsize}\selectfont}%
  \ifx\svgwidth\undefined%
    \setlength{\unitlength}{621.1426381bp}%
    \ifx\svgscale\undefined%
      \relax%
    \else%
      \setlength{\unitlength}{\unitlength * \real{\svgscale}}%
    \fi%
  \else%
    \setlength{\unitlength}{\svgwidth}%
  \fi%
  \global\let\svgwidth\undefined%
  \global\let\svgscale\undefined%
  \makeatother%
  \begin{picture}(1,0.84281296)%
    \lineheight{1}%
    \setlength\tabcolsep{0pt}%
    \put(0,0){\includegraphics[width=\unitlength,page=1]{fig_notation_tt_face.pdf}}%
    \put(0.51256344,0.14094381){\color[rgb]{0,0,0}\makebox(0,0)[lt]{\lineheight{1.25}\smash{\begin{tabular}[t]{l}$\ws_1$\end{tabular}}}}%
    \put(0.66711718,0.54412667){\color[rgb]{0,0,0}\makebox(0,0)[lt]{\lineheight{1.25}\smash{\begin{tabular}[t]{l}$\ws_2$\end{tabular}}}}%
    \put(0.6691359,0.28573222){\color[rgb]{0,0,0}\makebox(0,0)[lt]{\lineheight{1.25}\smash{\begin{tabular}[t]{l}$\bs_1$\end{tabular}}}}%
    \put(0.49578405,0.72041469){\color[rgb]{0,0,0}\makebox(0,0)[lt]{\lineheight{1.25}\smash{\begin{tabular}[t]{l}$\bs_2$\end{tabular}}}}%
    \put(0.04170425,0.33403004){\color[rgb]{0,0,0}\makebox(0,0)[lt]{\lineheight{1.25}\smash{\begin{tabular}[t]{l}$\ws_n$\end{tabular}}}}%
    \put(0.23208551,0.14325256){\color[rgb]{0,0,0}\makebox(0,0)[lt]{\lineheight{1.25}\smash{\begin{tabular}[t]{l}$\bs_n$\end{tabular}}}}%
    \put(0,0){\includegraphics[width=\unitlength,page=2]{fig_notation_tt_face.pdf}}%
    \put(0.56636817,0.80551412){\color[rgb]{0,0,0}\makebox(0,0)[lt]{\lineheight{1.25}\smash{\begin{tabular}[t]{l}$\beta_2$\end{tabular}}}}%
    \put(0,0){\includegraphics[width=\unitlength,page=3]{fig_notation_tt_face.pdf}}%
    \put(0.76189214,0.21569934){\color[rgb]{0,0,0}\makebox(0,0)[lt]{\lineheight{1.25}\smash{\begin{tabular}[t]{l}$\beta_1$\end{tabular}}}}%
    \put(0,0){\includegraphics[width=\unitlength,page=4]{fig_notation_tt_face.pdf}}%
    \put(0.61466623,0.02791502){\color[rgb]{0,0,0}\makebox(0,0)[lt]{\lineheight{1.25}\smash{\begin{tabular}[t]{l}$\alpha_1$\end{tabular}}}}%
    \put(0,0){\includegraphics[width=\unitlength,page=5]{fig_notation_tt_face.pdf}}%
    \put(-0.00400911,0.22777433){\color[rgb]{0,0,0}\makebox(0,0)[lt]{\lineheight{1.25}\smash{\begin{tabular}[t]{l}$\alpha_n$\end{tabular}}}}%
    \put(0.37317578,0.07138333){\color[rgb]{0,0,0}\makebox(0,0)[lt]{\lineheight{1.25}\smash{\begin{tabular}[t]{l}$\fs'_n$\end{tabular}}}}%
    \put(0.16790893,0.2114475){\color[rgb]{0,0,0}\makebox(0,0)[lt]{\lineheight{1.25}\smash{\begin{tabular}[t]{l}$\fs_n$\end{tabular}}}}%
    \put(0.63640008,0.189714){\color[rgb]{0,0,0}\makebox(0,0)[lt]{\lineheight{1.25}\smash{\begin{tabular}[t]{l}$\fs_1$\end{tabular}}}}%
    \put(0.72575131,0.42395924){\color[rgb]{0,0,0}\makebox(0,0)[lt]{\lineheight{1.25}\smash{\begin{tabular}[t]{l}$\fs'_1$\end{tabular}}}}%
    \put(0.39973961,0.43120382){\color[rgb]{0,0,0}\makebox(0,0)[lt]{\lineheight{1.25}\smash{\begin{tabular}[t]{l}$\fs$\end{tabular}}}}%
    \put(0,0){\includegraphics[width=\unitlength,page=6]{fig_notation_tt_face.pdf}}%
    \put(0.73541079,0.60266215){\color[rgb]{0,0,0}\makebox(0,0)[lt]{\lineheight{1.25}\smash{\begin{tabular}[t]{l}$\alpha_2$\end{tabular}}}}%
    \put(0.17273878,0.05930898){\color[rgb]{0,0,0}\makebox(0,0)[lt]{\lineheight{1.25}\smash{\begin{tabular}[t]{l}$\beta_n$\end{tabular}}}}%
  \end{picture}%
\endgroup%

%% file: fig_sep_zeros_poles_2.pdf_tex
\begingroup%
  \makeatletter%
  \providecommand\color[2][]{%
    \errmessage{(Inkscape) Color is used for the text in Inkscape, but the package 'color.sty' is not loaded}%
    \renewcommand\color[2][]{}%
  }%
  \providecommand\transparent[1]{%
    \errmessage{(Inkscape) Transparency is used (non-zero) for the text in Inkscape, but the package 'transparent.sty' is not loaded}%
    \renewcommand\transparent[1]{}%
  }%
  \providecommand\rotatebox[2]{#2}%
  \newcommand*\fsize{\dimexpr\f@size pt\relax}%
  \newcommand*\lineheight[1]{\fontsize{\fsize}{#1\fsize}\selectfont}%
  \ifx\svgwidth\undefined%
    \setlength{\unitlength}{103.50004927bp}%
    \ifx\svgscale\undefined%
      \relax%
    \else%
      \setlength{\unitlength}{\unitlength * \real{\svgscale}}%
    \fi%
  \else%
    \setlength{\unitlength}{\svgwidth}%
  \fi%
  \global\let\svgwidth\undefined%
  \global\let\svgscale\undefined%
  \makeatother%
  \begin{picture}(1,0.90655528)%
    \lineheight{1}%
    \setlength\tabcolsep{0pt}%
    \put(0,0){\includegraphics[width=\unitlength,page=1]{fig_sep_zeros_poles_2.pdf}}%
    \put(0.66886982,0.57148921){\color[rgb]{0,0,0}\makebox(0,0)[lt]{\lineheight{1.25}\smash{\begin{tabular}[t]{l}$\ys$\end{tabular}}}}%
    \put(0,0){\includegraphics[width=\unitlength,page=2]{fig_sep_zeros_poles_2.pdf}}%
    \put(0.14899089,0.53954567){\color[rgb]{0,0,0}\makebox(0,0)[lt]{\lineheight{1.25}\smash{\begin{tabular}[t]{l}$\Pi$\end{tabular}}}}%
    \put(0,0){\includegraphics[width=\unitlength,page=3]{fig_sep_zeros_poles_2.pdf}}%
    \put(0.12097473,0.08949722){\color[rgb]{0,0,0}\makebox(0,0)[lt]{\lineheight{1.25}\smash{\begin{tabular}[t]{l}$S_0$\end{tabular}}}}%
    \put(0.71376466,0.12200167){\color[rgb]{0,0,0}\makebox(0,0)[lt]{\lineheight{1.25}\smash{\begin{tabular}[t]{l}$S_1$\end{tabular}}}}%
    \put(0,0){\includegraphics[width=\unitlength,page=4]{fig_sep_zeros_poles_2.pdf}}%
    \put(0.38901372,0.42935106){\color[rgb]{0,0,0}\makebox(0,0)[lt]{\lineheight{1.25}\smash{\begin{tabular}[t]{l}$S_1^-$\end{tabular}}}}%
    \put(0,0){\includegraphics[width=\unitlength,page=5]{fig_sep_zeros_poles_2.pdf}}%
    \put(0.53795598,0.33923273){\color[rgb]{0,0,0}\makebox(0,0)[lt]{\lineheight{1.25}\smash{\begin{tabular}[t]{l}$\es$\end{tabular}}}}%
    \put(0,0){\includegraphics[width=\unitlength,page=6]{fig_sep_zeros_poles_2.pdf}}%
    \put(0.47462557,0.00794593){\color[rgb]{0,0,0}\makebox(0,0)[lt]{\lineheight{1.25}\smash{\begin{tabular}[t]{l}$T$\end{tabular}}}}%
    \put(0.2960294,0.74259255){\color[rgb]{0,0,0}\makebox(0,0)[lt]{\lineheight{1.25}\smash{\begin{tabular}[t]{l}tail\\of $S_1$\end{tabular}}}}%
    \put(0.49768499,0.80444302){\color[rgb]{0,0,0}\makebox(0,0)[lt]{\lineheight{1.25}\smash{\begin{tabular}[t]{l}tail of $T$\end{tabular}}}}%
    \put(0.728487,0.3050058){\color[rgb]{0,0,0}\makebox(0,0)[lt]{\lineheight{1.25}\smash{\begin{tabular}[t]{l}head\\of $T$\end{tabular}}}}%
    \put(0.074437,0.55487182){\color[rgb]{0,0,0}\makebox(0,0)[lt]{\lineheight{1.25}\smash{\begin{tabular}[t]{l}$\xs$\end{tabular}}}}%
    \put(0.82302928,0.65279247){\color[rgb]{0,0,0}\makebox(0,0)[lt]{\lineheight{1.25}\smash{\begin{tabular}[t]{l}admissible positions for\\the entry point of $T$\\\end{tabular}}}}%
    \put(0,0){\includegraphics[width=\unitlength,page=7]{fig_sep_zeros_poles_2.pdf}}%
    \put(0.4816372,0.68187775){\color[rgb]{0,0,0}\makebox(0,0)[lt]{\lineheight{1.25}\smash{\begin{tabular}[t]{l}\textcolor{red}{?}\end{tabular}}}}%
    \put(0.36479244,0.20104235){\color[rgb]{0,0,0}\makebox(0,0)[lt]{\lineheight{1.25}\smash{\begin{tabular}[t]{l}$R$\end{tabular}}}}%
    \put(0.4591471,0.427416){\color[rgb]{0,0,0}\makebox(0,0)[lt]{\lineheight{1.25}\smash{\begin{tabular}[t]{l}$T^+$\end{tabular}}}}%
    \put(0.60536007,0.87442488){\color[rgb]{0,0,0}\makebox(0,0)[lt]{\lineheight{1.25}\smash{\begin{tabular}[t]{l}$\partial B$\end{tabular}}}}%
  \end{picture}%
\endgroup%

%% file: fig_domaineD.pdf_tex
\begingroup%
  \makeatletter%
  \providecommand\color[2][]{%
    \errmessage{(Inkscape) Color is used for the text in Inkscape, but the package 'color.sty' is not loaded}%
    \renewcommand\color[2][]{}%
  }%
  \providecommand\transparent[1]{%
    \errmessage{(Inkscape) Transparency is used (non-zero) for the text in Inkscape, but the package 'transparent.sty' is not loaded}%
    \renewcommand\transparent[1]{}%
  }%
  \providecommand\rotatebox[2]{#2}%
  \newcommand*\fsize{\dimexpr\f@size pt\relax}%
  \newcommand*\lineheight[1]{\fontsize{\fsize}{#1\fsize}\selectfont}%
  \ifx\svgwidth\undefined%
    \setlength{\unitlength}{264.24496465bp}%
    \ifx\svgscale\undefined%
      \relax%
    \else%
      \setlength{\unitlength}{\unitlength * \real{\svgscale}}%
    \fi%
  \else%
    \setlength{\unitlength}{\svgwidth}%
  \fi%
  \global\let\svgwidth\undefined%
  \global\let\svgscale\undefined%
  \makeatother%
  \begin{picture}(1,0.61073191)%
    \lineheight{1}%
    \setlength\tabcolsep{0pt}%
    \put(0,0){\includegraphics[width=\unitlength,page=1]{fig_domaineD.pdf}}%
    \put(0.83865649,0.55949003){\color[rgb]{0,0,0}\makebox(0,0)[lt]{\lineheight{1.25}\smash{\begin{tabular}[t]{l}\scriptsize{Case 1}\end{tabular}}}}%
    \put(0.27667757,0.40622317){\color[rgb]{0,0,0}\makebox(0,0)[lt]{\lineheight{1.25}\smash{\begin{tabular}[t]{l}\scriptsize{Case 2}\end{tabular}}}}%
    \put(0.76486103,0.14111949){\color[rgb]{0,0,0}\makebox(0,0)[lt]{\lineheight{1.25}\smash{\begin{tabular}[t]{l}\scriptsize{Case 3}\end{tabular}}}}%
    \put(0.75918455,0.35336748){\color[rgb]{0,0,0}\makebox(0,0)[lt]{\lineheight{1.25}\smash{\begin{tabular}[t]{l}$D$\end{tabular}}}}%
    \put(0,0){\includegraphics[width=\unitlength,page=2]{fig_domaineD.pdf}}%
    \put(-0.00288456,0.24847582){\color[rgb]{0,0,0}\makebox(0,0)[lt]{\lineheight{1.25}\smash{\begin{tabular}[t]{l}$C_0$\end{tabular}}}}%
    \put(0,0){\includegraphics[width=\unitlength,page=3]{fig_domaineD.pdf}}%
    \put(-0.00288456,0.50108239){\color[rgb]{0,0,0}\makebox(0,0)[lt]{\lineheight{1.25}\smash{\begin{tabular}[t]{l}$C_1$\end{tabular}}}}%
  \end{picture}%
\endgroup%

%% file: graph_simple.pdf_tex
\begingroup%
  \makeatletter%
  \providecommand\color[2][]{%
    \errmessage{(Inkscape) Color is used for the text in Inkscape, but the package 'color.sty' is not loaded}%
    \renewcommand\color[2][]{}%
  }%
  \providecommand\transparent[1]{%
    \errmessage{(Inkscape) Transparency is used (non-zero) for the text in Inkscape, but the package 'transparent.sty' is not loaded}%
    \renewcommand\transparent[1]{}%
  }%
  \providecommand\rotatebox[2]{#2}%
  \newcommand*\fsize{\dimexpr\f@size pt\relax}%
  \newcommand*\lineheight[1]{\fontsize{\fsize}{#1\fsize}\selectfont}%
  \ifx\svgwidth\undefined%
    \setlength{\unitlength}{802.60857538bp}%
    \ifx\svgscale\undefined%
      \relax%
    \else%
      \setlength{\unitlength}{\unitlength * \real{\svgscale}}%
    \fi%
  \else%
    \setlength{\unitlength}{\svgwidth}%
  \fi%
  \global\let\svgwidth\undefined%
  \global\let\svgscale\undefined%
  \makeatother%
  \begin{picture}(1,0.74941856)%
    \lineheight{1}%
    \setlength\tabcolsep{0pt}%
    \put(0,0){\includegraphics[width=\unitlength,page=1]{graph_simple.pdf}}%
    \put(0.70096207,0.13246308){\color[rgb]{0,0,0}\makebox(0,0)[lt]{\lineheight{1.25}\smash{\begin{tabular}[t]{l}$\gamma_x$\end{tabular}}}}%
    \put(0,0){\includegraphics[width=\unitlength,page=2]{graph_simple.pdf}}%
    \put(0.18535449,0.66862141){\color[rgb]{0,0,0}\makebox(0,0)[lt]{\lineheight{1.25}\smash{\begin{tabular}[t]{l}$\gamma_y$\end{tabular}}}}%
    \put(0,0){\includegraphics[width=\unitlength,page=3]{graph_simple.pdf}}%
  \end{picture}%
\endgroup%

%% file: newton_simple.pdf_tex
\begingroup%
  \makeatletter%
  \providecommand\color[2][]{%
    \errmessage{(Inkscape) Color is used for the text in Inkscape, but the package 'color.sty' is not loaded}%
    \renewcommand\color[2][]{}%
  }%
  \providecommand\transparent[1]{%
    \errmessage{(Inkscape) Transparency is used (non-zero) for the text in Inkscape, but the package 'transparent.sty' is not loaded}%
    \renewcommand\transparent[1]{}%
  }%
  \providecommand\rotatebox[2]{#2}%
  \newcommand*\fsize{\dimexpr\f@size pt\relax}%
  \newcommand*\lineheight[1]{\fontsize{\fsize}{#1\fsize}\selectfont}%
  \ifx\svgwidth\undefined%
    \setlength{\unitlength}{360.58802472bp}%
    \ifx\svgscale\undefined%
      \relax%
    \else%
      \setlength{\unitlength}{\unitlength * \real{\svgscale}}%
    \fi%
  \else%
    \setlength{\unitlength}{\svgwidth}%
  \fi%
  \global\let\svgwidth\undefined%
  \global\let\svgscale\undefined%
  \makeatother%
  \begin{picture}(1,0.9584846)%
    \lineheight{1}%
    \setlength\tabcolsep{0pt}%
    \put(0,0){\includegraphics[width=\unitlength,page=1]{newton_simple.pdf}}%
    \put(0.70131223,0.02066118){\color[rgb]{0,0,0}\makebox(0,0)[lt]{\lineheight{1.25}\smash{\begin{tabular}[t]{l}$P_1$\end{tabular}}}}%
    \put(0.3269237,0.89423428){\color[rgb]{0,0,0}\makebox(0,0)[lt]{\lineheight{1.25}\smash{\begin{tabular}[t]{l}$P_2$\end{tabular}}}}%
    \put(0,0){\includegraphics[width=\unitlength,page=2]{newton_simple.pdf}}%
    \put(0.28532494,0.51152653){\color[rgb]{0,0,0}\makebox(0,0)[lt]{\lineheight{1.25}\smash{\begin{tabular}[t]{l}$\varphi(\alpha)$\end{tabular}}}}%
    \put(0,0){\includegraphics[width=\unitlength,page=3]{newton_simple.pdf}}%
    \put(-0.00586605,0.31185225){\color[rgb]{0,0,0}\makebox(0,0)[lt]{\lineheight{1.25}\smash{\begin{tabular}[t]{l}$P_3$\end{tabular}}}}%
    \put(0,0){\includegraphics[width=\unitlength,page=4]{newton_simple.pdf}}%
  \end{picture}%
\endgroup%

%% file: amoeba_simple.pdf_tex
\begingroup%
  \makeatletter%
  \providecommand\color[2][]{%
    \errmessage{(Inkscape) Color is used for the text in Inkscape, but the package 'color.sty' is not loaded}%
    \renewcommand\color[2][]{}%
  }%
  \providecommand\transparent[1]{%
    \errmessage{(Inkscape) Transparency is used (non-zero) for the text in Inkscape, but the package 'transparent.sty' is not loaded}%
    \renewcommand\transparent[1]{}%
  }%
  \providecommand\rotatebox[2]{#2}%
  \newcommand*\fsize{\dimexpr\f@size pt\relax}%
  \newcommand*\lineheight[1]{\fontsize{\fsize}{#1\fsize}\selectfont}%
  \ifx\svgwidth\undefined%
    \setlength{\unitlength}{365.65297205bp}%
    \ifx\svgscale\undefined%
      \relax%
    \else%
      \setlength{\unitlength}{\unitlength * \real{\svgscale}}%
    \fi%
  \else%
    \setlength{\unitlength}{\svgwidth}%
  \fi%
  \global\let\svgwidth\undefined%
  \global\let\svgscale\undefined%
  \makeatother%
  \begin{picture}(1,0.88898087)%
    \lineheight{1}%
    \setlength\tabcolsep{0pt}%
    \put(0,0){\includegraphics[width=\unitlength,page=1]{amoeba_simple.pdf}}%
    \put(0.85372767,0.35214386){\color[rgb]{0,0,0}\makebox(0,0)[lt]{\lineheight{1.25}\smash{\begin{tabular}[t]{l}$-\log|w|$\end{tabular}}}}%
    \put(0.4461823,0.86510046){\color[rgb]{0,0,0}\makebox(0,0)[lt]{\lineheight{1.25}\smash{\begin{tabular}[t]{l}$\log|z|$\end{tabular}}}}%
    \put(0,0){\includegraphics[width=\unitlength,page=2]{amoeba_simple.pdf}}%
  \end{picture}%
\endgroup%

%% file: fig_intersec_tt_edge_gammax.pdf_tex
\begingroup%
  \makeatletter%
  \providecommand\color[2][]{%
    \errmessage{(Inkscape) Color is used for the text in Inkscape, but the package 'color.sty' is not loaded}%
    \renewcommand\color[2][]{}%
  }%
  \providecommand\transparent[1]{%
    \errmessage{(Inkscape) Transparency is used (non-zero) for the text in Inkscape, but the package 'transparent.sty' is not loaded}%
    \renewcommand\transparent[1]{}%
  }%
  \providecommand\rotatebox[2]{#2}%
  \newcommand*\fsize{\dimexpr\f@size pt\relax}%
  \newcommand*\lineheight[1]{\fontsize{\fsize}{#1\fsize}\selectfont}%
  \ifx\svgwidth\undefined%
    \setlength{\unitlength}{582.2603951bp}%
    \ifx\svgscale\undefined%
      \relax%
    \else%
      \setlength{\unitlength}{\unitlength * \real{\svgscale}}%
    \fi%
  \else%
    \setlength{\unitlength}{\svgwidth}%
  \fi%
  \global\let\svgwidth\undefined%
  \global\let\svgscale\undefined%
  \makeatother%
  \begin{picture}(1,0.50221329)%
    \lineheight{1}%
    \setlength\tabcolsep{0pt}%
    \put(0,0){\includegraphics[width=\unitlength,page=1]{fig_intersec_tt_edge_gammax.pdf}}%
    \put(0.59337397,0.15384039){\color[rgb]{0,0,0}\makebox(0,0)[lt]{\lineheight{1.25}\smash{\begin{tabular}[t]{l}$\gamma_x$\end{tabular}}}}%
    \put(0.4130423,0.46298042){\color[rgb]{0,0,0}\makebox(0,0)[lt]{\lineheight{1.25}\smash{\begin{tabular}[t]{l}$\ws$\end{tabular}}}}%
    \put(0.07814057,0.10231704){\color[rgb]{0,0,0}\makebox(0,0)[lt]{\lineheight{1.25}\smash{\begin{tabular}[t]{l}$\bs$\end{tabular}}}}%
    \put(0,0){\includegraphics[width=\unitlength,page=2]{fig_intersec_tt_edge_gammax.pdf}}%
    \put(0.56761228,0.01215122){\color[rgb]{0,0,0}\makebox(0,0)[lt]{\lineheight{1.25}\smash{\begin{tabular}[t]{l}$T_\alpha$\end{tabular}}}}%
    \put(0.74794399,0.34705289){\color[rgb]{0,0,0}\makebox(0,0)[lt]{\lineheight{1.25}\smash{\begin{tabular}[t]{l}$T_\beta$\end{tabular}}}}%
  \end{picture}%
\endgroup%

%% file: shrink.pdf_tex
\begingroup%
  \makeatletter%
  \providecommand\color[2][]{%
    \errmessage{(Inkscape) Color is used for the text in Inkscape, but the package 'color.sty' is not loaded}%
    \renewcommand\color[2][]{}%
  }%
  \providecommand\transparent[1]{%
    \errmessage{(Inkscape) Transparency is used (non-zero) for the text in Inkscape, but the package 'transparent.sty' is not loaded}%
    \renewcommand\transparent[1]{}%
  }%
  \providecommand\rotatebox[2]{#2}%
  \newcommand*\fsize{\dimexpr\f@size pt\relax}%
  \newcommand*\lineheight[1]{\fontsize{\fsize}{#1\fsize}\selectfont}%
  \ifx\svgwidth\undefined%
    \setlength{\unitlength}{486.85714889bp}%
    \ifx\svgscale\undefined%
      \relax%
    \else%
      \setlength{\unitlength}{\unitlength * \real{\svgscale}}%
    \fi%
  \else%
    \setlength{\unitlength}{\svgwidth}%
  \fi%
  \global\let\svgwidth\undefined%
  \global\let\svgscale\undefined%
  \makeatother%
  \begin{picture}(1,0.25585539)%
    \lineheight{1}%
    \setlength\tabcolsep{0pt}%
    \put(0,0){\includegraphics[width=\unitlength,page=1]{shrink.pdf}}%
    \put(0.22735881,0.00905638){\color[rgb]{0,0,0}\makebox(0,0)[lt]{\lineheight{1.25}\smash{\begin{tabular}[t]{l}$\Gs$\end{tabular}}}}%
    \put(0.46973153,0.19826816){\color[rgb]{0,0,0}\makebox(0,0)[lt]{\lineheight{1.25}\smash{\begin{tabular}[t]{l}$\beta$\end{tabular}}}}%
    \put(-0.0051149,0.11421652){\color[rgb]{0,0,0}\makebox(0,0)[lt]{\lineheight{1.25}\smash{\begin{tabular}[t]{l}$\alpha$\end{tabular}}}}%
    \put(0.77437356,0.00946301){\color[rgb]{0,0,0}\makebox(0,0)[lt]{\lineheight{1.25}\smash{\begin{tabular}[t]{l}$\Gs'$\end{tabular}}}}%
    \put(0,0){\includegraphics[width=\unitlength,page=2]{shrink.pdf}}%
    \put(0.60801313,0.14314015){\color[rgb]{0,0,0}\makebox(0,0)[lt]{\lineheight{1.25}\smash{\begin{tabular}[t]{l}$\alpha$\end{tabular}}}}%
    \put(0.92441152,0.17523897){\color[rgb]{0,0,0}\makebox(0,0)[lt]{\lineheight{1.25}\smash{\begin{tabular}[t]{l}$\beta$\end{tabular}}}}%
    \put(0.19156573,0.23208323){\color[rgb]{0,0,0}\makebox(0,0)[lt]{\lineheight{1.25}\smash{\begin{tabular}[t]{l}$\fs_2$\end{tabular}}}}%
    \put(0.19200369,0.08457131){\color[rgb]{0,0,0}\makebox(0,0)[lt]{\lineheight{1.25}\smash{\begin{tabular}[t]{l}$\fs_1$\end{tabular}}}}%
    \put(0.75846689,0.2351642){\color[rgb]{0,0,0}\makebox(0,0)[lt]{\lineheight{1.25}\smash{\begin{tabular}[t]{l}$\fs'_2$\end{tabular}}}}%
    \put(0.7523052,0.07803364){\color[rgb]{0,0,0}\makebox(0,0)[lt]{\lineheight{1.25}\smash{\begin{tabular}[t]{l}$\fs'_1$\end{tabular}}}}%
    \put(0.06153551,0.15959486){\color[rgb]{0,0,0}\makebox(0,0)[lt]{\lineheight{1.25}\smash{\begin{tabular}[t]{l}$\bs_1$\end{tabular}}}}%
    \put(0.39572386,0.16115101){\color[rgb]{0,0,0}\makebox(0,0)[lt]{\lineheight{1.25}\smash{\begin{tabular}[t]{l}$\bs_2$\end{tabular}}}}%
  \end{picture}%
\endgroup%

%% file: trivalent.pdf_tex
\begingroup%
  \makeatletter%
  \providecommand\color[2][]{%
    \errmessage{(Inkscape) Color is used for the text in Inkscape, but the package 'color.sty' is not loaded}%
    \renewcommand\color[2][]{}%
  }%
  \providecommand\transparent[1]{%
    \errmessage{(Inkscape) Transparency is used (non-zero) for the text in Inkscape, but the package 'transparent.sty' is not loaded}%
    \renewcommand\transparent[1]{}%
  }%
  \providecommand\rotatebox[2]{#2}%
  \newcommand*\fsize{\dimexpr\f@size pt\relax}%
  \newcommand*\lineheight[1]{\fontsize{\fsize}{#1\fsize}\selectfont}%
  \ifx\svgwidth\undefined%
    \setlength{\unitlength}{517.92851782bp}%
    \ifx\svgscale\undefined%
      \relax%
    \else%
      \setlength{\unitlength}{\unitlength * \real{\svgscale}}%
    \fi%
  \else%
    \setlength{\unitlength}{\svgwidth}%
  \fi%
  \global\let\svgwidth\undefined%
  \global\let\svgscale\undefined%
  \makeatother%
  \begin{picture}(1,0.5491954)%
    \lineheight{1}%
    \setlength\tabcolsep{0pt}%
    \put(0,0){\includegraphics[width=\unitlength,page=1]{trivalent.pdf}}%
    \put(0.67177472,0.30082055){\color[rgb]{0,0,0}\makebox(0,0)[lt]{\lineheight{1.25}\smash{\begin{tabular}[t]{l}$\bs$\\\end{tabular}}}}%
    \put(0.16099672,0.31247014){\color[rgb]{0,0,0}\makebox(0,0)[lt]{\lineheight{1.25}\smash{\begin{tabular}[t]{l}$\ws$\end{tabular}}}}%
    \put(0,0){\includegraphics[width=\unitlength,page=2]{trivalent.pdf}}%
    \put(0.75741925,0.4614503){\color[rgb]{0,0,0}\makebox(0,0)[lt]{\lineheight{1.25}\smash{\begin{tabular}[t]{l}$\scriptstyle{\alpha}$\end{tabular}}}}%
    \put(-0.00408401,0.46017398){\color[rgb]{0,0,0}\makebox(0,0)[lt]{\lineheight{1.25}\smash{\begin{tabular}[t]{l}$\scriptstyle{\beta}$\end{tabular}}}}%
    \put(0.2833324,0.01277245){\color[rgb]{0,0,0}\makebox(0,0)[lt]{\lineheight{1.25}\smash{\begin{tabular}[t]{l}$\scriptstyle{\gamma}$\end{tabular}}}}%
    \put(0.41006578,0.14159718){\color[rgb]{0,0,0}\makebox(0,0)[lt]{\lineheight{1.25}\smash{\begin{tabular}[t]{l}$\fs_1$\end{tabular}}}}%
    \put(0.41006578,0.46017398){\color[rgb]{0,0,0}\makebox(0,0)[lt]{\lineheight{1.25}\smash{\begin{tabular}[t]{l}$\fs_2$\end{tabular}}}}%
    \put(0.00460443,0.28640482){\color[rgb]{0,0,0}\makebox(0,0)[lt]{\lineheight{1.25}\smash{\begin{tabular}[t]{l}$\fs_3$\end{tabular}}}}%
  \end{picture}%
\endgroup%